\documentclass[12pt]{article}

\usepackage{amsmath, amsthm, amssymb, stmaryrd}
\usepackage[all,cmtip]{xy}
\usepackage{fullpage}
\usepackage{titlesec}
\usepackage{mathrsfs}
\usepackage{amsbsy}
\usepackage{turnstile}
\usepackage{bbm}
\usepackage{yhmath}
\usepackage{tensor}
\usepackage{graphics}
\usepackage{enumitem}
\usepackage{calligra}
\usepackage{verbatim}
\usepackage{setspace}
\usepackage{hhline}
\usepackage{array}

\usepackage[pdftex,bookmarks=true]{hyperref}
\usepackage[svgnames]{xcolor}
\hypersetup{
    linkbordercolor = {PaleTurquoise},
    citebordercolor = {PaleGreen},
}

\usepackage{titletoc}

\titlecontents{section}
[0pt]                                               
{}
{\contentsmargin{0pt}                               
    \thecontentslabel\enspace
    }
{\contentsmargin{0pt}}                        
{\titlerule*[.5pc]{.}\contentspage}                 
[]

\makeatletter
\newcommand*{\@old@slash}{}\let\@old@slash\slash
\def\slash{\relax\ifmmode\delimiter"502F30E\mathopen{}\else\@old@slash\fi}
\makeatother

\titleformat{\section}{\normalsize\bfseries}{\thesection}{1em}{}
\titleformat{\subsection}{\normalsize\bfseries}{\thesubsection}{1em}{}

\numberwithin{equation}{subsection}

\theoremstyle{plain}

\newtheorem{prop}[subsection]{Proposition}
\newtheorem{lem}[subsection]{Lemma}
\newtheorem{cor}[subsection]{Corollary}
\newtheorem{thm}[subsection]{Theorem}

\newtheorem{obj}[subsection]{Objective}

\newtheorem{propsub}[subsubsection]{Proposition}

\newtheorem{thmsub}[subsubsection]{Theorem}

\theoremstyle{definition}

\newtheorem{defn}[subsection]{Definition}
\newtheorem{exa}[subsection]{Example}
\newtheorem{rem}[subsection]{Remark}
\newtheorem{para}[subsection]{}

\newtheorem{notn}[subsection]{Notation}

\newtheorem{exasub}[subsubsection]{Example}
\newtheorem{parasub}[subsubsection]{}

\newcommand*{\emptybox}{\leavevmode\hbox{}}

\DeclareMathOperator{\Hom}{Hom}
\DeclareMathOperator{\sfHom}{\mathsf{Hom}}

\newdir{ >}{{}*!/-10pt/@{>}}

\DeclareMathAlphabet{\mathpzc}{OT1}{pzc}{m}{it}
\DeclareMathAlphabet{\mathcalligra}{T1}{calligra}{m}{n}

\newcommand{\pref}[1]{\textnormal{(\ref{#1})}}

\newcommand{\wavy}{\ensuremath{\rightsquigarrow}}

\newcommand{\A}{\ensuremath{\mathscr{A}}}
\newcommand{\B}{\ensuremath{\mathscr{B}}}
\newcommand{\C}{\ensuremath{\mathscr{C}}}
\newcommand{\D}{\ensuremath{\mathscr{D}}}
\newcommand{\E}{\ensuremath{\mathscr{E}}}
\newcommand{\F}{\ensuremath{\mathscr{F}}}

\newcommand{\normalJ}{\ensuremath{\mathscr{J}}}
\newcommand{\J}{\ensuremath{{\kern -0.4ex \mathscr{J}}}}

\newcommand{\JF}{\ensuremath{{\kern -0.4ex \mathscr{J}_{\kern -0.05ex F}}}}

\newcommand{\K}{\ensuremath{\mathscr{K}}}

\newcommand{\R}{\ensuremath{\mathscr{R}}}
\newcommand{\sS}{\ensuremath{\mathscr{S}}}
\newcommand{\T}{\ensuremath{\mathscr{T}}}
\newcommand{\U}{\ensuremath{\mathscr{U}}}
\newcommand{\V}{\ensuremath{\mathscr{V}}}

\newcommand{\X}{\ensuremath{\mathscr{X}}}
\newcommand{\Y}{\ensuremath{\mathscr{Y}}}

\newcommand{\CAT}{\ensuremath{\operatorname{\textnormal{\text{CAT}}}}}

\newcommand{\VCAT}{\ensuremath{\V\textnormal{-\text{CAT}}}}

\newcommand{\ALGCAT}{\ensuremath{\operatorname{\textnormal{\text{ALGCAT}}}}}
\newcommand{\ALGCATJ}{\ensuremath{\ALGCAT_{\kern -0.7ex \normalJ}}}
\newcommand{\ADA}{\ensuremath{\operatorname{\textnormal{\text{ADA}}}}}
\newcommand{\LSADA}{\ensuremath{\operatorname{\textnormal{\text{LSADA}}}}}
\newcommand{\RSADA}{\ensuremath{\operatorname{\textnormal{\text{RSADA}}}}}
\newcommand{\SADA}{\ensuremath{\operatorname{\textnormal{\text{SADA}}}}}
\newcommand{\ADAJ}{\ensuremath{\ADA_{\kern -0.6ex \normalJ}}}
\newcommand{\LSADAJ}{\ensuremath{\LSADA_{\kern -0.6ex \normalJ}}}
\newcommand{\RSADAJ}{\ensuremath{\RSADA_{\kern -0.6ex \normalJ}}}
\newcommand{\SADAJ}{\ensuremath{\SADA_{\kern -0.6ex \normalJ}}}

\newcommand{\NN}{\ensuremath{\mathbb{N}}}
\newcommand{\PP}{\ensuremath{\mathbb{P}}}

\newcommand{\RR}{\ensuremath{\mathbb{R}}}

\newcommand{\TT}{\ensuremath{\mathbb{T}}}

\newcommand{\UU}{\ensuremath{\mathbb{U}}}

\newcommand{\ZZ}{\ensuremath{\mathbb{Z}}}

\newcommand{\ob}{\ensuremath{\operatorname{\textnormal{\textsf{ob}}}}}

\newcommand{\Lan}{\ensuremath{\operatorname{\textnormal{\textsf{Lan}}}}}

\newcommand{\End}{\ensuremath{\operatorname{\textnormal{End}}}}

\newcommand{\Th}{\ensuremath{\textnormal{Th}}}

\newcommand{\ThJ}{\ensuremath{\Th_{\kern -0.5ex \normalJ}}}
\newcommand{\SubThJ}{\ensuremath{\textnormal{SubTh}_{\kern -0.5ex \normalJ}}}

\newcommand{\otimesJ}[2]{\ensuremath{#1 \otimes_{\kern -1ex \normalJ} \kern-0.5ex #2}}
\newcommand{\totimes}{\mathbin{\tilde{\otimes}}}
\newcommand{\totimesJ}[2]{\ensuremath{#1 \totimes_{\kern -1ex \normalJ} \kern-0.5ex #2}}

\newcommand{\Vfp}{\ensuremath{\V_{\kern -0.5ex fp}}}
\newcommand{\TTperpJ}{\ensuremath{\TT^\perp_{\kern-.1ex\scriptscriptstyle\J}}}
\newcommand{\TTperpJwrt}[1]{\ensuremath{\TT^\perp_{\kern-.1ex\scriptscriptstyle\J#1}}}
\newcommand{\PPperpJ}{\ensuremath{\PP^\perp_{\kern-.1ex\scriptscriptstyle\J}}}
\newcommand{\PPperpJwrt}[1]{\ensuremath{\PP^\perp_{\kern-.1ex\scriptscriptstyle\J#1}}}

\newcommand{\y}{\ensuremath{\textnormal{\textsf{y}}}}

\newcommand{\inJ}[1]{#1 \in \normalJ\kern -1.2ex}

\newcommand{\VCATJ}{\VCAT_{\kern -0.7ex \normalJ}}
\newcommand{\PhiJ}{\Phi_{\kern -0.8ex \normalJ}}
\newcommand{\MndJ}{\Mnd_{\kern -0.8ex \normalJ}}

\newcommand{\Set}{\ensuremath{\operatorname{\textnormal{\text{Set}}}}}
\newcommand{\FinCard}{\ensuremath{\operatorname{\textnormal{\text{FinCard}}}}}

\newcommand{\SF}{\ensuremath{\operatorname{\textnormal{\text{SF}}}}}

\newcommand{\SLat}{\ensuremath{\operatorname{\textnormal{\text{SLat}}}}}

\newcommand{\Ab}{\ensuremath{\operatorname{\textnormal{\text{Ab}}}}}
\newcommand{\sAb}{\ensuremath{\operatorname{\textnormal{\textsl{Ab}}}}}

\newcommand{\Mod}[1]{\ensuremath{#1\textnormal{-\textsl{Mod}}}}
\newcommand{\pMod}[1]{\ensuremath{#1\textnormal{-\text{Mod}}}}

\newcommand{\Aff}[1]{\ensuremath{#1\textnormal{-\text{Aff}}}}
\newcommand{\sAff}[1]{\ensuremath{#1\textnormal{-\textsl{Aff}}}}
\newcommand{\Cvx}[1]{\ensuremath{#1\textnormal{-\textsl{Cvx}}}}

\newcommand{\Bimod}{\ensuremath{\textnormal{\text{Bimod}}}}
\newcommand{\Mat}{\ensuremath{\textnormal{\text{Mat}}}}

\newcommand{\Top}{\ensuremath{\operatorname{\textnormal{\text{Top}}}}}

\newcommand{\Conv}{\ensuremath{\operatorname{\textnormal{\text{Conv}}}}}

\newcommand{\Cah}{\ensuremath{\operatorname{\textnormal{\text{Cah}}}}}

\newcommand{\Mnd}{\ensuremath{\operatorname{\textnormal{\text{Mnd}}}}}

\newcommand{\Mon}{\ensuremath{\operatorname{\textnormal{\text{Mon}}}}}
\newcommand{\CMon}{\ensuremath{\operatorname{\textnormal{\text{CMon}}}}}
\newcommand{\sCMon}{\ensuremath{\operatorname{\textnormal{\textsl{CMon}}}}}

\newcommand{\Alg}[1]{\ensuremath{#1\kern -.5ex\operatorname{\textsl{-Alg}}}}
\newcommand{\pAlg}[1]{\ensuremath{#1\kern -.5ex\operatorname{\textnormal{-\textnormal{Alg}}}}}
\newcommand{\CAlg}[1]{\ensuremath{#1\kern -.5ex\operatorname{\textnormal{-\textlf{CAlg}}}}}
\newcommand{\CAlgPair}[2]{\ensuremath{(#1,#2)\kern -.5ex\operatorname{\textnormal{-\textsl{CPair}}}}}
\newcommand{\pCAlgPair}[2]{\ensuremath{(#1,#2)\kern -.5ex\operatorname{\textnormal{-\textnormal{CPair}}}}}
\newcommand{\CinftyRing}{\ensuremath{C^\infty\kern -.5ex\operatorname{\textnormal{-\text{Ring}}}}}

\newcommand{\CRProf}{\ensuremath{\operatorname{\textnormal{\text{CRProf}}}}}
\newcommand{\CRProfJ}{\ensuremath{\CRProf_{\kern -0.6ex \normalJ}}}
\newcommand{\Algs}{\ensuremath{\operatorname{\textnormal{\text{Alg}}}}}
\newcommand{\SatAlgs}{\ensuremath{\operatorname{\textnormal{\text{SatAlg}}}}}
\newcommand{\AlgsJ}{\Algs_{\kern -0.6ex \normalJ}}
\newcommand{\SatAlgsJ}{\SatAlgs_{\kern -0.6ex \normalJ}}

\newcommand{\AlgT}{\Algs_{\mathsf{th}}}
\newcommand{\AlgS}{\Algs_{\mathsf{sth}}}
\newcommand{\SatAlgS}{\SatAlgs_{\mathsf{sth}}}
\newcommand{\AlgO}{\Algs}
\newcommand{\SatAlgO}{\SatAlgs}
\newcommand{\sAlgs}{\ensuremath{\operatorname{\textnormal{\textsl{Alg}}}}}
\newcommand{\BAlg}{\ensuremath{\operatorname{\textnormal{\text{BAlg}}}}}
\newcommand{\BAlgJ}{\BAlg_{\kern -0.6ex \normalJ}}
\newcommand{\BAlgO}{\BAlg}
\newcommand{\RComBAlg}{\ensuremath{\operatorname{\textnormal{\text{RComBAlg}}}}}
\newcommand{\LComBAlg}{\ensuremath{\operatorname{\textnormal{\text{LComBAlg}}}}}
\newcommand{\ComBAlg}{\ensuremath{\operatorname{\textnormal{\text{ComBAlg}}}}}
\newcommand{\RComBAlgO}{\RComBAlg}
\newcommand{\LComBAlgO}{\LComBAlg}
\newcommand{\ComBAlgO}{\ComBAlg}
\newcommand{\CAlgPairs}{\ensuremath{\operatorname{\textnormal{\text{CAlgPair}}}}}
\newcommand{\CAlgPairO}{\CAlgPairs}

\newcommand{\Dualn}{\ensuremath{\operatorname{\mathsf{Dualn}}}}
\newcommand{\Opfib}{\ensuremath{\operatorname{\mathsf{Opfib}}}}

\newcommand{\Fib}{\ensuremath{\operatorname{\mathsf{Fib}}}}

\newcommand{\Adjn}[6]{\xymatrix {#1 \ar@/_0.5pc/[rr]_{#2}^(0.4){#4}^(0.6){#5}^{\top} & & #6 \ar@/_0.5pc/[ll]_{#3}}}

\newcommand{\RAdjn}[6]{\xymatrix {#1 \ar@/^0.45pc/[rr]^{#2}^(0.4){#4}^(0.6){#5}_{\top} & & #6 \ar@/^0.45pc/[ll]^{#3}}}

\newcommand{\Equiv}[6]{\xymatrix {#1 \ar@/_0.5pc/[rr]_{#2}^(0.4){#4}^(0.6){#5}^{\sim} & & #6 \ar@/_0.5pc/[ll]_{#3}}}

\newcommand{\EquivAlt}[6]{\xymatrix {#1 \ar@{<-}@/_0.5pc/[rr]_{#3}^(0.4){#4}^(0.6){#5}^{\sim} & & #6 \ar@{<-}@/_0.5pc/[ll]_{#2}}}
\newcommand{\AlgDual}[4]{\xymatrix {#1:#2 \ar@/_0.2pc/[r] & #3:#4 \ar@/_0.2pc/[l] }}

\newcommand{\sslashps}{\sslash_{\kern -0.5ex\textnormal{\tiny{p}}}}

\newcommand{\op}{\ensuremath{\textnormal{\hspace{0.3ex}\tiny op}}}
\newcommand{\co}{\ensuremath{\textnormal{\hspace{0.3ex}\tiny co}}}

\newcommand{\aff}{\ensuremath{\textnormal{\tiny aff}}}

\newcommand{\pushoutcorner}{\ar@{}[dr]|(.3)\ulcorner}
\newcommand{\pullbackcorner}{\ar@{}[dr]|(.3)\lrcorner}

\newcommand{\ladjto}[2]{\nsststile{#2}{#1}}

\newcommand{\blanktwo}{\mathop{?}}

\newcommand{\sx}{{\mathsf{sc}}}
\newcommand{\si}{{\scriptscriptstyle\cong}}
\newcommand{\seq}{{\scriptscriptstyle\simeq}}

\newcommand{\bit}[1]{\textbf{\textit{#1}}}

\begin{document}

\author{\normalsize  Rory B. B. Lucyshyn-Wright\thanks{We acknowledge the support of the Natural Sciences and Engineering Research Council of Canada (NSERC), [funding reference numbers RGPIN-2019-05274, RGPAS-2019-00087, DGECR-2019-00273].  Cette recherche a \'et\'e financ\'ee par le Conseil de recherches en sciences naturelles et en g\'enie du Canada (CRSNG), [num\'eros de r\'ef\'erence RGPIN-2019-05274, RGPAS-2019-00087, DGECR-2019-00273].}\let\thefootnote\relax\footnote{Keywords: Dual adjunction; enriched category; Lawvere theory; monad; algebraic theory; system of arities; algebraic category; commutant; duality; reflexivity; monoidal category; closed category}\footnote{2020 Mathematics Subject Classification: 18A40, 18C05, 18C10, 18C15, 18C20, 18C40, 18D15, 18D20, 18M05, 18N10}
\\
\small Brandon University, Brandon, Manitoba, Canada}

\title{\large \textbf{Dual adjunctions between enriched algebraic categories}}

\date{}

\maketitle

\abstract{Working in the setting of enriched algebraic theories for a system of arities, we study several aspects of dual adjunctions between enriched algebraic categories, which we call \textit{algebraic dual adjunctions}. Firstly, we generalize Freyd's theorem on contravariant algebra-valued right-adjoint functors to this setting. Secondly, we establish a biequivalence between a 2-category of algebraic dual adjunctions and a locally discrete 2-category of \textit{bifold algebras}, in a sense defined in prior work of the author. Thirdly, we define special classes of algebraic dual adjunctions that we call (\textit{left-} and \textit{right-})\textit{stable}, in which certain free objects are reflexive, and we establish biequivalences between these and special classes of bifold algebras defined in terms of \textit{commutants} in prior work of the author. Fourthly, we show that every algebra for an enriched algebraic theory canonically induces left- and right-stable algebraic dual adjunctions, and we establish a biequivalence between such algebras and left- (or right-)stable algebraic dual adjunctions, and also between \textit{saturated} algebras and stable algebraic dual adjunctions. We also discuss examples of algebraic dual adjunctions, including dualization of internal modules, Pontryagin and Binz-Butzmann duality, dualization of internal affine spaces and convex spaces, dualization of semilattices, dualization of complete sup-lattices, and dualization of abelian groups with reference to a theorem of Ehrenfeucht and \L o\'s.}

\section{Introduction} \label{sec:intro}

\emptybox

The unifying role of \textit{dualities} or \textit{dual equivalences} $\A^\op \simeq \B$ in algebra, geometry, and logic is evident in examples like the Gelfand, Stone, and Pontryagin dualities. A clarifying insight of category theory has been that dualities often arise from the more basic notion of \textit{dual adjunction}, i.e.~an adjunction of the form $\nabla \dashv \Delta:\A^\op \to \B$. Indeed, as noted in \cite{LamRa:Stone-Gelfand,PorTho:ConcreteDualities}, every such adjunction restricts to a dual equivalence between the full subcategories of $\A$ and $\B$ spanned by the \textit{reflexive} objects; an object $A$ of $\A$ is reflexive (with respect to $\nabla \dashv \Delta$) if the associated morphism $A \to \nabla\Delta A$ is invertible.

Category theory has also been used to give explanations for the familiar feature of many dualities and dual adjunctions $\nabla \dashv \Delta:\A^\op \to \B$ that the functors $\nabla$ and $\Delta$ are given by \textit{dualizing} or \textit{`hom'ing} into an object $D$ (the \textit{dualizer}) that lives both in $\A$ and in $\B$, in some sense; see, for example, \cite{Freyd:AlgebraValued,Isbell:GenFuncSemI,LamRa:Stone-Gelfand,FoltzLairKelly,Joh:StSp,PorTho:ConcreteDualities}. One of the earliest results on this topic was a 1966 theorem of Freyd \cite{Freyd:AlgebraValued}, which begins with the old observation that if $\TT$ is a finitary algebraic theory (in the logical sense) and $E \in \ob\A$ underlies a $\TT$-algebra internal to a category $\A$, then the hom-functor $\A(-,E):\A^\op \to \Set$ lifts to a functor $\llbracket -,E\rrbracket:\A^\op \to \pAlg{\TT}$ valued in the category of $\TT$-algebras in $\Set$. Freyd's theorem \cite[Theorem 2]{Freyd:AlgebraValued} on ``algebra-valued functors'', in its contravariant form, is the statement that if $\A$ is a complete category, then a functor $\Delta:\A^\op \to \pAlg{\TT}$ has a left adjoint $\nabla$ if and only if $\Delta \cong \llbracket -,E\rrbracket$ for some $\TT$-algebra $E$ in $\A$. As a corollary, Freyd showed that for a given pair of finitary algebraic theories $\TT$ and $\UU$, a right adjoint functor $\Delta:\pAlg{\TT}^\op \to \pAlg{\UU}$ is equivalently given by any one of the following: (1) a $\TT$-algebra in $\pAlg{\UU}$, (2) a $\UU$-algebra in $\pAlg{\TT}$, (3) a set $D$ equipped with the structure of both a $\TT$-algebra and a $\UU$-algebra, such that these two algebra structures \textit{commute}, meaning that each $\TT$-operation $D^n \to D$ is a $\UU$-homomorphism, equivalently, each $\UU$-operation is a $\TT$-homomorphism. In the early 1970s, Pultr \cite{Pultr:RightAdjointsRelSys} and Isbell \cite{Isbell:GenFuncSemI} generalized Freyd's theorem to more general kinds of theories. Porst and Tholen \cite{PorTho:ConcreteDualities} later provided an approach to the relationship between dual adjunctions and dualization that does not involve theories and is instead formulated in terms of initial sources and initial lifts in concrete categories over $\Set$.

In this paper, we pursue aspects of these themes in the setting of category theory enriched in a symmetric monoidal closed category $\V$ \cite{Ke:Ba}. In contemporary category theory, there are several possible notions of theory to which one might seek to extend the above results of Freyd, Pultr, and Isbell, in view of the extensive development of enriched and 2-dimensional categories since the early 1970s. We begin with the basic case of \textit{enriched algebraic theories}, which have received thorough development, starting with \cite{Dub:StrSem,BoDay,Pow:EnrLaw} for example, and we seek to not only develop analogues of Freyd's results in this case but also to study further specific aspects of dual adjunctions between enriched algebraic categories, which we call \textit{algebraic dual adjunctions}\footnote{At CT2017, this name was suggested by Walter Tholen to the author, who at the time had called them \textit{algebraic dualities}.}.

\medskip

In more detail, the main contributions of this paper are as follows:

\medskip

\noindent(1). Our first contribution is to develop analogues of the above results of Freyd in the setting of \textit{enriched algebraic theories for systems of arities} $\J \hookrightarrow \V$ \cite{Lu:EnrAlgTh}, or \mbox{\bit{$\J$-theories}}, which generalize not only the usual Lawvere theories \cite{Law:PhD} but also various kinds of enriched Lawvere theories and monads appearing in \cite{Dub:StrSem,BoDay,KeLa:SF,Pow:EnrLaw,LaRo,RosTen:TowardsEnrUnivAlg,Tendas:MoreOnSoundness}. Given a $\J$-theory $\T$, a \bit{$\T$-algebra} (in $\V$) is a suitable $\V$-functor $A:\T \to \V$, while an \bit{algebra} is a pair $(\T,A)$ consisting of a $\J$-theory $\T$ and a $\T$-algebra $A$. The $\V$-categories of $\T$-algebras $\Alg{\T}$ can be characterized up to equivalence in the pseudo-slice 2-category $\VCAT \sslash \V$ as \bit{$\J$-algebraic $\V$-categories over $\V$} \cite{Lu:EnrAlgTh,LuPa:Struct}, which are $\V$-categories $\A$ equipped with a suitable $\V$-functor $\A \to \V$. We generalize Freyd's theorem on algebra-valued functors, showing that if $\C$ is a $\V$-category with certain weighted limits, then right adjoint $\V$-functors $\C^\op \to \Alg{\T}$ are equivalently given by $\T$-algebras $E$ in $\C$, which are suitable $\V$-functors $E:\T \to \C$. We also develop a Yoneda principle for the resulting notion of \textit{representable} algebra-valued $\V$-functor. To generalize Freyd's corollary on commuting algebraic structures, we employ the theory of commutation for $\J$-theories developed in \cite{Lu:Cmt,Lu:BAlgCmt}, wherein commuting pairs of algebraic structures on an object of $\V$ are equivalently given by \bit{bifold algebras} \cite{Lu:BAlgCmt}, which are triples $(\T,\U,D)$ consisting of $\J$-theories $\T,\U$ and a suitable $\V$-functor $D:\T \otimes \U \to \V$, called a \textit{$(\T,\U)$-algebra}. Each bifold algebra $(\T,\U,D)$ has two underlying algebras $(\T,D_\ell)$ and $(\U,D_r)$ that we call the \bit{left and right faces} of $(\T,\U,D)$. We show that right adjoint $\V$-functors $\Delta:\Alg{\T}^\op \to \Alg{\U}$ are equivalently given by $(\T,\U)$-algebras $D$.

\medskip

\noindent(2). Our second contribution is to study the functoriality and basic 2-dimensional aspects of constructions and results in (1). In particular, we establish a biequivalence between a 2-category of \bit{$\J$-algebraic dual adjunctions} $\nabla \dashv \Delta:\A^\op \to \B$ (between various $\J$-algebraic $\V$-categories $\A$ and $\B$ over $\V$) and a locally discrete 2-category of bifold algebras $(\T,\U,D)$ (over various $\J$-theories $\T$ and $\U$). As part of this, we show that the assignment to each bifold algebra $(\T,\U,D)$ its associated $\J$-algebraic dual adjunction
$$\RAdjn{\Alg{\T}^\op}{\Hom_\T(-,D)}{\Hom_\U(-,D)}{}{}{\Alg{\U}}$$
is functorial in $(\T,\U,D)$. More fundamentally, we first use methods of fibred categories to show that $\Hom_\T(A,D)$ is functorial $A$ and $D$ in such a way that $\T$ and $\U$ vary as well. Under the correspondence between bifold algebras $D$ and $\J$-algebraic dual adjunctions $\nabla \dashv \Delta:\A^\op \to \B$, the left and right faces of $D$ correspond to special objects of $\A$ and of $\B$ that we call the \bit{left dualizer} and \bit{right dualizer}, respectively.

\medskip

\noindent(3). Our third contribution is to show that certain special classes of bifold algebras defined in \cite{Lu:BAlgCmt} using \bit{commutants} \cite{Lu:Cmt,Lu:BAlgCmt} correspond to special classes of $\J$-algebraic dual adjunctions in which certain free objects are required to be reflexive. Explicitly, we say that a $\J$-algebraic dual adjunction $\nabla \dashv \Delta:\A^\op \to \B$ is \bit{left-stable} (resp.~\bit{right-stable}) if every $\J$-generated free object of $\A$ (resp.~of $\B$) is reflexive (where an object of $\A$ is \textit{$\J$-generated free} if it is free on an object in the full subcategory $\J \hookrightarrow \V$). A $\J$-algebraic dual adjunction is \bit{stable} if it is both left- and right-stable. We show that stable $\J$-algebraic dual adjunctions correspond (under the above biequivalence) to a special class of bifold algebras introduced in \cite{Lu:BAlgCmt}, called \bit{commutant bifold algebras}, whose left and right faces are mutually interdetermined. In more detail, a \bit{left-commutant bifold algebra} \cite{Lu:BAlgCmt} is a bifold algebra $(\T,\U,D)$ whose left face $(\T,D_\ell)$ is the \textit{commutant} of its right face $(\U,D_r)$, in the sense defined in \cite{Lu:BAlgCmt}; the notion of \bit{right-commutant bifold algebra} is defined analogously. Accordingly, a \textit{commutant bifold algebra} is one whose left and right faces are commutants of one another. This essentially means that each of the two faces of $(\T,\U,D)$ consists of those operations that commute with the operations supplied by the other face; see \ref{para:cmt}, \ref{para:cmtnt_bif_algs}. We show that the biequivalence between $\J$-algebraic dual adjunctions and bifold algebras restricts to a biequivalence between stable (resp.~left-stable, right-stable) $\J$-algebraic dual adjunctions and commutant (resp.~left-commutant, right-commutant) bifold algebras. 

\medskip

\noindent(4). Our fourth contribution is to show that every algebra canonically induces right- and left-stable $\J$-algebraic dual adjunctions, and that right-stable (resp.~left-stable) $\J$-algebraic dual adjunctions can be reconstructed up to equivalence from their left (resp.~right) dualizers and thus are equivalently given by algebras. In more detail, we show that every algebra $(\T,A)$ determines a right-stable $\J$-algebraic dual adjunction that we write as
\begin{equation}\label{eq:dualn-adjn-induced-from-left-intro}\RAdjn{\Alg{\T}^\op}{\Hom_\T(-,A)}{\Hom_{\T^\perp_A}(-,A^\perp)}{}{}{\Alg{\T^\perp_A}}\end{equation}
and call the \bit{dualization adjunction induced from the left by $A$}, where the $\J$-theory $\T^\perp_A$ is the \bit{commutant of $\T$ with respect to $A$} \cite{Lu:Cmt}, and the algebra $(\T^\perp_A,A^\perp)$ is the commutant of $(\T,A)$ \cite{Lu:BAlgCmt}. Using this, we establish a biequivalence between a 2-category of right-stable $\J$-algebraic dual adjunctions and a locally discrete 2-category of algebras over various $\J$-theories, and a similar biequivalence involving instead left-stable $\J$-algebraic dual adjunctions. We show that the $\J$-algebraic dual adjunction \eqref{eq:dualn-adjn-induced-from-left-intro} is stable if and only if $(\T,A)$ is \textit{saturated} in the sense of \cite{Lu:FDistn,Lu:BAlgCmt}. Using this, we obtain a biequivalence between a 2-category of stable $\J$-algebraic dual adjunctions and a locally discrete 2-category of saturated algebras.

\medskip

We now discuss in more detail the setting of this paper, namely  $\J$-theories for systems of arities $\J \hookrightarrow \V$ \cite{Lu:EnrAlgTh}. These generalize all of the following notions of enriched Lawvere theory: (1) The \textit{$\V$-theories} of Dubuc \cite{Dub}, which correspond to arbitrary $\V$-monads on $\V$; (2) the finitary enriched theories of Borceux and Day \cite{BoDay}, which correspond to \textit{strongly finitary $\V$-monads} on $\V$ \cite{KeLa:SF} when $\V$ is cartesian closed; (3) the $\alpha$-ary \textit{enriched Lawvere theories} of Power \cite{Pow:EnrLaw}, which correspond to \textit{$\alpha$-accessible $\V$-monads} on $\V$ when $\V$ is locally $\alpha$-presentable; (4) the \textit{Lawvere $\Phi$-theories in $\V$} in the sense of Lack and Rosick\'y \cite{LaRo}, and the closely related \textit{$\Phi\mathcal{I}$-theories} considered by Tendas \cite{Tendas:MoreOnSoundness} for weakly sound classes of weights $\Phi$, both of which correspond to \textit{$\Phi$-accessible $\V$-monads} on $\V$. One can contrast \textit{systems} of arities $\J \hookrightarrow \V$ with the \textit{subcategories of arities} $\J \hookrightarrow \C$ of \cite{LuPa:Pres,LuPa:DiagrPres,LuPa:Struct}, which are simply full and dense sub-$\V$-categories. The latter are more general than the former, in two respects: Firstly, the latter involve an additional specified $\V$-category $\C$ that we call the \textit{base}, over which the categories of algebras are monadic, while the former necessarily pertain to the base $\C = \V$; secondly, the former require $\J \hookrightarrow \V$ to be closed under the monoidal product and contain the unit object. Thus, in this paper we work in a setting of \textit{$\V$-based $\V$-enriched} algebra, rather than the more general \textit{$\C$-based $\V$-enriched} algebra \cite{LuPa:Pres,LuPa:DiagrPres,LuPa:Struct}, special cases of which one finds also in \cite{KellyPower,NishizawaPower,BouGar:MndTh}. This restriction of generality is appropriate for the topic of this paper because (1) we can lift a hom $\V$-functor $\A(-,E):\A^\op \to \V$ to $\Alg{\T}$ in analogy with Freyd's case of $\V = \Set$ only if $\Alg{\T}$ is equipped with a $\V$-functor to $\V$, and (2) a notion of \textit{commutation} of operations in $\J$-theories has been developed only in the $\V$-based case \cite{Lu:Cmt,Lu:BAlgCmt}, and it uses the fact that systems of arities $\J \hookrightarrow \V$ are closed under the monoidal product in $\V$. Nevertheless, we formulate our results in a setting slightly more general than that provided by the \textit{eleutheric systems of arities} of \cite{Lu:EnrAlgTh}, as we work here with systems of arities that are only assumed to be \textit{amenable} in the sense of \cite{LuPa:Struct}; nonetheless, all the classes of examples (1)--(4) above are described by eleutheric systems of arities.

\medskip

We now survey the organization of the paper. In Section \ref{sec:supporting-theory}, we review background material on lax comma 2-categories (\ref{sec:lax-comma}), enriched algebraic theories (\ref{sec:enr-alg-th}), and bifold algebras (\ref{sec:balg}). In \ref{sec:enr-alg-th} we also establish some further supporting results (\ref{thm:bieq_th_algcat}, \ref{thm:loc-sound-eleutheric}), and we specify the blanket assumptions in place throughout the paper (\ref{para:setting2}). In Section \ref{sec:alg-val-func}, we establish (i) our enriched analogue of Freyd's theorem on contravariant algebra-valued functors, (ii) a Yoneda principle and Yoneda embedding for algebra-valued $\V$-functors, and (iii) an equivalence between $\T$-algebras in a suitable $\V$-category $\C$ and right adjoint $\V$-functors $\C^\op \to \Alg{\T}$. In Section \ref{sec:dualn_adj}, we define the dualization adjunction induced by a bifold algebra, and we provide an equivalence between bifold $(\T,\U)$-algebras, for a fixed pair of $\J$-theories $\T$ and $\U$, and right adjoint $\V$-functors from $\Alg{\T}^\op$ to $\Alg{\U}$. In Section \ref{sec:alg-dual-adj}, we discuss $\J$-algebraic dual adjunctions and the 2-category of these, and we prove that each of them is equivalent to the dualization adjunction induced by a bifold algebra. In Section \ref{sec:global-func}, we use methods of fibred categories to establish the global functoriality of the algebra-valued hom-functor. In Section \ref{sec:bieq-bif-alg-alg-dual-adjn}, we establish a biequivalence between bifold algebras and $\J$-algebraic dual adjunctions. In Section \ref{exa:bim}, we discuss a basic example that involves modules and bimodules and is obtained by specializing to the  subcategory of arities $\{I\} \hookrightarrow \V$ spanned by the unit object, noting a connection with Morita's classic paper \cite{Morita}. In Section \ref{sec:background-II}, we review background material on commutants and commutation for $\J$-theories. In Section \ref{sec:alg-dual-adj-comm-alg-pairs}, we deduce that there is a biequivalence between commuting algebra pairs and $\J$-algebraic dual adjunctions. In Section \ref{sec:stabl-alg-dual-adj}, we define left-stable, right-stable, and stable $\J$-algebraic dual adjunctions, and we establish biequivalences between these and left-commutant, right-commutant, and commutant bifold algebras, respectively. In Section \ref{sec:stable-ada-alg}, we discuss the dualization adjunction induced from the left by an algebra, and we establish biequivalences between algebras and right-stable (or left-stable) $\J$-algebraic dual adjunctions and between saturated algebras and stable $\J$-algebraic dual adjunctions. In Section \ref{sec:exa_ada}, we discuss several examples of $\J$-algebraic dual adjunctions, including the basic example mentioned above (\ref{exa:dln_mod}), dualization of modules for internal rings and rigs in a cartesian closed category such as the category of convergence spaces or a model of synthetic differential geometry (\ref{exa:dualn_mod}), Pontryagin duality and Binz-Butzmann duality (\ref{exa:binz-butmzmann-duality}), dualization of internal $R$-affine spaces and pointed $R$-modules (\ref{exa:r_aff}), dualization of internal $R$-convex spaces and pointed $R_+$-modules (\ref{exa:cvx}), dualization of several kinds of semilattices (\ref{exa:dualn_semilattices}), the self-duality of complete sup-lattices (\ref{exa:sup-lat}), the restricted Pontryagin duality of compact and discrete topological groups (\ref{exa:cpct-disc}), and lastly the example of dualization of abelian groups, where we observe that a theorem of Ehrenfeucht and \L o\'s \cite{EhrLos} entails that the dualization adjunction for $\ZZ$ is stable if and only if there are no measurable cardinals (\ref{exa:los}).

\medskip

We now comment on the genesis of this paper. Some of the central ideas and results of this paper were presented in a talk at CT2017 \cite{Lu:CT2017}, while the main results in the form of biequivalences were announced in an online talk at CT2021 \cite{Lu:CT2021}. Supporting material on bifold algebras was published in 2022 \cite{Lu:BAlgCmt}. The treatment of algebraic dual adjunctions in this paper incorporates methodological refinements devised since the 2021 talk, including the Yoneda principle for algebra-valued $\V$-functors (\ref{thm:yon-alg-val}, \ref{thm:yon-emb-alg-val}) and a streamlined approach to the supporting biequivalence between $\J$-theories and $\J$-algebraic $\V$-categories over $\V$, on the basis of recent joint work with Jason Parker \cite{LuPa:Struct} that has also enabled a reduction of the axiomatic assumptions in this paper from eleuthericity to amenability (\ref{para:setting2}, \ref{thm:eleutheric-implies-amenable}).

\tableofcontents

\section{Supporting theory I}\label{sec:supporting-theory}

\subsection{Lax comma categories}\label{sec:lax-comma}

We now recall background material on lax comma categories and the Grothendieck construction.

\begin{parasub}\label{para:lax_comma}
Given 2-functors $F:\A \rightarrow \C$ and $G:\B \rightarrow \C$, we denote by $(F \downarrow G)_\textnormal{lax}$ the \bit{lax comma 2-category} (called the \textit{2-comma category} in \cite[I,2.5]{Gray}).  The objects of $(F \downarrow G)_\textnormal{lax}$ are triples $(A,B,c)$ consisting of an object $A$ of $\A$, an object $B$ of $\B$, and a 1-cell $c:FA \rightarrow GB$ in $\C$.  A 1-cell $(a,b,\gamma):(A,B,c) \rightarrow (A',B',c')$ in $(F \downarrow G)_\textnormal{lax}$ consists of a 1-cell $a:A \rightarrow A'$ in $\A$, a 1-cell $b:B \rightarrow B'$ in $\B$, and a 2-cell $\gamma:c' \circ Fa \Rightarrow Gb \circ c$ in $\C$.  A 2-cell $(\alpha,\beta):(a,b,\gamma) \Rightarrow (a',b',\gamma'):(A,B,c) \rightarrow (A',B',c')$ in $(F \downarrow G)_\textnormal{lax}$ consists of a 2-cell $\alpha:a \Rightarrow a'$ in $\A$ and a 2-cell $\beta:b \Rightarrow b'$ in $\B$ such that $\gamma' \cdot (c' \circ F\alpha) = (G\beta \circ c) \cdot \gamma$, where $\circ$ denotes whiskering and $\cdot$ denotes vertical composition.

We write $(F \downarrow G)_\textnormal{colax}$ to denote the \bit{colax comma 2-category}, which is defined as $((F^\co \downarrow G^\co)_{\textnormal{lax}})^\co$, where $F^\co:\A^\co \rightarrow \C^\co$ and $G^\co:\B^\co \rightarrow \C^\co$ are the 2-functors obtained by reversing the 2-cells in each of the 2-categories involved.  Hence, $(F \downarrow G)_\textnormal{lax}$ and $(F \downarrow G)_\textnormal{colax}$ have the same objects, and their 1-cells $(a,b,\gamma)$ differ only in that the direction of $\gamma$ is reversed, while the 2-cells in both these 2-categories have same form (but with the required equation modified in the evident way).

The \bit{pseudo comma 2-category} $(F \downarrow G)_\textnormal{ps}$ is the locally full sub-2-category of $(F \downarrow G)_\textnormal{lax}$ with the same objects, but with just those 1-cells $(a,b,\gamma)$ in which $\gamma$ is invertible.

If a 1-cell $(a,b,\gamma):(A,B,c) \rightarrow (A',B',c')$ in $(F \downarrow G)_\textnormal{lax}$ has a right adjoint $(a',b',\gamma'):(A',B',c') \to (A,B,c)$, then $\gamma:c' \circ Fa \Rightarrow Gb \circ c$ is invertible; indeed, $Fa \dashv Fa'$ and $Gb \dashv Gb'$ in $\C$, and the \textit{mate} $\delta:Gb \circ c \Rightarrow c' \circ Fa$ of $\gamma':c \circ Fa' \Rightarrow Gb' \circ c'$ (in the sense of \cite[\S 2.2, Prop.~2.1]{KeStr:RevEl2Cats}) is an inverse of $\gamma$. Dually, if a 1-cell $(a,b,\gamma)$ in $(F \downarrow G)_\textnormal{colax}$ has a left adjoint, then $\gamma$ is invertible. In particular, if a 1-cell $(a,b,\gamma)$ in $(F \downarrow G)_\textnormal{lax}$ or in $(F \downarrow G)_\textnormal{colax}$ is an equivalence, then $\gamma$ is invertible.
\end{parasub}

\begin{parasub}\label{para:pseudo_slice}
Given an object $B$ of a 2-category $\A$, we write $\A \sslash B$ to denote the \bit{pseudo-slice 2-category} over $B$ in $\A$, defined as the pseudo comma 2-category $(1_\A \downarrow [B])_\textnormal{ps}$ determined by the identity 2-functor $1_\A:\A \rightarrow \A$ and the  2-functor $[B]:1 \rightarrow \A$ determined by $B$, where $1$ is the terminal 2-category.  We may write the objects of $\A \sslash B$ as pairs $(A,c)$ consisting of an object $A$ of $\A$ and a 1-cell $c:A \rightarrow B$, while a 1-cell $(a,\gamma):(A,c) \rightarrow (A',c')$ consists of a 1-cell $a:A \rightarrow A'$ and an invertible 2-cell $\gamma:c'a \Rightarrow c$ in $\A$.  A 2-cell $\alpha:(a,\gamma) \Rightarrow (a',\gamma'):(A,c) \rightarrow (A',c')$ in $\A \sslash B$ is a 2-cell $\alpha:a \Rightarrow a'$ in $\A$ such that $\gamma' \cdot (c'\alpha) = \gamma$. 
\end{parasub}

\begin{parasub}
We regard categories equivalently as a locally discrete 2-categories.  If $\B$ is a category and $\C$ a 2-category, then by a \textit{functor} $\Phi:\B \to \C$ we mean a functor $\Phi:\B \to \C_0$, or equivalently a 2-functor $\Phi:\B \to \C$. Given categories $\A$ and $\B$ and functors $F:\A \to \C$ and $G:\B \to \C$ valued in a 2-category $\C$, the lax and colax comma 2-categories $(F \downarrow G)_\textnormal{lax}$ and $(F \downarrow G)_\textnormal{colax}$ are locally discrete, so we call them simply the \bit{lax comma category} and the \bit{colax comma category}, respectively.
\end{parasub}

It is observed in \cite{Ke:ClDo} that the Grothendieck construction for split op-fibrations is an example of a colax comma category, as we now recall:

\begin{parasub}\label{para:gro_constr_opfibr}
Given a category $\B$ and a functor $\Phi:\B \rightarrow \CAT$, we write $\Opfib(\B,\Phi)$ to denote the colax comma category $([1]\downarrow\Phi)_\textnormal{colax}$, where $1$ is the terminal category and $[1]:1 \rightarrow \CAT$ is the functor determined by $1$. Hence, the objects of $\Opfib(\B,\Phi)$ are pairs $(B,F)$ consisting of an object $B$ of $\B$ and an object $F$ of $\Phi B$.  A morphism $(b,f):(B,F) \rightarrow (B',F')$ in $\Opfib(\B,\Phi)$ is a pair consisting of a morphism $b:B \rightarrow B'$ and a morphism $f:(\Phi b)F \rightarrow F'$ in $\Phi B'$.  The evident projection functor $\Opfib(\B,\Phi) \rightarrow \B$ is then an op-fibration.  We call $\Opfib(\B,\Phi)$ the \bit{op-fibred category over $\B$ determined by $\Phi$}.
\end{parasub}

\begin{parasub}\label{para:fibr_op}
Given a category $\B$ and a functor $\Phi:\B \rightarrow \CAT$ (i.e.~an ordinary functor $\Phi:\B \to \CAT_0$), let us write $\Phi^*:\B \to \CAT$ to denote the composite functor 
$$\B \xrightarrow{\Phi} \CAT_0 \xrightarrow{(-)^\op} \CAT_0\;.$$
We call $\Opfib(\B,\Phi^*)$ the \bit{fibrewise opposite} of $\Opfib(\B,\Phi)$.  A morphism $(b,f):(B,F) \rightarrow (B',F')$ in $\Opfib(\B,\Phi^*)$ has $b:B \rightarrow B'$ in $\B$ and $f:F' \rightarrow (\Phi b)F$ in $\Phi B'$.  If we are instead given a functor $\Psi:\B^\op \rightarrow \CAT$, then
$$\Fib(\B,\Psi) = \Opfib(\B^\op,\Psi^*)^\op$$
is the fibred category over $\B$ determined by $\Psi$ via the classic Grothendieck construction \cite{Gr:Fibred}.
\end{parasub}

\subsection{Enriched algebraic theories}\label{sec:enr-alg-th}

In this section, we recall background material on enriched algebraic theories, we specify the blanket assumptions in place throughout the paper (\ref{para:setting2}), and we establish some further supporting results that we require (\ref{thm:bieq_th_algcat}, \ref{thm:loc-sound-eleutheric}).

We write $\V$ to denote a given symmetric monoidal closed category whose underlying ordinary category $\V_0$ has finite limits, and we employ the theory of categories enriched in $\V$ \cite{Ke:Ba,Dub}. We use the notation and terminology of \cite{Ke:Ba} throughout, except where otherwise indicated. As usual, we use the term \textit{weight} to mean \textit{indexing type} in the sense of \cite{Ke:Ba}, i.e.~a $\V$-functor whose codomain is $\V$ (and whose domain may be large).

\begin{parasub}\label{para:univ_enl}
We write $\Set$ to denote the category of sets in some universe $\U$ of \textit{small sets} for which $\V_0$ is locally small. We say that a $\V$-category $\C$ is \textit{small} if $\C$ is equivalent to a $\V$-category with a small set of objects, while we call a weight \textit{small} its domain is so. We employ the universe enlargement methodology that is discussed in \cite[\S 3.11]{Ke:Ba}. In particular, we let $\U'$ be a universe for which both $\Set$ and $\V_0$ are $\U'$-small, and for which all given $\V$-categories under discussion are $\U'$-small, and we write $\Set'$ for the category of $\U'$-small sets. With the convention of \cite{Ke:Ba}, we write $\V'$ for a symmetric monoidal closed category that has all $\U'$-small limits and colimits and is equipped with a fully faithful symmetric strong monoidal functor $\V \hookrightarrow \V'$ (written as an inclusion) that preserves all limits that exist in $\V$.
\end{parasub}

\begin{parasub}\label{para:sys_arities}
Throughout the paper, we fix a \bit{system of arities} in the sense of \cite[3.1]{Lu:EnrAlgTh}, i.e.~a fully faithful symmetric strong monoidal $\V$-functor $j:\J \rightarrow \V$. In view of \cite[3.9]{Lu:EnrAlgTh}, we assume without loss of generality that $j$ is the inclusion of a full sub-$\V$-category $\J \hookrightarrow \V$ closed under the monoidal product and containing the unit object $I$. The inclusion $j$ is dense, in the enriched sense, simply because $I \in \J$. In \cite{LuPa:Pres,LuPa:DiagrPres,LuPa:Struct}, a dense full sub-$\V$-category of a $\V$-category $\C$ is called a \textit{subcategory of arities}, so in the present case $\J \hookrightarrow \V$ is a subcategory of arities in $\C = \V$.
\end{parasub}

\begin{parasub}\label{para:jcots}
Given an object $X$ of $\V$ and an object $C$ of a $\V$-category $\C$, a \bit{power} $C^X$ (resp.~\bit{copower} $X \cdot C$) in $\C$ is by definition a \textit{cotensor} (resp.~\textit{tensor}) of $C$ by $X$ in the sense of \cite[\S 3.7]{Ke:Ba}. Given a system of arities $j:\J \hookrightarrow \V$, we say that a $\V$-category $\C$ \bit{has $\J$-powers} if $\C$ is equipped with a (designated) power $C^J$ for each $J \in \ob\J$ and each $C \in \ob\C$, of course with a specified counit $\gamma_J:J \to \C(C^J,C)$ in the sense of \cite[(3.43)]{Ke:Ba}.  We assume that $C^I = C$ and that $\gamma_I:I \to \C(C,C)$ is the identity.
\end{parasub}

\begin{parasub}\label{para:jth}
A \bit{$\J$-theory} (enriched in $\V$) is a $\V$-category $\T$ equipped with a $\V$-functor $\tau:\J^\op \rightarrow \T$ that is the identity on objects and preserves $\J$-powers \cite[4.1]{Lu:EnrAlgTh}. Each object $J$ of $\J^\op$ is itself a power $I^J = J$ in $\J^\op$, so by applying the $\V$-functor $\tau$ associated to a $\J$-theory $\T$ we find that each object  $J \in \ob\T = \ob\J$ is a power $I^J = J$ in $\T$ \cite[5.8]{Lu:EnrAlgTh}. Thus $\tau$ supplies designated $\J$-powers of $I$. Moreover, $\T$ has $\J$-powers \cite[4.3, 4.5]{Lu:EnrAlgTh}.

There is a category $\ThJ$ whose objects are $\J$-theories and whose morphisms $M:\T_1 \rightarrow \T_2$ are $\V$-functors that commute with the associated $\V$-functors $\tau_i:\J^\op \rightarrow \T_i$ $(i = 1,2)$, equivalently, that strictly preserve the designated $\J$-powers of $I$ \cite[5.16]{Lu:EnrAlgTh}. It then follows that $M$ preserves $\J$-powers \cite[5.16]{Lu:EnrAlgTh}.

A $\V$-functor $G:\C \to \D$ is \bit{faithful} if $G_{AB}:\C(A,B) \to \D(GA,GB)$ is a monomorphism in $\V$ for all $A,B \in \ob\C$. A \bit{subtheory} of a $\J$-theory $\T$ is a $\J$-theory $\sS$ with a morphism of $\J$-theories $\sS \to \T$ that is faithful as a $\V$-functor.
\end{parasub}

\begin{parasub}\label{para:talg}
Let $\T$ be a $\J$-theory and $\C$ a $\V$-category with $\J$-powers. A \bit{$\T$-algebra} in $\C$ is a $\V$-functor $A:\T \rightarrow \C$ that preserves $\J$-powers. In view of \ref{para:jth}, this entails that $AJ \cong (AI)^J$ for each $J \in \ob\J = \ob\T$. By definition, a \bit{normal $\T$-algebra} in $\C$ \cite[5.10]{Lu:EnrAlgTh} is a $\V$-functor $A:\T \to \C$ that sends the designated $\J$-powers $I^J = J$ of $I$ in $\T$ to the designated $\J$-powers $(AI)^J$ $(J \in \ob\J)$ of $AI$ in $\C$. Every normal $\T$-algebra is, in particular, a $\T$-algebra \cite[5.9, 5.10]{Lu:EnrAlgTh}.

As usual $[\T,\C]$ denotes the $\V'$-category of $\V$-functors from $\T$ to $\C$, which is not in general a $\V$-category since in general $\T$ need not be small and $\V_0$ need not be complete. We write $\Alg{\T}(\C)$ (resp.~$\Alg{\T}^{\,!}(\C)$) for the full sub-$\V'$-category of $[\T,\C]$ spanned by the $\T$-algebras (resp.~the normal $\T$-algebras), and we write $\pAlg{\T}(\C)$ and $\pAlg{\T}^!(\C)$ and for the ordinary categories underlying $\Alg{\T}(\C)$ and $\Alg{\T}^{\,!}(\C)$, respectively\footnote{Although not the usual notation for underlying ordinary categories, this notation is especially convenient in this paper because of our frequent use of these ordinary categories within further formal constructions.}. We
are mainly interested in cases where $\Alg{\T}(\C)$ \textit{exists as a $\V$-category}, i.e.~when the ends needed in forming $\Alg{\T}(\C)$ exist in $\V$. There are significant cases where this is so even when $\J$ is not small, as we discuss below. Regardless, the inclusion $\Alg{\T}^{\,!}(\C) \hookrightarrow \Alg{\T}(\C)$ is an equivalence $\Alg{\T}^{\,!}(\C) \simeq \Alg{\T}(\C)$ by \cite[5.14]{Lu:EnrAlgTh}.

The \bit{carrier} of a $\T$-algebra $A:\T \to \C$ is the object $|A| = AI$ of $\C$. There is a faithful $\V'$-functor $G^\T:\Alg{\T}(\C) \rightarrow \C$ that is given by evaluation at the unit object $I$ (\cite[5.4]{Lu:EnrAlgTh}, \cite[4.8]{Lu:Cmt}), and whose underlying ordinary functor we denote also by $G^\T$. If $f:A \rightarrow B$ is a morphism in $\pAlg{\T}(\C)$, then we say that $f$ is a \bit{$\T$-homomorphism} and we write $|f| = G^\T f:|A| \rightarrow |B|$.
\end{parasub}

\begin{parasub}[\textbf{The setting}]\label{para:setting2}
Throughout the rest of this paper, we let $j:\J \hookrightarrow \V$ be a system of arities in a symmetric monoidal closed category $\V$ such that $\V_0$ has finite limits. We suppose throughout that $j:\J \hookrightarrow \V$ is  \bit{amenable} \cite[4.12]{LuPa:Struct}, meaning that for every $\J$-theory $\T$, $\Alg{\T}(\V)$ exists as a $\V$-category and the $\V$-functor $G^\T:\Alg{\T}(\V) \to \V$ has a left adjoint $F^\T$. We discuss broad classes of examples of amenable systems of arities in \ref{thm:loc-bdd-small-sys-ar}--\ref{exa:sys-ar}.
\end{parasub}

\begin{parasub}\label{para:talgs}
Given a $\J$-theory $\T$, we abbreviate our notation by writing
$$\Alg{\T} = \Alg{\T}(\V),\;\;\;\;\;\;\pAlg{\T} = \pAlg{\T}(\V),$$
so that we have a faithful $\V$-functor
$$G^\T:\Alg{\T} \longrightarrow \V\;.$$
We call $\T$-algebras in $\V$ simply \bit{$\T$-algebras} when there is no cause for confusion. Note that $\Alg{\T}$ has $\J$-powers, formed pointwise. In view of the Yoneda lemma, there is a fully faithful $\V$-functor $Y:\T^\op \to \Alg{\T}$ given by $YJ = \T(J,-)$ $(J \in \T)$, and $Y$ is also dense, by \cite[\S 5.1, Theorem 5.13]{Ke:Ba}. Its opposite $Y^\op:\T \to \Alg{\T}^\op$ preserves $\J$-powers (i.e.~sends $\J$-powers in $\T$ to copowers in $\Alg{\T}$) because for each $\T$-algebra $A$, the $\V$-functor $\Alg{\T}(Y^\op-,A) \cong A:\T \to \V$ preserves $\J$-powers.

By definition, an \bit{algebra} (with arities $\J$) is a pair $(\T,A)$ consisting of a $\J$-theory $\T$ and a $\T$-algebra $A$.  We also write simply $A$ to denote the algebra $(\T,A)$.

There are functors
\begin{equation}\label{eq:func_talg}\Alg{(-)}:\ThJ^\op \rightarrow \VCAT,\;\;\;\;\pAlg{(-)}:\ThJ^\op \rightarrow \CAT\end{equation}
sending each $\J$-theory $\T$ to $\Alg{\T}$ and to its underlying ordinary category $\pAlg{\T}$, respectively.
The first of these functors sends each morphism of $\J$-theories $M:\T \rightarrow \U$ to the $\V$-functor
\begin{equation}\label{eq:m_upper_star}M^*\;:\;\Alg{\U} \longrightarrow \Alg{\T}\;,\end{equation}
given by $M^*(A) = AM$ $(A \in \Alg{\U})$, while the second sends $M$ to the ordinary functor underlying $M^*$, which we write also as $M^*$. Similarly, there are functors 
\begin{equation}\label{eq:func_normal_talg}\Alg{(-)}^{\,!}:\ThJ^\op \rightarrow \VCAT,\;\;\;\;\pAlg{(-)}^!:\ThJ^\op \rightarrow \CAT.\end{equation}

\begin{parasub}\label{para:jalg}
A \textit{$\V$-category over $\V$}, denoted by $(\A,G)$ or simply $\A$, is a $\V$-category $\A$ equipped with a $\V$-functor $G:\A \to \V$. Thus $\V$-categories over $\V$ are the objects of the pseudo-slice 2-category $\VCAT \sslash \V$ \pref{para:pseudo_slice} and of the (ordinary) slice category $\VCAT_0 \slash \C$ where $\VCAT_0$ is the ordinary category of $\V$-categories. A 1-cell in $\VCAT \sslash \V$ is an equivalence in $\VCAT \sslash \V$ iff its underlying $\V$-functor is an equivalence \cite[2.5]{Lu:EnrAlgTh}. For each $\J$-theory $\T$, we regard $\Alg{\T}$ and $\Alg{\T}^!$ as $\V$-categories over $\V$ via $G^\T$ and its restriction to $\Alg{\T}^!$, respectively. The inclusion $\Alg{\T}^! \hookrightarrow \Alg{\T}$ is an equivalence (\ref{para:talgs}) and so underlies an equivalence $\Alg{\T}^! \simeq \Alg{\T}$ in $\VCAT \sslash \V$. Given a $\V$-category $(\A,G)$ over $\V$, we say that $(\A,G)$ is a \bit{$\J$-algebraic} (or that $G$ is $\J$-algebraic) if there exists an equivalence $\A \simeq \Alg{\T}$ in $\VCAT\sslash\V$ for some $\J$-theory $\T$ (\cite[12.1]{Lu:EnrAlgTh}, \cite[4.29--4.31]{LuPa:Struct}). We say that $(\A,G)$ (or $G$) is \bit{strictly $\J$-algebraic} if there is an isomomorphism $\A \cong \Alg{\T}^!$ in $\VCAT_0 \slash \V$ for some $\J$-theory $\T$ \cite[4.29--4.31]{LuPa:Struct}. Hence a $\V$-category over $\V$ is $\J$-algebraic iff it is equivalent, in $\VCAT \sslash \V$, to a strictly $\J$-algebraic $\V$-category over $\V$. Intrinsic characterizations of $\J$-algebraic $\V$-functors are given in \cite[4.35]{LuPa:Struct} and, in the case where $\J \hookrightarrow \V$ is \textit{eleutheric} (\ref{para:eleutheric}), in \cite[12.2]{Lu:EnrAlgTh}.

Let $G:\A \to \V$ be a $\V$-functor. If $G$ is $\J$-algebraic (resp.~strictly $\J$-algebraic), then $G$ is monadic (resp.~strictly monadic) \cite[5.14]{LuPa:Struct}, so in particular, $G$ is conservative (i.e.~if $Gf$ is an isomorphism then $f$ is an isomorphism). Moreover, we note in passing that, by \cite[5.14]{LuPa:Struct}, a $\V$-functor $G:\A \to \V$ is $\J$-algebraic (resp.~strictly $\J$-algebraic) if and only if $G$ is monadic (resp.~strictly monadic) and the associated $\V$-monad on $\V$ is \textit{$\J$-nervous} as defined in \cite[5.9]{LuPa:Struct} (following \cite[Definition 17]{BouGar:MndTh}). Furthermore, there is an equivalence of categories between $\ThJ$ and the category of $\J$-nervous $\V$-monads on $\V$ \cite[5.13]{LuPa:Struct}.

A $\V$-functor $G:\A \to \C$ is a \textit{discrete isofibration} if its underlying ordinary functor $G_0$ is a discrete isofibration (in the sense discussed in \cite{Lack:2CatsCompanion}), i.e., for every isomorphism $g:GA \xrightarrow{\sim} X$ in $\C$ (where $A \in \ob\A$), there is a unique isomorphism $f:A \xrightarrow{\sim} A'$ in $\A$ with $Gf = g$ (and in particular $GA' = X$). If a $\V$-functor $G:\A \to \C$ is a discrete isofibration, then the functor $\VCAT(\B,G):\VCAT(\B,\A) \to \VCAT(\B,\C)$ is discrete isofibration for each $\V$-category $\B$. Every strictly monadic functor is a discrete isofibration (e.g.~by \cite[Proposition 20.12]{AHS}). In particular, every strictly $\J$-algebraic $\V$-functor $G:\A \to \V$ is a discrete isofibration.

We write
$$\ALGCATJ \hookrightarrow \VCAT \sslash \V$$
for the full sub-2-category spanned by the $\J$-algebraic $\V$-categories over $\V$. A 1-cell $(Q,\gamma):(\A,G) \to (\B,H)$ in $\ALGCATJ$ therefore consists of a $\V$-functor $Q:\A \to \B$ and a $\V$-natural isomorphism $HQ \overset{\sim}{\Rightarrow} G:\A \to \V$. We write
$$\ALGCATJ^! \hookrightarrow \VCAT_0 \slash \V$$
for the full subcategory spanned by the strictly $\J$-algebraic $\V$-categories over $\V$. Regarding $\ALGCATJ^!$ as a locally discrete 2-category, we shall show that there is a biequivalence $\ALGCATJ^! \simeq \ALGCATJ$, but first we recall some standard material on biequivalences in the case where the bicategories involved are (strict) 2-categories:
\end{parasub}

\begin{parasub}\label{sec:bieq}
A $2$-functor $G:\A \rightarrow \X$ is a \bit{local equivalence} if for all objects $A$ and $B$ of $\A$, the functor $G_{AB}:\A(A,B) \rightarrow \X(GA,GB)$ is an equivalence of categories.  We say that $G$ is \bit{biessentially surjective on objects} if for every object $X$ of $\X$ there is an object $A$ of $\A$ and an equivalence $GA \simeq X$ in $\X$.  Putting these two properties together, $G$ is a \bit{biequivalence} if $G$ is a local equivalence and is biessentially surjective on objects \cite[(1.33)]{Str:FibrBicats}.   This definition generalizes immediately to the case where $G:\A \rightarrow \X$ is assumed only to be a homomorphism between bicategories $\A$ and $\X$ and provides a standard notion of equivalence of bicategories.  In particular, if $G$ is a biequivalence, then there is a homomorphism of bicategories $F:\X \rightarrow \A$ that is a left biadjoint of $G$ and is also a biequivalence \cite[(1.33)]{Str:FibrBicats}.
\end{parasub}

\begin{propsub}\label{thm:iota-bieq}
Regarding $\ALGCATJ^!$ as a locally discrete 2-category, the evident 2-functor $\iota:\ALGCATJ^! \to \ALGCATJ$ is a biequivalence $$\ALGCATJ^! \simeq \ALGCATJ\;.$$
\end{propsub}
\begin{proof}
Let us write $\K = \ALGCATJ$ and $\K^! = \ALGCATJ^!$. The 2-functor $\iota:\K^! \to \K$ sends each 1-cell $P:(\A,G) \to (\B,H)$ in $\K^!$ to the 1-cell $(P,1_G):(\A,G) \to (\B,H)$ in $\K$. In view of \ref{para:jalg}, $\iota$ is biessentially surjective on objects, so it suffices to let $\A = (\A,G)$ and $\B = (\B,H)$ be objects of $\K^!$ and show that the functor $\iota_{\A\B}:\K^!(\A,\B) \to \K(\A,\B)$ is an equivalence, recalling that $\K^!(\A,\B)$ is a discrete category. In general, if $\X$ is any discrete category, then an equivalence $F:\X \to \Y$ is equivalently a family of objects $FX$ of $\Y$ $(X \in \X)$ such that the following statements hold for each object $Y$ of $\Y$: (1) There is a unique pair $(X,f)$ with $X \in \X$ and $f:Y \to FX$ in $\Y$, and (2) $f$ is an isomorphism. To apply this to $\iota_{\A\B}$, let $(Q,\gamma):\A \to \B$ be a 1-cell in $\K$, so that $Q:\A \to \B$ in $\VCAT$ and $\gamma:HQ \overset{\sim}{\Rightarrow} G$ in $\VCAT$. Since $H:\B \to \V$ is a discrete isofibration (\ref{para:jalg}), there is a unique pair $(P,\phi)$ consisting of a 1-cell $P:\A \to \B$ in $\VCAT$ and an invertible 2-cell $\phi:Q \overset{\sim}{\Rightarrow} P$ in $\VCAT$ such that $HP = G$ and $H\phi = \gamma$. Equivalently, there is a unique pair $(P,\phi)$ consisting of a 1-cell $P:(\A,G) \to (\B,H)$ in $\K^!$ and an invertible 2-cell $\phi:(Q,\gamma) \overset{\sim}{\Rightarrow} (P,1_G):(\A,G) \to (\B,H)$ in $\K$. But \textit{any} 2-cell $\phi:(Q,\gamma) \Rightarrow (P,1_G)$ in $\K$ is necessarily invertible, because $H\phi = \gamma$ is invertible and $H$ is conservative (\ref{para:jalg}).
\end{proof}

By \cite[4.32]{LuPa:Struct}, the functor $\Alg{(-)}^{\,!}$ in \eqref{eq:func_normal_talg} lifts to an equivalence of (ordinary) categories
\begin{equation}\label{eqn:equiv_thj_str-jalg}\sAlgs^!\;:\;\ThJ^\op\overset{\sim}{\longrightarrow} \ALGCATJ^!\;.\end{equation}
There is also a functor
\begin{equation}\label{eqn:func_alg_valued_in_psslice}\sAlgs\;:\;\ThJ^\op\longrightarrow \ALGCATJ\end{equation}
given on objects by $\T \mapsto (\Alg{\T},G^\T)$ and on morphisms by $M \mapsto (M^*,1)$. Regarding $\ThJ$ as a locally discrete 2-category, we may regard $\sAlgs$ as a 2-functor, which is in fact a biequivalence:
\end{parasub}

\begin{thmsub}\label{thm:bieq_th_algcat}
The 2-functor $\sAlgs$ in \eqref{eqn:func_alg_valued_in_psslice} is a biequivalence $\ThJ^\op \simeq \ALGCATJ$.
\end{thmsub}
\begin{proof}
In view of \ref{para:jalg}, we have equivalences $\sAlgs^! \T \xrightarrow{\sim} \sAlgs \T$ in $\ALGCATJ$ that are (strictly) natural in $\T \in \ThJ$. Thus the composite
$$\ThJ^\op \xrightarrow{\sAlgs^!} \ALGCATJ^! \xrightarrow{\iota} \ALGCATJ$$
is equivalent to $\sAlgs$, which is therefore a biequivalence, by \ref{thm:iota-bieq} and \eqref{eqn:equiv_thj_str-jalg}.
\end{proof}

Next we discuss two broad classes of examples of amenable systems of arities. For the first, recall that Kelly introduced the notion of \textit{locally bounded} symmetric monoidal closed category \cite[6.1]{Ke:Ba} as a general setting for the advanced aspects of enriched category theory developed in \cite[Chapter 6]{Ke:Ba}. Examples of locally bounded symmetric monoidal closed categories include all locally presentable symmetric monoidal closed categories and also various convenient categories of topological spaces and related structures; we refer the reader to \cite[6.1]{Ke:Ba} and \cite[5.14]{LuPa:LocBdd} for extensive lists of examples. We have the following result on the plenitude of amenable systems of arities in the locally bounded setting, by \cite[6.3.12]{LuPa:Struct}:

\begin{thmsub}\label{thm:loc-bdd-small-sys-ar}
If $\V$ is a \textit{locally bounded} symmetric monoidal closed category, then every small system of arities $\J \hookrightarrow \V$ is amenable.
\end{thmsub}

\begin{parasub}\label{para:eleutheric}
Our second broad class of examples of amenable subcategories of arities includes cases where $\J$ is not small and $\V$ is not locally bounded (e.g. $\J = \V$, \ref{exa:sys-ar}). A system of arities $j:\J \hookrightarrow \V$ is \bit{eleutheric} \cite{Lu:EnrAlgTh,LuPa:Pres,LuPa:Struct} if for every $\V$-functor $S:\J \to \V$ the (enriched) left Kan extension $\Lan_j S:\V \to \V$ exists and is preserved by $\V(J,-):\V \to \V$ for each $J \in \ob\J$.  Equivalently, a system of arities $j:\J \hookrightarrow \V$ is eleutheric if $\V$ has $\PhiJ$-colimits for the class of (possibly large) weights $\PhiJ = \{\V(j-,X) \mid X \in \ob\V\}$ and each $\V$-functor $\V(J,-):\V \to \V$ $(J \in \ob\J)$ preserves $\PhiJ$-colimits. More abstractly, a system of arities $j:\J \hookrightarrow \V$ is eleutheric iff $j$ presents $\V$ as a free $\PhiJ$-cocompletion of $\J$ \cite[7.8]{Lu:EnrAlgTh}. Given an eleutheric system of arities $\J \hookrightarrow \V$, a \bit{$\J$-ary $\V$-monad} is a $\V$-monad on $\V$ whose underlying $\V$-functor $T:\V \to \V$ preserves $\PhiJ$-colimits (equivalently, preserves left Kan extensions along $j$). We have the following result, by \cite[8.8, 8.9]{Lu:EnrAlgTh} (or \cite[6.1.9]{LuPa:Struct}), \cite[11.8, 12.2]{Lu:EnrAlgTh}, and \cite[3.1.5]{Lu:FDistn}:
\end{parasub}

\begin{thmsub}\label{thm:eleutheric-implies-amenable}
Every eleutheric system of arities $j:\J \hookrightarrow \V$ is amenable. If $j$ is eleutheric, then (1) there is an equivalence of categories between $\ThJ$ and the category of $\J$-ary $\V$-monads on $\V$, and (2) a $\V$-functor $G:\A \to \V$ is $\J$-algebraic (resp.~strictly $\J$-algebraic) if and only if $G$ is monadic (resp.~strictly monadic) and the associated $\V$-monad on $\V$ is $\J$-ary.
\end{thmsub}

\noindent Moreover, for eleutheric $\J$, a $\V$-monad on $\V$ is $\J$-ary iff it is $\J$-nervous \cite[6.1.15]{LuPa:Struct}.

\begin{parasub}\label{para:small-ele}
In the case where $\V$ is complete and cocomplete, a small system of arities $j:\J \hookrightarrow \V$ is eleutheric iff there exists a class of small weights $\Psi$ such that $j:\J \hookrightarrow \V$ is a free $\Psi$-cocompletion \cite[3.8]{LuPa:Pres}, i.e.~$j$ presents $\V$ as a free $\Psi$-cocompletion of $\J$, in which case a $\V$-monad on $\V$ is $\J$-ary iff it preserves $\Psi$-colimits \cite[4.2]{LuPa:Pres}.
\end{parasub}

\begin{parasub}\label{para:weakly-sound}
Again considering the case where $\V$ is complete and cocomplete, an important class of examples of eleutheric systems of arities arises by starting with a class of small weights $\Phi$ that is \textit{locally small} \cite[8.10]{KeSch} and \textit{weakly sound} \cite{DayLack,KeSch,LaRo,Tendas:MoreOnSoundness}. A class of small weights $\Phi$ is said to be \textit{locally small} if the free $\Phi$-cocompletion $\Phi(\C)$ of any small $\V$-category $\C$ is small; recall that $\Phi(\C)$ is constructed as the closure $\Phi(\C) \hookrightarrow [\C^\op,\V]$ of the representables under $\Phi$-colimits, and that a weight $W:\C^\op \to \V$ lies in $\Phi(\C)$ iff $W$ lies in the \textit{saturation} $\Phi^*$ of $\Phi$ \cite{AlbertKelly,KeSch} (i.e.~every $\Phi$-cocomplete $\V$-category has $W$-colimits and every $\Phi$-cocontinuous $\V$-functor between $\Phi$-cocomplete $\V$-categories preserves $W$-colimits). A weight $W:\C^\op \to \V$ is \textit{$\Phi$-flat} if $\C$ is small and $W$-colimits commute with $\Phi$-limits in $\V$, i.e.~$W * (-):[\C,\V] \to \V$ is $\Phi$-continuous \cite{KeSch}; we write $\Phi^+$ for the class of $\Phi$-flat weights. We say that $\Phi$ is \textit{weakly sound} if for every small $\Phi$-cocomplete $\V$-category $\C$, every $\Phi$-continuous $\V$-functor $W:\C^\op \to \V$ is a $\Phi$-flat weight.

As discussed in \cite[\S 6.4]{LaRo} and \cite[3.9(7)]{LuPa:Pres} (on the basis of \cite[\S 8]{KeSch}), if $\Phi$ is a locally small and weakly sound class of weights and $\C$ is a small $\Phi$-cocomplete $\V$-category, then the canonical $\V$-functor $\y_\Phi:\C \to \textnormal{$\Phi$-Cts}(\C^\op,\V)$ is a free $\Phi^+$-cocompletion, where $\textnormal{$\Phi$-Cts}(\C^\op,\V) \hookrightarrow [\C^\op,\V]$ is the full sub-$\V$-category spanned by the $\Phi$-continuous $\V$-functors. (Consequently, $\y_\Phi$ is an \textit{eleutheric subcategory of arities}, \cite[3.9(7)]{LuPa:Pres}).

Given a locally small and weakly sound class of weights $\Phi$, we now consider the free $\Phi$-cocompletion $\Phi(\mathbb{I})$ of the unit $\V$-category $\mathbb{I}$, which is simply the full sub-$\V$-category $\Phi(\mathbb{I}) \hookrightarrow \V$ obtained as the closure of the unit object $I$ in $\V$ under $\Phi$-colimits. Under much stronger assumptions on $\V$, Tendas \cite[\S 5]{Tendas:MoreOnSoundness} proved that $\Phi(\mathbb{I}) \hookrightarrow \V$ is a system of arities and discussed $\Phi(\mathbb{I})$-theories and $\V$-monads on $\V$ preserving $\Phi$-flat colimits. Even without those stronger assumptions, $\Phi(\mathbb{I}) \hookrightarrow \V$ is an eleutheric system of arities and the latter are the $\Phi(\mathbb{I})$-ary $\V$-monads:
\end{parasub}

\begin{thmsub}\label{thm:loc-sound-eleutheric}
Suppose $\V$ is complete and cocomplete, and let $\Phi$ be a locally small and weakly sound class of weights. Then $\Phi(\mathbb{I}) \hookrightarrow \V$ is a small and eleutheric system of arities. Also, a $\V$-monad $T:\V \to \V$ is $\Phi(\mathbb{I})$-ary iff $T$ preserves $\Phi$-flat colimits.
\end{thmsub}
\begin{proof}
Locally smallness of $\Phi$ entails that $\Phi(\mathbb{I})$ is small. By \ref{para:weakly-sound}, the canonical $\V$-functor $\Phi(\mathbb{I}) \to \textnormal{$\Phi$-Cts}(\Phi(\mathbb{I})^\op,\V)$ presents $\textnormal{$\Phi$-Cts}(\Phi(\mathbb{I})^\op,\V) \simeq [\mathbb{I},\V] = \V$ as a free $\Phi^+$-cocompletion of $\Phi(\mathbb{I})$, so the inclusion $\Phi(\mathbb{I}) \hookrightarrow \V$ presents $\V$ as a free $\Phi^+$-cocompletion of $\Phi(\mathbb{I})$. Hence, by \ref{para:small-ele} it suffices to show that $\Phi(\mathbb{I}) \hookrightarrow \V$ is a system of arities. An object $X$ of $\V$ lies in $\Phi(\mathbb{I})$ iff $X:\mathbb{I}^\op \to \V$ (regarded as a weight) lies in the saturation $\Phi^*$, i.e.~iff every $\Phi$-cocomplete $\V$-category has $X$-copowers and every $\Phi$-cocontinuous $\V$-functor between $\Phi$-cocomplete $\V$-categories preserves $X$-copowers. Consequently, $\Phi(\mathbb{I}) \hookrightarrow \V$ is closed under the monoidal product, since if $X,Y \in \Phi(\mathbb{I})$ then every $\Phi$-cocomplete $\V$-category $\A$ has $X \otimes Y$-copowers $(X \otimes Y) \cdot A \cong X \cdot (Y \cdot A)$ $(A \in \ob\A)$, and every $\Phi$-cocontinuous $\V$-functor between $\Phi$-cocomplete $\V$-categories preserves $(X \otimes Y)$-copowers, showing that $X \otimes Y \in \Phi^*$ and hence $X \otimes Y \in \Phi(\mathbb{I})$.
\end{proof}

We now discuss some specific examples, recalling that all we require of $\V$ in general is that $\V_0$ have finite limits.

\begin{exasub}\label{exa:sys-ar}\emptybox

\medskip

\noindent(1) The full subcategory $\FinCard \hookrightarrow \Set$ spanned by the finite cardinals is an eleutheric system of arities, for which $\FinCard$-theories are precisely Lawvere theories in the usual sense, and for which $\FinCard$-ary monads are precisely finitary monads \cite[4.2, 7.5, 11.3]{Lu:EnrAlgTh}.

\medskip

\noindent(2) The identity functor on $\Set$ is an eleutheric system of arities $\Set \hookrightarrow \Set$  for which $\Set$-theories are precisely Linton's \textit{varietal theories} \cite{Lin:Eq}, and $\Set$-ary monads are simply arbitrary monads on $\Set$ \cite[4.2, 7.5, 11.3]{Lu:EnrAlgTh}.

\medskip

\noindent(3) Generalizing (2), the identity $\V$-functor on $\V$ is an eleutheric system of arities $\J = \V \hookrightarrow \V$,  for which $\J$-theories are precisely the $\V$-theories of Dubuc \cite{Dub:StrSem}, and $\J$-ary $\V$-monads are arbitrary $\V$-monads on $\V$ \cite[4.2,7.5,11.3]{Lu:EnrAlgTh}.

\medskip

\noindent(4) There is an eleutheric system of arities $\J = \{I\} \hookrightarrow \V$ spanned by just the unit object $I$ of $\V$, for which $\Th_{\{I\}}$ is equivalent to the category $\Mon(\V)$ on monoids in $\V$ \cite[4.2]{Lu:EnrAlgTh}. Given a monoid $R$ in $\V$, with corresponding $\{I\}$-theory $\T$, a $\T$-algebra is equivalently a left $R$-module in $\V$ \cite[5.3]{Lu:EnrAlgTh} (i.e., an $R$-left-object in the sense of \cite[\S 2]{Parei:I}). In particular, left $R$-modules in $\V$ are the objects of the $\V$-category $\Mod{R} = \Alg{\T}$.

\medskip

\noindent(5) Generalizing (1), if $\V$ is \textit{locally $\alpha$-presentable as a closed category} in the sense of \cite{Ke:FL} (where $\alpha$ is a regular cardinal), then the full sub-$\V$-category $\V_\alpha \hookrightarrow \V$ spanned by the $\alpha$-presentable objects is an eleutheric system of arities, for which $\V_\alpha$-theories are precisely (the $\alpha$-ary version of) the enriched Lawvere theories of Power \cite{Pow:EnrLaw}, and $\V_\alpha$-ary $\V$-monads are precisely $\alpha$-ary $\V$-monads on $\V$ (i.e.~those that preserve conical $\alpha$-filtered colimits) \cite[3.9, 4.7]{LuPa:Pres}.

\medskip

\noindent(6) If $\V$ is a complete and cocomplete cartesian closed category, or more generally, a \textit{$\pi$-category} in the sense of \cite{BoDay}, then by \cite[4.2, 7.5]{Lu:EnrAlgTh} there is an eleutheric system of arities $\SF(\V) \hookrightarrow \V$ spanned by the copowers $n \cdot I$ of the unit object $I$ by finite cardinals $n$, for which $\SF(\V)$-theories are the enriched algebraic theories of Borceux and Day \cite{BoDay}. In the case where $\V$ is cartesian closed, $\SF(\V)$-ary $\V$-monads are precisely \textit{strongly finitary $\V$-monads} on $\V$ in the sense of \cite{KeLa:SF}, by \cite[4.7]{LuPa:Pres}.
\end{exasub}

\begin{parasub}\label{para:specific-cccs}
As in \cite{Lu:BAlgCmt}, in \ref{sec:str-fin-alg-dual-adjn} we briefly discuss some examples involving two specific cartesian closed categories, as follows.

Firstly, we consider the category of \textit{convergence spaces} in the sense of \cite{BeBu}, $\Conv$, i.e.~the category of Limesr\"aume (or limit spaces) of Fischer \cite{Fischer} and Kowalsky \cite{Kowalsky}, into which the category of topological spaces $\Top$ embeds as a full, reflective subcategory. Following \cite[2.3]{Lu:FDistn}, we now provide a brief overview. A \textit{convergence space} is a (small) set $X$ with a binary relation $\to$ from the set of all \textit{proper filters on $X$} (i.e.~filters $\X \subsetneq \mathcal{P}(X)$) to $X$ such that (1) for every $x \in X$, $\{A \subseteq X \mid x \in A\} \to x$; (2) if $\X \to x$ and $\X \subseteq \Y$ for proper filters $\X$ and $\Y$, then $\Y \to x$; (3) if $\X \to x$ and $\Y \to x$ then $\X \cap \Y \to x$. The internal hom from $X$ to $Y$ in $\Conv$ is the convergence space $C_c(X,Y)$ of all continuous maps from $X$ to $Y$, equipped with \textit{continuous convergence}; see \cite[1.1.5]{BeBu} or \cite[2.3]{Lu:FDistn}. Explicitly, a map $f:X \to Y$ is \textit{continuous} if $f[\X] \to f(x)$ in $Y$ whenever $\X \to x$ in $X$, where $f[\X] = \{B \subseteq Y \mid f(A) \subseteq B, A \in \X\}$. If $\F$ is a proper filter on the set $C_c(X,Y)$, then $\F \to f$ iff $\F[\X] \to f(x)$ in $Y$ whenever $\X \to x$ in $X$, where we write $\F[\X] = \{B \subseteq Y \mid F(A) \subseteq B, F \in \F, A \in \X\}$ and $F(A) = \{g(a) \mid g \in F, a \in A\}$. Identifying $\Top$ with a full subcategory of $\Conv$, the inclusion $\Top \hookrightarrow \Conv$ preserves all limits, and upon identifying sets with discrete topological spaces, the inclusion $\Set \hookrightarrow \Top$ preserves finite limits. Consequently, every Lawvere theory $\T$ in the usual sense is an example of an $\SF(\Conv)$-theory, and every $\T$-algebra in $\Set$ or in $\Top$ determines a $\T$-algebra in $\Conv$.
 
Secondly, in two examples in \ref{sec:str-fin-alg-dual-adjn} we briefly mention the \textit{Cahiers topos} \cite{Dub:ModSDG,Kock:SDG}, $\Cah$, which is a (fully) well-adapted model of synthetic differential geometry (\textit{un mod\`ele pleinement bien adapt\'e}, \cite[4.10]{Dub:ModSDG}); in particular, there is a fully faithful functor from the category of smooth manifolds into $\Cah$, under which the manifold $\RR$ is sent to an internal commutative ring $R$ in $\Cah$ called the \textit{line object}.
\end{parasub}

\subsection{Bifold algebras}\label{sec:balg}

We now review some basic material on bifold algebras from \cite{Lu:BAlgCmt}. Given $\J$-theories $\T$ and $\U$, a \bit{$(\T,\U)$-algebra} is a functor $D:\T \otimes \U \rightarrow \V$ that preserves $\J$-powers in each variable separately \cite[5.2]{Lu:BAlgCmt}, where $\T \otimes \U$ is simply the monoidal product of $\V$-categories.  We call $|D| = D(I,I)$ the \bit{carrier} of $D$.  A \bit{bifold algebra} is a triple $(\T,\U,D)$ consisting of $\J$-theories $\T$ and $\U$ and a $(\T,\U)$-algebra $D$.  We often write simply $D$ to denote the bifold algebra $(\T,\U,D)$, thus implicitly regarding $(\T,\U)$-algebras as bifold algebras. Bifold algebras are equivalently given by \textit{commuting algebra pairs}, as we recall in \ref{para:cmt}. Various examples of bifold algebras are discussed in Section \ref{sec:exa_ada}. $(\T,\U)$-algebras are the objects of a full sub-$\V'$-category of $[\T \otimes \U,\V]$ that we denote by $\Alg{(\T,\U)}$, but $\Alg{(\T,\U)}$ is in fact a $\V$-category, by  \cite[2.2.8, 11.11]{Lu:BAlgCmt}.  We call $\Alg{(\T,\U)}$ the \bit{$\V$-category of $(\T,\U)$-algebras}, and we denote its underlying ordinary category by $\pAlg{(\T,\U)}$.  There are isomorphisms
\begin{equation}\label{eq:ba_transposition_isos}\Alg{\T}(\Alg{\U}) \cong \Alg{(\T,\U)} \cong \Alg{(\U,\T)} \cong \Alg{\U}(\Alg{\T})\end{equation}
that witness how $(\T,\U)$-algebras can be described equivalently as $\T$-algebras in $\Alg{\U}$, as $(\U,\T)$-algebras, or as $\U$-algebras in $\Alg{\T}$ \cite[5.3]{Lu:BAlgCmt}.  Sufficient conditions for the existence of a $\J$-theory $\otimesJ{\T}{\U}$, called the \textit{tensor product} of $\T$ and $\U$, with $\Alg{(\T,\U)} \simeq \Alg{(\otimesJ{\T}{\U})}$ are discussed in \cite[\S 6]{Lu:BAlgCmt} on the basis of methods of \cite[\S 6.5]{Ke:Ba}, but we do not require the tensor product in this paper.

Every $(\T,\U)$-algebra $D$ determines a $\T$-algebra $D_\ell = D(-,I)$ in $\V$ and a $\U$-algebra $D_r = D(I,-)$ in $\V$ that both have the same carrier $|D|$ and are called the \bit{left face} and \bit{right face} of $D$, respectively \cite[5.4]{Lu:BAlgCmt}.  As in \cite[5.4]{Lu:BAlgCmt}, we denote the transposes of $D$ under the isomorphisms \eqref{eq:ba_transposition_isos} by
\begin{equation}\label{eq:transpose-of-bif-alg}D_{\ell r}:\T \rightarrow \Alg{\U},\;\;\;\;D^\circ:\U \otimes \T \to \V,\;\;\;\;\text{and}\;\;\;\;D_{r\ell}:\U \rightarrow \Alg{\T},\end{equation}
respectively, so that $(D_{\ell r} J)K = D(J,K) = D^\circ(K,J) = (D_{r\ell} K)J$ naturally in $J \in \T$, $K \in \U$. Thus $D_{\ell r}$ is a $\T$-algebra in $\Alg{\U}$ with carrier $D_r$, $D^\circ$ is a $(\U,\T)$-algebra, and $D_{r\ell}$ is a $\U$-algebra in $\Alg{\T}$ with carrier $D_\ell$.

\section{Representable algebra-valued $\V$-functors}\label{sec:alg-val-func}

In this section, we establish an enriched generalization of Freyd's theorem on algebra-valued functors (Theorem \ref{thm:rep-equiv-ra}) and a Yoneda principle for algebra-valued $\V$-functors (Lemma \ref{thm:yon-alg-val}), which together entail an enriched functorial version of Freyd's theorem (Theorem \ref{thm:talgs-in-c-equiv-radj}).

Let $\T$ be a $\J$-theory, and let $\C$ be a $\V$-category with $\J$-powers. We define a $\V'$-functor
\begin{equation}\label{eq:alg-val-hom}\llbracket-,?\rrbracket\;:\;\C^\op \otimes \Alg{\T}(\C) \longrightarrow \Alg{\T}\end{equation}
by
$$\llbracket C,E\rrbracket = \C(C,E-)\;\;\;\;(C \in \C, E \in \Alg{\T}(\C)).$$
Hence, for each $\T$-algebra $E:\T \to \C$, we obtain a $\V$-functor
$$\llbracket-,E\rrbracket\;:\;\C^\op \longrightarrow \Alg{\T}$$
that sends each object $C$ of $\C$ to the $\T$-algebra $\llbracket C,E\rrbracket = \C(C,E-):\T \to \V$. For each morphism of $\J$-theories $M:\U \to \T$, we note for later use that
\begin{equation}\label{eq:hom-nat}M^*\llbracket C,E\rrbracket = \llbracket C,EM\rrbracket\;\;\;\;\;\;(C \in \C, E \in \Alg{\T}(\C)).\end{equation}

\begin{defn}\label{defn:repbl-alg-val-vfunc}
A $\V$-functor $S:\C^\op \to \Alg{\T}$ is \bit{representable} if $S \cong \llbracket-,E\rrbracket$ for some $\T$-algebra $E:\T \to \C$, in which case we say that $S$ is \bit{represented by} $E$.
\end{defn}

We now study the relationship between representability of such a $\V$-functor $S$ and the existence of a left adjoint to $S$.

\begin{prop}\label{thm:rep-has-la}\emptybox
\begin{enumerate}
\item For each $\T$-algebra $E:\T \to \C$, the $\V$-functor $\llbracket-,E\rrbracket:\C^\op \to \Alg{\T}$ has a left adjoint if and only if the weighted limit $\{A,E\}$ exists in $\C$ for all $\T$-algebras $A:\T \to \V$.  In this case, $\{-,E\}:\Alg{\T} \to \C^\op$ is left adjoint to $\llbracket-,E\rrbracket$.
\item If the weighted limit $\{A,E\}$ exists for all $\T$-algebras $A:\T \to \V$ and $E:\T \to \C$, then the isomorphisms
$\C(C,\{A,E\}) \cong \Alg{\T}(A,\llbracket C,E\rrbracket)$ are $\V'$-natural in $C \in \C$, $A \in \Alg{\T}$, and $E \in \Alg{\T}(\C)$.
\end{enumerate}
\end{prop}
\begin{proof}
(1). Given a $\T$-algebra $A:\T \to \V$, since $\llbracket C,E\rrbracket = \C(C,E-):\T \to \V$ is a $\T$-algebra, the object of $\V$-natural transformations $[\T,\V](A,\C(C,E-)) = \Alg{\T}(A,\llbracket C,E\rrbracket)$ exists in $\V$, as the system of arities $\J \hookrightarrow \V$ is amenable (\ref{para:setting2}, \ref{para:talgs}). Therefore a weighted limit $\{A,E\}$ is equivalently an object of $\C$ with isomorphisms $\C(C,\{A,E\}) \cong \Alg{\T}(A,\llbracket C,E\rrbracket)$ that are $\V$-natural in $C \in \C$. (2). The $\V'$-naturality in the additional variables $A$ and $E$ is automatic, by \cite[\S 1.10]{Ke:Ba}.
\end{proof}

\begin{para}\label{para:limits-of-algs-weighted-by-algs}
We say that $\C$ \bit{admits limits of $\T$-algebras weighted by $\T$-algebras} if the weighted limit $\{A,E\}$ exists for every pair of $\T$-algebras $A:\T \to \V$ and $E:\T \to \C$. For example, if the system of arities $\J$ is small and $\C$ is complete, then $\C$ admits limits of $\T$-algebras weighted by $\T$-algebras. As another example, under just the blanket assumptions of \ref{para:setting2}, we show in Section \ref{sec:dualn_adj} that if $\T$ and $\U$ are $\J$-theories, then $\Alg{\U}$ admits limits of $\T$-algebras weighted by $\T$-algebras.
\end{para}

\begin{thm}\label{thm:rep-equiv-ra}
Let $\C$ be a $\V$-category with $\J$-powers, and suppose $\C$ admits limits of $\T$-algebras weighted by $\T$-algebras.
\begin{enumerate}
\item A $\V$-functor $S:\C^\op \to \Alg{\T}$ is representable if and only if $S$ has a left adjoint.
\item If $S$ has a left adjoint $L:\Alg{\T} \to \C^\op$, then the composite
$$\T \xrightarrow{Y^\op} \Alg{\T}^\op \xrightarrow{L^\op} \C$$
is a $\T$-algebra with $S \cong \llbracket-,L^\op Y^\op\rrbracket$, where $Y$ is as defined in \ref{para:talgs}.
\end{enumerate}
\end{thm}
\begin{proof}
The forward implication in (1) is immediate from \ref{thm:rep-has-la}, so it suffices to prove (2). Suppose $S$ has a left adjoint $L$. Then $L^\op$ preserves all limits, while $Y^\op$ preserves $\J$-powers by \ref{para:talgs}, so $L^\op Y^\op$ is a $\T$-algebra in $\C$. By adjointness and the Yoneda lemma,
$$\llbracket C,L^\op Y^\op\rrbracket = \C(C,L^\op Y^\op-) \cong \C^\op(LY-,C) \cong \Alg{\T}(Y-,SC) \cong SC,$$
$\V$-naturally in $C \in \C$.
\end{proof}

Next we develop variants of the Yoneda lemma and Yoneda embedding involving algebra-valued $\V$-functors.

\begin{defn}
Let $S:\C^\op \to \Alg{\T}$ be a $\V$-functor, and let $E:\T \to \C$ be a $\T$-algebra. We write
$$S[E] \;:\;\T^\op \otimes \T \longrightarrow \V$$
to denote the $\V$-functor given by
$$S[E](J,K) = (SEJ)K \;\;\;\;(J,K \in \T).$$
Thus $S[E]$ is the transpose of the composite $\T^\op \xrightarrow{E^\op} \C^\op \xrightarrow{S} \Alg{\T} \hookrightarrow [\T,\V]$. A \bit{section of $S$ at $E$} is a family $\sigma$ of morphisms
$$\sigma_J:I \to S[E](J,J)\;\;\;\;(J \in \J)$$
in $\V$ that are (extraordinarily) $\V$-natural in $J \in \J$. The \bit{object of sections of $S$ at $E$} is the end $\int S[E] := \int_{J \in \T} S[E](J,J)$ in $\V'$.
\end{defn}

\begin{lem}[\textbf{Yoneda principle for algebra-valued $\V$-functors}]\label{thm:yon-alg-val}\emptybox\newline
There are isomorphisms
$$[\C^\op,\Alg{\T}](\llbracket-,E\rrbracket,S) \cong \int S[E]$$
$\V'$-natural in $E \in \Alg{\T}(\C)$ and $S \in [\C^\op,\Alg{\T}]$. In particular, given a $\T$-algebra $E:\T \to \C$ and a $\V$-functor $S:\C^\op \to \Alg{\T}$, there is a bijective correspondence between $\V$-natural transformations $\phi:\llbracket-,E\rrbracket \Rightarrow S$ and sections $\sigma$ of $S$ at $E$.
\end{lem}
\begin{proof}
By the $\V$-enriched Yoneda lemma,
\begin{eqnarray*}
[\C^\op,\Alg{\T}](\llbracket-,E\rrbracket,S) & = & \int_{C \in \C}\int_{J \in \T} \V(\C(C,EJ),(SC)J)\\
                               & \cong & \int_{J \in \T} [\C^\op,\V](\C(-,EJ),(S-)J)\\
                               & \cong & \int_{J \in \T} (SEJ)J\;.
\end{eqnarray*}
\end{proof}

\begin{rem}
Under the bijection in \ref{thm:yon-alg-val}, each section $\sigma$ corresponds to the transformation $\phi = \tilde{\sigma}$ whose constituent morphisms $\tilde{\sigma}_{CJ}:\C(C,EJ) \to (SC)J$ $(C \in \C,J \in \T)$ correspond by the enriched Yoneda lemma to $\sigma_J:I \to S[E](J,J) = (SEJ)J$ $(J \in \T)$.
\end{rem}

\begin{exa}\label{exa:yon-lemma-for-yon-emb}
Let $E_1,E_2:\T \to \C$ be $\T$-algebras. Then the object of sections of $\llbracket-,E_2\rrbracket:\C^\op \to \Alg{\T}$ at $E_1$ is
$$\int \llbracket-,E_2\rrbracket[E_1] = \int_{J \in \T} \C(E_1J,E_2 J) = [\T,\C](E_1,E_2) = \Alg{\T}(\C)(E_1,E_2).$$
In particular, a section of $\llbracket-,E_2\rrbracket$ at $E_1$ is given by a $\V$-natural transformation $\sigma:E_1 \Rightarrow E_2$, i.e.~a morphism $\sigma:E_1 \to E_2$ in $\Alg{\T}(\C)$. By \ref{thm:yon-alg-val} we have isomorphisms
\begin{equation}\label{eq:yon-emb-isos}\Alg{\T}(\C)(E_1,E_2) \xrightarrow{\sim} [\C^\op,\Alg{\T}](\llbracket-,E_1\rrbracket,\llbracket-,E_2\rrbracket)\;.\end{equation}
\end{exa}

\begin{prop}[\textbf{Yoneda embedding for algebra-valued $\V$-functors}]\label{thm:yon-emb-alg-val}
There is a fully faithful $\V'$-functor $\mathsf{Y}:\Alg{\T}(\C) \to [\C^\op,\Alg{\T}]$ given by $\mathsf{Y} E = \llbracket-,E\rrbracket$.
\end{prop}
\begin{proof}
By definition, $\mathsf{Y}$ is the transpose of the $\V'$-functor $\llbracket-,?\rrbracket:\C^\op \otimes \Alg{\T}(\C) \to \Alg{\T}$ of \eqref{eq:alg-val-hom}. The isomorphisms in \eqref{eq:yon-emb-isos} are the structural morphisms $\mathsf{Y}_{E_1E_2}$.
\end{proof}

\begin{exa}\label{exa:ind-transf-between-homs}
We require the following further application of the Yoneda principle (\ref{thm:yon-alg-val}) in the sequel. Suppose both $\B$ and $\C$ are $\V$-categories with $\J$-powers, and let $P:\C \to \B$ be a $\V$-functor. Given a $\T$-algebra $E:\T \to \C$, the associated $\T$-algebra $PE:\T \to \B$ determines a $\V$-functor $\llbracket-,PE\rrbracket:\B^\op \to \Alg{\T}$.
There is a family of morphisms
$$P_{CE}\;:\;\llbracket C,E\rrbracket \to \llbracket PC,PE\rrbracket$$
$\Alg{\T}$ that are $\V'$-natural in $C \in \C$ and $E \in \Alg{\T}(\C)$, namely the $\V$-natural transformations
$$P_{CE} = P_{C,E-}\;:\;\C(C,E-) \Rightarrow \B(PC,PE-).$$
Given also a $\T$-algebra $E':\T \to \B$, a section $\sigma$ of $\llbracket P-,E'\rrbracket:\C^\op \to \Alg{\T}$ at $E$ is equivalently a $\V$-natural transformation $\sigma:PE \Rightarrow E'$ and corresponds via the Yoneda principle to the composite $\llbracket -,E\rrbracket \overset{P_{-E}}{\Longrightarrow} \llbracket P-,PE \rrbracket \overset{\llbracket P-,\sigma\rrbracket}{\Longrightarrow} \llbracket P-,E'\rrbracket$. In particular, taking $E' = PE$, the identity $1_{PE}:PE \Rightarrow PE$ corresponds to $P_{-E}:\llbracket-,E\rrbracket \Rightarrow \llbracket P-,PE\rrbracket$.
\end{exa}

In the following, we write $\textnormal{RAdj}(\C^\op,\Alg{\T}) \hookrightarrow [\C^\op,\Alg{\T}]$ to denote the full sub-$\V'$-category spanned by the right adjoint $\V$-functors.

\begin{thm}\label{thm:talgs-in-c-equiv-radj}
Let $\T$ be a $\J$-theory, let $\C$ be a $\V$-category with $\J$-powers, and suppose that $\C$ admits limits of $\T$-algebras weighted by $\T$-algebras (\ref{para:limits-of-algs-weighted-by-algs}). Then there is an equivalence of $\V'$-categories
$$\Alg{\T}(\C) \;\simeq\; \textnormal{RAdj}(\C^\op,\Alg{\T})$$
under which each $\T$-algebra $E:\T \to \C$ corresponds to the right adjoint $\V$-functor $\llbracket-,E\rrbracket:\C^\op \to \Alg{\T}$.
\end{thm}
\begin{proof}
This follows immediately from \ref{thm:rep-equiv-ra} and \ref{thm:yon-alg-val}.
\end{proof}

\section{Dual adjunctions between categories of algebras}\label{sec:dualn_adj}

Let $\T$ and $\U$ be $\J$-theories. By \eqref{eq:alg-val-hom} we obtain a $\V$-functor
$$\llbracket-,?\rrbracket\;:\;\Alg{\T}^\op \otimes \Alg{\U}(\Alg{\T}) \longrightarrow \Alg{\U}$$
that sends each $\T$-algebra $A:\T \to \V$ and each $\U$-algebra $E:\U \to \Alg{\T}$ to the $\U$-algebra $\llbracket A,E\rrbracket = \Alg{\T}(A,E-):\U \to \V$. But by \eqref{eq:ba_transposition_isos} we have an isomorphism of $\V$-categories $(-)_{r\ell}:\Alg{(\T,\U)} \xrightarrow{\sim} \Alg{\U}(\Alg{\T})$ that sends each $(\T,\U)$-algebra $D:\T \otimes \U \to \V$ to its transpose $D_{r\ell}:\U \to \Alg{\T}$ (\ref{eq:transpose-of-bif-alg}).

\begin{defn}\label{defn:hom-t}
Given $\J$-theories $\T$ and $\U$, we write
$$\Hom_\T\;:\;\Alg{\T}^\op \otimes\:\Alg{(\T,\U)} \longrightarrow \Alg{\U}$$
to denote the composite $\V$-functor
$$\Alg{\T}^\op \otimes \Alg{(\T,\U)} \xrightarrow[\sim]{1 \otimes (-)_{r\ell}} \Alg{\T}^\op \otimes \Alg{\U}(\Alg{\T}) \xrightarrow{\llbracket-,?\rrbracket} \Alg{\U}\;.$$
Explicitly
$$\Hom_\T(A,D) = \llbracket A,D_{r\ell}\rrbracket = \Alg{\T}(A,D_{r\ell}-)\;\;\;\;(A \in \Alg{\T}, D \in \Alg{(\T,\U)}),$$
or equivalently
$$\bigl(\Hom_\T(A,D)\bigr)K = \Alg{\T}\bigl(A,D(-,K)\bigr)\;\;\;\;(A \in \Alg{\T}, D \in \Alg{(\T,\U)}, K \in \U).$$
\end{defn}

\begin{prop}\label{thm:hom-t-weighted-limit}
Let $D:\T \otimes \U \to \V$ be a $(\T,\U)$-algebra. For each $\T$-algebra $A:\T \to \V$, the $\U$-algebra $\Hom_\T(A,D) = \Alg{\T}(A,D_{r\ell}-):\U \to \V$ is a weighted limit $\{A,D_{\ell r}\}$ in $\Alg{\U}$ for the weight $A$ and the diagram $D_{\ell r}:\T \to \Alg{\U}$. Moreover,
$$\Hom_\T(A,D) = \{A,D_{\ell r}\}$$
$\V$-naturally in $A \in \Alg{\T}$, $D \in \Alg{(\T,\U)}$.
\end{prop}
\begin{proof}
For each $K \in \ob\U$, the object of $\V$-natural transformations $\Alg{\T}(A,D_{r\ell}K) = [\T,\V](A,D_{r\ell} K)$
is a weighted limit $\{A,D_{r\ell}K\}$ of the diagram $D_{r\ell}K = D(-,K) = (D_{\ell r} -)K:\T \to \V$. Moreover, $\Alg{\T}(A,D_{r\ell}K) = \{A,(D_{\ell r}-)K\}$, $\V$-naturally in $A \in \Alg{\T}, D \in \Alg{(\T,\U)}, K \in \U$. Hence the $\V$-functor $\Alg{\T}(A,D_{r\ell}-):\U \to \V$ is a pointwise limit $\{A,D_{\ell r}\}$ and, in particular, a limit in $\Alg{\U}$, and the result follows.
\end{proof}

\begin{cor}
The $\V$-category $\Alg{\U}$ admits limits of $\T$-algebras weighted by $\T$-algebras (\ref{para:limits-of-algs-weighted-by-algs}).
\end{cor}
\begin{proof}
Every $\T$-algebra $E:\T \to \Alg{\U}$ is of the form $D_{\ell r}$ for some $(\T,\U)$-algebra $D$, so this follows from \ref{thm:hom-t-weighted-limit}.
\end{proof}

\begin{para}\label{para:dualn-func}
Let $D$ be a $(\T,\U)$-algebra. By \ref{defn:hom-t} and \ref{thm:hom-t-weighted-limit}, we have a $\V$-functor
$$\Hom_\T(-,D) = \llbracket-,D_{r\ell}\rrbracket = \{-,D_{\ell r}\}\;:\;\Alg{\T}^\op \longrightarrow \Alg{\U}$$
whose composite with $G^\U:\Alg{\U} \to \V$ is precisely $\Alg{\T}(-,D_\ell)$, where $D_\ell = D(-,I):\T \to \V$ is the left face of $D$ (\ref{sec:balg}). For each $\T$-algebra $A$, we call the $\U$-algebra $\Hom_\T(A,D)$ the \bit{dual} of $A$ with respect to $D$. Applying \ref{defn:hom-t} and \ref{thm:hom-t-weighted-limit} to the $(\U,\T)$-algebra $D^\circ$ of \eqref{eq:transpose-of-bif-alg}, we obtain a $\V$-functor
$$\Hom_\U(-,D^\circ) = \llbracket-,D_{\ell r}\rrbracket = \{-,D_{r\ell}\}\;:\;\Alg{\U}^\op \longrightarrow \Alg{\T}\;,$$
whose opposite we write as
$$\Hom_\U(-,D) := \Hom_\U(-,D^\circ)^\op = \{-,D_{r\ell}\}\;:\;\Alg{\U} \longrightarrow \Alg{\T}^\op\;.$$
\end{para}

\begin{thm}\label{thm:dualn_adjn}
Given a $(\T,\U)$-algebra $D$ for $\J$-theories $\T$ and $\U$, there is a \mbox{$\V$-adjunction}
\begin{equation}\label{eq:dualn_adjn}\RAdjn{\Alg{\T}^\op}{\Hom_\T(-,D)}{\Hom_\U(-,D)}{}{}{\Alg{\U}.}\end{equation}
Moreover,
$$\Alg{\T}(A,\Hom_\U(B,D)) \cong \Alg{\U}(B,\Hom_\T(A,D))$$
$\V$-naturally in $A \in \Alg{\T}$, $B \in \Alg{\U}$, and $D \in \Alg{(\T,\U)}$.
\end{thm}
\begin{proof}
This follows immediately from \ref{thm:rep-has-la} and \ref{thm:hom-t-weighted-limit}. 
\end{proof}

\begin{defn}\label{defn:dualn_adjn}
Given a bifold algebra $(\T,\U,D)$, we call the associated $\V$-adjunction \eqref{eq:dualn_adjn} the  \bit{dualization adjunction} induced by $(\T,\U,D)$.
\end{defn}

Right adjoints $\Alg{\T}^\op \to \Alg{\U}$ are equivalently given by $(\T,\U)$-algebras:

\begin{thm}\label{thm:tu-alg-equiv-radj}
There is an equivalence of $\V$-categories
$$\Alg{(\T,\U)} \;\simeq\; \textnormal{RAdj}(\Alg{\T}^\op,\Alg{\U})$$
that assigns to each $(\T,\U)$-algebra $D$ the right adjoint $\V$-functor $\Hom_\T(-,D)$. In particular, $\textnormal{RAdj}(\Alg{\T}^\op,\Alg{\U})$ is a $\V$-category.
\end{thm}
\begin{proof}
By \ref{thm:talgs-in-c-equiv-radj}, there is an equivalence of $\V'$-categories
$$\Alg{\U}(\Alg{\T}) \simeq \textnormal{RAdj}(\Alg{\T}^\op,\Alg{\U})\;,$$
but $\Alg{\U}(\Alg{\T}) \cong \Alg{(\T,\U)}$ is a $\V$-category, by \ref{sec:balg}.
\end{proof}

This entails the following, in view of \ref{para:dualn-func}:

\begin{cor}\label{thm:right-adj-between-cats-algs}
Given a $\V$-functor $S:\Alg{\T}^\op \to \Alg{\U}$, the following are equivalent: (1) $S$ is a right adjoint, (2) $S \cong \Hom_\T(-,D)$ for some $(\T,\U)$-algebra $D$, (3) $S$ is representable (in the sense of \ref{defn:repbl-alg-val-vfunc}).
\end{cor}

\hspace{-0.65ex}We now characterize the $(\T,\U)$-algebra $D$ corresponding to such a right adjoint $S$:

\begin{prop}\label{thm:bifold-alg-corresp-to-dual-adjn}
Let $L \dashv S:\Alg{\T}^\op \to \Alg{\U}$ be a $\V$-adjunction. Given a $(\T,\U)$-algebra $D$, the following are equivalent: (1) $S \cong \Hom_\T(-,D)$, (2) $L \cong \Hom_\U(-,D)$, (3) $S$ is represented by $D_{r\ell}:\U \to \Alg{\T}$, (4) $L^\op$ is represented by $D_{\ell r}:\T \to \Alg{\U}$, (5) $D_{\ell r} \cong SY^\op$ for the $\V$-functor $Y:\T^\op \to \Alg{\T}$ of \ref{para:talgs}, (6) $D_{r\ell} \cong L^\op Y^\op$ for the $\V$-functor $Y:\U^\op \to \Alg{\U}$.
\end{prop}
\begin{proof}
(1) through (4) are equivalent, by \ref{para:dualn-func} and \ref{thm:dualn_adjn}. By \ref{thm:rep-equiv-ra} and \ref{thm:yon-emb-alg-val}, (3) is equivalent to (6). Applying this result to $S^\op \dashv L^\op$ and $D^\circ$, we deduce that (4) is equivalent to (5).
\end{proof}

\section{Algebraic dual adjunctions}\label{sec:alg-dual-adj}

\begin{defn}
A \bit{$\J$-algebraic dual adjunction} is a $\V$-adjunction of the form
\begin{equation}\label{eq:jalg_dual_adjn}\RAdjn{\A^\op}{\Delta}{\nabla}{}{}{\B}\end{equation}
where $\A$ and $\B$ are $\J$-algebraic $\V$-categories over $\V$ \pref{para:jalg}.
\end{defn}

\begin{notn}
We write $\nabla \dashv \Delta:\A^\op \rightarrow \B$ or $(\A,\B,\nabla \dashv \Delta)$ to denote a given $\J$-algebraic dual adjunction as in \eqref{eq:jalg_dual_adjn}.  If $\D = (\A,\B,\nabla \dashv \Delta)$ is a $\J$-algebraic dual adjunction, then by applying $(-)^\op$ we  obtain another $\J$-algebraic dual adjunction $\Delta^\op \dashv \nabla^\op:\B^\op \rightarrow \A$ that we write as
\begin{equation}\label{eq:rev}\D^\circ = (\B,\A,\Delta^\op \dashv \nabla^\op)\;.\end{equation}
We call $\D^\circ$ the \bit{opposite} of $\D$.
\end{notn}

In Section \ref{sec:exa_ada}, we discuss various examples of $\J$-algebraic dual adjunctions. In Theorem \ref{thm:alg_dual_bifold} we show that, up to equivalence, every $\J$-algebraic dual adjunction is of the following form:

\begin{exa}\label{exa:dualn_adjn}
The dualization adjunction induced by a bifold algebra $(\T,\U,D)$ (\ref{defn:dualn_adjn}) is a $\J$-algebraic dual adjunction in which
$$\A = \Alg{\T},\;\;\;\;\B = \Alg{\U},\;\;\;\;\Delta = \Hom_\T(-,D),\;\;\;\;\nabla = \Hom_\U(-,D)\;.$$
We denote the dualization adjunction induced by $(\T,\U,D)$ by
$$\Dualn(\T,\U,D) \;\;=\;\; \bigl(\Alg{\T},\:\Alg{\U},\:\Hom_\U(-,D) \dashv \Hom_\T(-,D)\bigr)$$
or simply by $\Dualn D$.  It is clear from the construction of $\Dualn D$ that
\begin{equation}\label{eq:dualnop}(\Dualn D)^\circ = \Dualn(D^\circ)\;.\end{equation}
\end{exa}

We now define a 2-category $\ADAJ$ whose objects are $\J$-algebraic dual adjunctions.  There are multiple such 2-categories that one could define, but we have chosen the following definition because it provides a suitable notion of equivalence of $\J$-algebraic dual adjunctions \pref{para:equiv_jalg_dual} while simultaneously allowing us establish a biequivalence between $\ADAJ$ and a locally discrete 2-category of bifold algebras over various pairs of $\J$-theories.

\begin{para}\label{para:algcatj_preamble}
In the following discussion, we write $U:\ALGCATJ \rightarrow \VCAT$ to denote the forgetful 2-functor on the 2-category $\ALGCATJ$ of $\J$-algebraic $\V$-categories over $\V$ \pref{para:jalg}, and write $U^*$ to denote the composite 2-functor
$$\ALGCATJ^\co \xrightarrow{U^\co} \VCAT^\co \xrightarrow{(-)^\op} \VCAT\;.$$
We can now form the colax comma 2-category $(U^* \downarrow U)_\textnormal{colax}$ \pref{para:lax_comma}, whose objects are triples $(\A,\B,\Delta)$ consisting of $\J$-algebraic $\V$-categories $\A$ and $\B$ over $\V$ together with a $\V$-functor $\Delta:\A^\op \rightarrow \B$.  Hence every $\J$-algebraic dual adjunction $(\A,\B, \nabla \dashv \Delta)$ has an underlying object $(\A,\B,\Delta)$ of $(U^* \downarrow U)_\textnormal{colax}$.
\end{para}

\begin{defn}\label{defn:dualadjj}
The \bit{2-category of $\J$-algebraic dual adjunctions}, denoted by $\ADAJ$, is the 2-category whose objects are $\J$-algebraic dual adjunctions, with 1-cells, 2-cells, and composition defined as in the 2-category $(U^* \downarrow U)_\textnormal{colax}$ discussed in \ref{para:algcatj_preamble}.   Therefore, a 1-cell
\begin{equation}\label{eq:1cell_algdual_1}P = (P_\ell,\lambda^P,P_r,\rho^P,\phi^P)\;:\;(\A_1,\B_1,\nabla_1 \dashv \Delta_1) \longrightarrow (\A_2,\B_2,\nabla_2 \dashv \Delta_2)\end{equation}
in $\ADAJ$ consists of 1-cells $(P_\ell,\lambda^P):\A_1 \rightarrow \A_2$ and $(P_r,\rho^P):\B_1 \rightarrow \B_2$ in $\ALGCATJ$ together with a 2-cell of the form 
\begin{equation}\label{eq:1cell_algdual}
\xymatrix{
\A_1^\op \ar[r]^{P_\ell^\op} \ar[d]_{\Delta_1} & \A_2^\op \ar[d]^{\Delta_2}\\
\B_1 \ar@{}[ur]|(.4){}="t1"|(.6){}="s1"_{\phi^P} \ar[r]_{P_r} & \B_2
\ar@{<=}"s1";"t1"
}
\end{equation}
in $\VCAT$.  A 2-cell 
\begin{equation}\label{eq:2cell_algdual}\gamma\;:\;P \Longrightarrow Q\;:\;(\A_1,\B_1,\nabla_1 \dashv \Delta_1) \longrightarrow (\A_2,\B_2,\nabla_2 \dashv \Delta_2)\end{equation}
in $\ADAJ$ is a pair $\gamma = (\gamma_\ell,\gamma_r)$ consisting of 2-cells
$$\gamma_\ell:(Q_\ell,\lambda^Q) \Rightarrow (P_\ell,\lambda^P):\A_1 \rightarrow \A_2,\;\;\;\;\gamma_r:(P_r,\rho^P) \Rightarrow (Q_r,\rho^Q):\B_1 \rightarrow \B_2$$
in $\ALGCATJ$ (noting the difference in direction between $\gamma_\ell$ and $\gamma_r$) such that the following pasted 2-cells are equal:
\begin{equation}\label{eqn:pasting_2cell_algdual}
\xymatrix@!0@C=8ex @R=4ex{
\A_1^\op \ar[dd]_{\Delta_1} \ar@/^1pc/[rrr]|{Q_\ell^\op}="s2" \ar@/_1pc/[rrr]|{P_\ell^\op}="t2"  & & & \A_2^\op \ar[dd]^{\Delta_2} \ar@{}[ddll]|(.6){}="t1"|(.8){}="s1" & & \A_1^\op \ar[dd]_{\Delta_1} \ar@/^1pc/[rrr]^{Q_\ell^\op} & & \ar@{}[dl]|(.3){}="t5"|(.7){}="s5" & \A_2^\op \ar[dd]^{\Delta_2}\\
 & & & & & & & & \\
\B_1 \ar@/_1pc/[rrr]|{P_r} & & & \B_2 & & \B_1 \ar@/^1pc/[rrr]|{Q_r}="s4" \ar@/_1pc/[rrr]|{P_r}="t4" & & & \B_2
\ar@{=>}"s1";"t1"_{\phi^P}
\ar@{}"s2";"t2"|(.25){}="t3"|(.75){}="s3"
\ar@{=>}"s3";"t3"_{\gamma_\ell^\op}
\ar@{=>}"s5";"t5"^{\phi^Q}
\ar@{}"s4";"t4"|(.25){}="t6"|(.75){}="s6"
\ar@{=>}"s6";"t6"_{\gamma_r}
}
\end{equation}
\end{defn}

\begin{rem}\label{rem:mates_alg_dualadj}
As formulated in \eqref{eq:1cell_algdual_1} and \eqref{eq:1cell_algdual}, the definition of 1-cells in $\ADAJ$ involves $\Delta_i$ but not $\nabla_i$ $(i = 1,2)$, but this apparent asymmetry is illusory, because the calculus of mates (\cite[\S 2.2, Prop.~2.1]{KeStr:RevEl2Cats}, \cite[I,6.6]{Gray}) entails that 2-cells $\phi^P$ as in \eqref{eq:1cell_algdual} correspond bijectively to 2-cells of the form
\begin{equation}\label{eq:mate_2cell}
\xymatrix{
\B_1 \ar[r]^{P_r} \ar[d]_{\nabla_1} & \B_2 \ar[d]^{\nabla_2}\\
\A_1^\op \ar@{}[ur]|(.4){}="t1"|(.6){}="s1"_{} \ar[r]_{P_\ell^\op} & \A_2^\op
\ar@{=>}"s1";"t1"
}
\end{equation}
in $\VCAT$, and the latter 2-cell is called the \textit{mate} of $\phi^P$. Therefore, 2-cells $\phi^P$ as in \eqref{eq:1cell_algdual} are also in bijective correspondence with 2-cells of the form
\begin{equation}\label{eq:mate_1cell_algdual}
\xymatrix{
\B_1^\op \ar[r]^{P_r^\op} \ar[d]_{\nabla_1^\op} & \B_2^\op \ar[d]^{\nabla_2^\op}\\
\A_1 \ar@{}[ur]|(.4){}="t1"|(.6){}="s1"_{\bar{\phi}^P} \ar[r]_{P_\ell} & \A_2
\ar@{<=}"s1";"t1"
}
\end{equation}
in $\VCAT$.

This shows that for any pair of $\J$-algebraic dual adjunctions $\D_1$ and $\D_2$, each 1-cell $P = (P_\ell,\lambda^P,P_r,\rho^P,\phi^P):\D_1 \rightarrow \D_2$ in $\ADAJ$ determines an associated 1-cell
\begin{equation}\label{eq:rev_1cell}P^\circ = (P_r,\rho^P,P_\ell,\lambda^P,\bar{\phi}^P)\;:\;\D_1^\circ \longrightarrow \D_2^\circ\;\;\;\;\;\;\text{in $\ADAJ$,}\end{equation}
where we employ the notation of \eqref{eq:rev}.  This describes a bijection between 1-cells $\D_1 \rightarrow \D_2$ and 1-cells $\D_1^\circ \rightarrow \D_2^\circ$ in $\ADAJ$, and we obtain the following:
\end{rem}

\begin{prop}\label{thm:rev_iso_2cats}
There is an isomorphism of 2-categories
\begin{equation}\label{eq:rev_2func}(-)^\circ\;:\;\ADAJ^\co \overset{\sim}\longrightarrow \ADAJ\end{equation}
given on objects by \eqref{eq:rev} and on 1-cells by \eqref{eq:rev_1cell}.
\end{prop}
\begin{proof}
Given any 2-cell $\gamma = (\gamma_\ell,\gamma_r):P \Rightarrow Q$ in $\ADAJ$, with the notation of \eqref{eq:2cell_algdual}, we claim that $\gamma^\circ := (\gamma_r,\gamma_\ell):Q^\circ \Rightarrow P^\circ$ is a 2-cell in $\ADAJ$.  To show this, it suffices to verify that the pasting equation \eqref{eqn:pasting_2cell_algdual} for $\gamma$ entails the analogous pasting equation for $\gamma^\circ$, but this follows from \cite[(6.14), p.~145]{Gray}, or from \cite[Prop.~2.2]{KeStr:RevEl2Cats}.  The 2-functoriality of $(-)^\circ$ now follows from \cite[Prop.~2.2]{KeStr:RevEl2Cats}.  If we write $S = (-)^\circ$, then $S$ and $S^\co$ are inverses of one another.
\end{proof}

\begin{para}[\textbf{Equivalence of $\J$-algebraic dual adjunctions}]\label{para:equiv_jalg_dual}
We say that $\J$-algebraic dual adjunctions $\D_1 = (\A_1,\B_1,\nabla_1 \dashv \Delta_1)$ and $\D_2 = (\A_2,\B_2,\nabla_2 \dashv \Delta_2)$ are \bit{equivalent} if $\D_1 \simeq \D_2$ in $\ADAJ$.  Supposing that this is the case, there is then an adjoint equivalence $P \dashv Q:\D_2 \rightarrow \D_1$ in $\ADAJ$, and we can write $P = (P_\ell,\lambda^P,P_r,\rho^P,\phi^P)$ and $Q = (Q_\ell,\lambda^Q,Q_r,\rho^Q,\phi^Q)$ in the notation of \eqref{eq:1cell_algdual_1}.  In particular, we have adjoint equivalences
$$(Q_\ell,\lambda^Q) \dashv (P_\ell,\lambda^P):\A_1 \rightarrow \A_2\;,\;\;\;\;(P_r,\rho^P) \dashv (Q_r,\rho^Q):\B_2 \rightarrow \B_1$$
in the pseudo-slice $\VCAT\sslash\V$.  Furthermore, $\phi^P$ and $\phi^Q$ are $\V$-natural transformations of the form
$$
\xymatrix{
\A_1^\op \ar[r]^{P_\ell^\op} \ar[d]_{\Delta_1} & \A_2^\op \ar[d]^{\Delta_2} & & \A_2^\op \ar[d]_{\Delta_2} \ar[r]^{Q_\ell^\op} & \A_1^\op \ar[d]^{\Delta_1}\\
\B_1 \ar@{}[ur]|(.4){}="t1"|(.6){}="s1"_{} \ar[r]_{P_r} & \B_2 & & \B_2 \ar@{}[ur]|(.4){}="t2"|(.6){}="s2"_{} \ar[r]_{Q_r} & \B_1
\ar@{<=}"s1";"t1"^{\phi^P}
\ar@{<=}"s2";"t2"^{\phi^Q}
}
$$
and $\phi^P$ and $\phi^Q$ are invertible by \ref{para:lax_comma}.  By \ref{rem:mates_alg_dualadj} and \ref{thm:rev_iso_2cats}, we also obtain invertible $\V$-natural transformations of the following forms:
$$
\xymatrix{
\B_1^\op \ar[r]^{P_r^\op} \ar[d]_{\nabla_1^\op} & \B_2^\op \ar[d]^{\nabla_2^\op} & & \B_2^\op \ar[d]_{\nabla_2^\op} \ar[r]^{Q_r^\op} & \B_1^\op \ar[d]^{\nabla_1^\op}\\
\A_1 \ar@{}[ur]|(.4){}="t1"|(.6){}="s1"_{} \ar[r]_{P_\ell} & \A_2 & & \A_2 \ar@{}[ur]|(.4){}="t2"|(.6){}="s2"_{} \ar[r]_{Q_\ell} & \A_1
\ar@{<=}"s1";"t1"^{\bar{\phi}^P}
\ar@{<=}"s2";"t2"^{\bar{\phi}^Q}
}
$$
\end{para}

\begin{thm}\label{thm:alg_dual_bifold}
Every $\J$-algebraic dual adjunction $\D$ is equivalent to the dualization \mbox{adjunction} $\Hom_\U(-,D) \dashv \Hom_\T(-,D):\Alg{\T}^\op \rightarrow \Alg{\U}$ induced by a bifold algebra $D$, in the sense that $\D \simeq \Dualn(D)$ in the 2-category $\ADAJ$.
\end{thm}
\begin{proof} 
Let $\D = (\A,\B,\nabla \dashv \Delta)$ be a $\J$-algebraic dual adjunction.  Since $\A$ and $\B$ are $\J$-algebraic $\V$-categories over $\V$, there are equivalences $\A \simeq \Alg{\T}$ and $\B \simeq \Alg{\U}$ in $\ALGCATJ$, for $\J$-theories $\T$ and $\U$.  Hence, there is an adjoint equivalence
$$(P_r,\rho^P) \ladjto{\eta_r}{\varepsilon_r} (Q_r,\rho^Q):\Alg{\U} \rightarrow \B\;\;\;\;\textnormal{in $\ALGCATJ$}\;,$$
and since $\A \simeq \Alg{\T}$ in $\ALGCATJ^\co$ there is an adjoint equivalence
$$(P_\ell,\lambda^P) \ladjto{\eta_\ell}{\varepsilon_\ell} (Q_\ell,\lambda^Q):\Alg{\T} \rightarrow \A\;\;\;\;\textnormal{in $\ALGCATJ^\co$}\;.$$
In particular, we obtain adjoint equivalences
$$P_\ell^\op \ladjto{\eta_\ell^\op}{\varepsilon_\ell^\op} Q_\ell^\op:\Alg{\T}^\op \rightarrow \A^\op\;,\;\;\;\;\;\;P_r \ladjto{\eta_r}{\varepsilon_r} Q_r:\Alg{\U} \rightarrow \B$$
in $\VCAT$.  Let $\Delta'$ denote the composite $\V$-functor
$$\Alg{\T}^\op \overset{Q_\ell^\op}{\longrightarrow} \A^\op \overset{\Delta}{\longrightarrow} \B \overset{P_r}{\longrightarrow} \Alg{\U}\;.$$
Since $\Delta$ is a right adjoint and both $Q_\ell^\op$ and $P_r$ are equivalences, it follows that $\Delta'$ has a left adjoint $\nabla'$.  Thus we obtain a $\J$-algebraic dual adjunction
$$\D' = (\Alg{\T},\Alg{\U},\nabla' \dashv \Delta')\;.$$
Also, we obtain invertible 2-cells of the form
$$
\xymatrix{
\A^\op \ar[d]_{\Delta} \ar[r]^(.45){P_\ell^\op} & \Alg{\T}^\op \ar@{}[dl]|(.4){}="t1"|(.6){}="s1" \ar[d]^{\Delta'} \ar[r]^(.6){Q_\ell^\op} & \A^\op \ar@{}[dl]|(.4){}="t2"|(.6){}="s2" \ar[d]^{\Delta}\\
\B \ar[r]_(.4){P_r} & \Alg{\U} \ar[r]_(.6){Q_r} & \B
\ar@{=>}"s1";"t1"_{\phi^P}
\ar@{=>}"s2";"t2"_{\phi^Q}
}
$$
by taking
$$\phi^P = P_r\Delta\eta_\ell^\op:P_r\Delta \Rightarrow P_r\Delta Q_\ell^\op P_\ell^\op = \Delta' P_\ell^\op$$
$$\phi^Q = \eta_r^{-1} \Delta Q_\ell^\op:Q_r\Delta' = Q_rP_r\Delta Q_\ell^\op \Rightarrow \Delta Q_\ell^\op\;.$$
Thus we have constructed 1-cells
$$P = (P_\ell,\lambda^P,P_r,\rho^P,\phi^P):\D \rightarrow \D',$$
$$Q = (Q_\ell,\lambda^Q,Q_r,\rho^Q,\phi^Q):\D' \rightarrow \D$$
in $\ADAJ$, and it is straightforward to verify that $\eta = (\eta_\ell,\eta_r):1_\D \Rightarrow QP$ and $\varepsilon = (\varepsilon_\ell,\varepsilon_r):PQ \Rightarrow 1_{\D'}$ are invertible 2-cells in $\ADAJ$.  This proves that $\D \simeq \D'$ in $\ADAJ$. By \ref{thm:tu-alg-equiv-radj} there is some $(\T,\U)$-algebra $D$ with an invertible 1-cell $\phi:\Hom_\T(-,D) \Rightarrow \Delta'$ in $\VCAT$.  Employing the notation of \eqref{eq:1cell_algdual_1}, we obtain 1-cells
$$P = (1,1,1,1,\phi):\Dualn(D) \rightarrow \D',\;\;\;\;Q = (1,1,1,1,\phi^{-1}):\D' \rightarrow \Dualn(D)$$
in $\ADAJ$, each of whose first four components are identities, and we find that $P$ and $Q$ are inverses of one another. Hence $\D \simeq \D' \cong \Dualn(D)$ in $\ADAJ$.
\end{proof}

We have thus shown that, up to equivalence, $\J$-algebraic dual adjunctions are the dualization adjunctions induced by bifold algebras.  Because of this, we proceed to define some terminology for $\J$-algebraic dual adjunctions that is inspired by the situation of dualization with respect to a bifold algebra:

\begin{defn}\label{defn:alg_dual_terms}
Let $\D = (\A,\B,\nabla \dashv \Delta)$ be a $\J$-algebraic dual adjunction, so that $\nabla \dashv \Delta:\A^\op \rightarrow \B$.  Given an object $A$ of $\A$, we call $\Delta A$ the \bit{dual} of $A$, and we denote it by
$$A^\Delta = \Delta A\;.$$
Of course, this terminology depends implicitly on the the given dual adjunction $\nabla \dashv \Delta$, and by interpreting this same definition instead with respect to the opposite adjunction $\D^\circ = (\Delta^\op \dashv \nabla^\op:\B^\op \rightarrow \A)$ of \eqref{eq:rev}, we also arrive at the following terminology:  For each object $B$ of $\B$, we call $B^\nabla = \nabla B$ the \bit{dual} of $B$.  We call
$$A^{\Delta\nabla} = \nabla\Delta A\;\;\;\;\;\text{and}\;\;\;\;B^{\nabla\Delta} = \Delta\nabla B$$
the \bit{double duals} of $A$ and of $B$, respectively, and we call the induced $\V$-monads
$$\nabla^\op\Delta^\op:\A \rightarrow \A\;\;\;\;\text{and}\;\;\;\;\Delta\nabla:\B \rightarrow \B$$
the \bit{double dualization monads} on $\A$ and on $\B$, respectively\footnote{This can be contrasted with Kock's \cite{Kock:DblDln} use of the same term to denote the $\V$-monad $\V(\V(-,D),D)$ on $\V$ determined by an object $D$ of $\V$, which can be regarded as a special case of \ref{defn:alg_dual_terms}.}.  We write the unit and counit of $\D$ as $d^{\D}:1_{\B} \Rightarrow \Delta\nabla$ and $e^{\D}:\nabla\Delta \Rightarrow 1_{\A^\op}$.  Applying $(-)^\op$ sends $e^{\D}$ to the unit $d^{\D^\circ}$ of $\D^\circ$ and vice-versa.

We say that an object $A$ of $\A$ is \bit{reflexive} (or \bit{$\nabla\Delta$-reflexive}) if 
$$d^{\D^\circ}_A\;:\;A \longrightarrow \nabla\Delta A = A^{\Delta\nabla}$$
is an isomorphism in $\A$.  Similarly, $B \in \ob\B$ is \bit{reflexive} (or \bit{$\Delta\nabla$-reflexive}) if $d^{\D}_B:B \rightarrow \Delta\nabla B$ is an isomorphism in $\B$.
\end{defn}

We discuss various mathematical examples of these notions of enriched algebraic dualization and reflexivity in Section \ref{sec:exa_ada}.

\begin{exa}
In view of Theorem \ref{thm:alg_dual_bifold}, there is essentially no loss of generality in considering the case where $\nabla \dashv \Delta:\A^\op \rightarrow \B$ is the dualization adjunction induced by a bifold algebra $(\T,\U,D)$.  Given a $\T$-algebra $A$ and an $\U$-algebra $B$, $A^\Delta = \Hom_\T(A,D)$, $B^\nabla = \Hom_\U(B,D)$, and the double dualization monads, $\nabla^\op\Delta^\op$ on $\Alg{\T}$ and $\Delta\nabla$ on $\Alg{\U}$ are given by
$\nabla\Delta A = \Hom_\U(\Hom_\T(A,D),D)$ and $\Delta\nabla B  = \Hom_\T(\Hom_\U(B,D),D)$.
\end{exa}

\begin{defn}
Let $\D = (\A,\B,\nabla \dashv \Delta)$ be a $\J$-algebraic dual adjunction, and write $F_\ell \dashv G_\ell:\A \to \V$ and $F_r \dashv G_r:\B \to \V$ for the associated $\V$-adjunctions (\ref{para:jalg}).  The \bit{left dualizer} of $\D$ is the object $\nabla F_r I$ of $\A$, and the \bit{right dualizer} of $\D$ is the object $\Delta F_\ell I$ of $\B$ (where $I$ is the unit object of $\V$).
\end{defn}

\begin{prop}\label{thm:dualizers-represent}
Let $\D = (\A,\B,\nabla \dashv \Delta)$ be a $\J$-algebraic dual adjunction. Then the left dualizer $\nabla F_r I$ represents the $\V$-functor $G_r\Delta:\A^\op \to \V$, and the right dualizer $\Delta F_\ell I$ represents the $\V$-functor $G_\ell\nabla^\op:\B^\op \to \V$.
\end{prop}
\begin{proof}
By adjointness, $\A(A,\nabla F_r I) \cong \B(F_r I,\Delta A) \cong \V(I,G_r \Delta A) \cong G_r \Delta A$, \linebreak $\V$-naturally in $A \in \A$, and similarly for the right dualizer.
\end{proof}

\begin{prop}\label{thm:rl_dualizer}
Let $D$ be a bifold algebra. Then the left dualizer of the dualization adjunction $\Dualn(D)$ is isomorphic to the left face $D_\ell$ of $D$, and the right dualizer of $\Dualn(D)$ is isomorphic to the right face $D_r$ of $D$.
\end{prop}
\begin{proof}
Write $D = (\T,\U,D)$. By \ref{para:dualn-func}, the object $D_\ell$ of $\Alg{\T}$ represents the composite $\V$-functor $G_r \Hom_\T(-,D) = \Alg{\T}(-,D_\ell):\Alg{\T}^\op \to \V$ and so, by \ref{thm:dualizers-represent}, is isomorphic to the left dualizer of $\Dualn(D)$. The claim regarding $D_r$ is proved similarly.
\end{proof}

\section{Global functoriality of the algebra-valued hom-functor}\label{sec:global-func}

We now establish a further way in which $\Hom_\T(A,D)$ \pref{defn:hom-t} is functorial. In addition to being of interest in its own right, this result will be an essential tool in later establishing a biequivalence between bifold algebras and algebraic dual adjunctions.  The resulting functoriality of $\Hom_\T(A,D)$ will enable us to vary the $\T$-algebra $A$ and the $(\T,\U)$-algebra $D$ in such a way that the $\J$-theories $\T$ and $\U$ vary as well.  First we employ the Grothendieck construction to define the categories over which this variation takes place:

\begin{para}\label{defn:algs}
We write $\AlgO$ to denote the op-fibred category over $\ThJ^\op$ determined by the functor $\pAlg{(-)}:\ThJ^\op \rightarrow \CAT$ of \eqref{eq:func_talg}.  Hence,
$$\AlgO = \Opfib\bigl(\ThJ^\op,\pAlg{(-)}\bigr)\;,$$
and we call $\AlgO$ the \bit{(op-fibred) category of algebras}.  The objects of $\AlgO$ are algebras $(\T,A)$ in the sense of \ref{para:jalg}, written also simply as $A$.  An \bit{op-theoretic morphism of algebras} is by definition a morphism
$$(M,f)\;:\;(\T_1,A_1) \rightarrow (\T_2,A_2)\;\;\;\;\text{in $\AlgO$}$$
and so consists of a morphism of $\J$-theories $M:\T_2 \rightarrow \T_1$ and a $\T_2$-homomorphism $f:M^*A_1 \rightarrow A_2$. Later, in Section \ref{sec:background-II}, we also discuss other categories of algebras over various theories \pref{para:cats-algs}, but here we employ $\AlgO$ and its fibrewise opposite (\ref{para:alg-star}) as they enable our result \ref{thm:descn_hom} on the global functoriality of $\Hom_\T(A,D)$ to exhibit the expected variances.
\end{para}

\begin{para}\label{para:alg-star}
We write $\AlgO^*$ to denote the \textit{fibrewise opposite} of $\AlgO$ \pref{para:fibr_op}, i.e.~the op-fibred category over $\ThJ^\op$ determined by the functor $\ThJ^\op \to\CAT$ given by $\T \mapsto \pAlg{\T}^\op$. An object of $\AlgO^*$ is an algebra $(\T,A)$, while a morphism $(M,f):(\T_1,A_1) \rightarrow (\T_2,A_2)$ in $\AlgO^*$ consists of a morphism of $\J$-theories $M:\T_2 \rightarrow \T_1$ and a \mbox{$\T_2$-homomorphism} $f:A_2 \rightarrow M^* A_1$. (By \ref{para:fibr_op}, $\AlgO^*$ is precisely the opposite of the fibred category over $\ThJ$ determined by $\pAlg{(-)}$, which we encounter later in \ref{para:cats-algs}.)
\end{para}

\begin{para}\label{defn:cat_bifold_alg}\emptybox
Writing $\ThJ^2 = \ThJ \times \ThJ$, we define functors
\begin{equation}\label{eq:func_tsalg}\Alg{(-,\blanktwo)}:(\ThJ^2)^\op \longrightarrow \VCAT,\;\;\;\;\pAlg{(-,?)}:(\ThJ^2)^\op \longrightarrow \CAT\end{equation}
that send each pair of $\J$-theories $(\T,\U)$ to the $\V$-category $\Alg{(\T,\U)}$ (\ref{sec:balg}) and its underlying ordinary category $\pAlg{(\T,\U)}$, respectively. The first functor sends each morphism $M = (M_\ell,M_r):(\T_2,\U_2) \rightarrow (\T_1,\U_1)$ in $\ThJ^2$ to the $\V$-functor $$M^*:\Alg{(\T_1,\U_1)} \longrightarrow \Alg{(\T_2,\U_2)}$$
given by $M^*(D) = D(M_\ell-,M_r\blanktwo)$ $(D \in \Alg{(\T_1,\U_1)})$, and the second sends $M$ to the ordinary functor underlying $M^*$. We write
$$\BAlgO = \Opfib\bigl((\ThJ^2)^\op,\pAlg{(-,\blanktwo)}\bigr)$$
to denote the op-fibred category over $(\ThJ^2)^\op$ determined by the functor $\pAlg{(-,\blanktwo)}$.  The objects of $\BAlgO$ may be identified with bifold algebras $(\T,\U,D)$, so we call $\BAlgO$ the \bit{(op-fibred) category of bifold algebras}.  An \bit{op-theoretic morphism of bifold algebras} is by definition a morphism
$$(M,g)\;:\;(\T_1,\U_1,D_1) \rightarrow (\T_2,\U_2,D_2)\;\;\;\;\text{in $\BAlgO$}$$
and so consists of a morphism $M = (M_\ell,M_r):(\T_2,\U_2) \rightarrow (\T_1,\U_1)$ in $\ThJ^2$ and a morphism $g:M^*D_1 \rightarrow D_2$ in $\pAlg{(\T_2,\U_2)}$.
\end{para}

\begin{para}\label{para:pb_algstar_bifoldalg}
In the following theorem, we write $P:\AlgO^* \rightarrow \ThJ^\op$ to denote the projection functor, and we write $L,R:\BAlgO \rightarrow \ThJ^\op$ to denote the first and second components of the projection functor $\BAlgO \rightarrow (\ThJ^2)^\op = \ThJ^\op \times \ThJ^\op$, so that the latter is the induced functor $(L,R)$.  Also, we write
$$\AlgO^* \tensor[_P]{\times}{_L} \BAlgO$$
to denote the pullback in $\CAT$ of the cospan consisting of $P$ and $L$.  The objects of $\AlgO^* \tensor[_P]{\times}{_L} \BAlgO$ are therefore pairs $((\T,A), (\T,\U,D))$ consisting of an algebra $(\T,A)$ and a bifold algebra $(\T,\U,D)$, and we may write these objects more briefly as $(\T,\U,A,D)$ or simply as $(A,D)$. Our next objective is as follows:
\end{para}

\begin{obj}\label{obj:global_func_alg_valued_hom}
Define a functor
$$\sfHom\;\;:\;\;\AlgO^* \tensor[_P]{\times}{_L} \BAlgO \;\;\longrightarrow\;\; \AlgO$$
given on objects by
$$\sfHom(A,D) = \Hom_\T(A,D)\;\;\;\;\;\;\;(A \in \ob\pAlg{\T},\;\;D \in \ob\:\pAlg{(\T,\U)}).$$
More explicitly, $\sfHom$ must send each object $(A,D) = ((\T,A),(\T,\U,D))$ of \newline $\AlgO^* \tensor[_P]{\times}{_L} \BAlgO$ to the object $\sfHom(A,D) = (\U,\Hom_\T(A,D))$ of $\AlgO$.
\end{obj}

\begin{para}\label{para:discn_pb_algstar_bifoldalg}
In order to define the functor $\sfHom$, we first examine its domain in more detail.  By definition, a morphism 
\begin{equation}\label{eq:mor_pb}((N,f),(M,g)):((\T_1,A_1),(\T_1,\U_1,D_1)) \rightarrow ((\T_2,A_2),(\T_2,\U_2,D_2))\end{equation}
in the pullback category $\AlgO^* \tensor[_P]{\times}{_L} \BAlgO$ is a pair consisting of morphisms $(N,f):(\T_1,A_1) \rightarrow (\T_2,A_2)$ in $\AlgO^*$ and $(M,g):(\T_1,\U_1,D_1) \rightarrow (\T_2,\U_2,D_2)$ in $\BAlgO$ with the property that $N = M_\ell$ when we write $M = (M_\ell,M_r)$.  Hence, up to a bijection, such a morphism can be written equivalently as a triple
\begin{equation}\label{eq:mor_pb2}(M,f,g):(\T_1,\U_1,A_1,D_1) \rightarrow (\T_2,\U_2,A_2,D_2)\;.\end{equation}
with $M:(\T_2,\U_2) \rightarrow (\T_1,\U_1)$ in $\ThJ^2$, $f:A_2 \to M_\ell^*A_1$ in $\pAlg{\T_2}$, and $g:M^* D_1 \rightarrow D_2$ in $\pAlg{(\T_2,\U_2)}$.  In fact, the correspondence between morphisms of the two forms \eqref{eq:mor_pb} and \eqref{eq:mor_pb2} leads to the following observation:
\end{para}

\begin{prop}\label{thm:pb_cat_as_gr_constr}
The category $\AlgO^* \tensor[_P]{\times}{_L} \BAlgO$ is isomorphic to the op-fibred category over $(\ThJ^2)^\op$ determined by the functor
$$\pAlg{(-)}^\op \times \pAlg{(-,\blanktwo)}\;\;:\;\;(\ThJ^2)^\op = \ThJ^\op \times \ThJ^\op \rightarrow \CAT$$
given on objects by  $(\T,\U) \mapsto \pAlg{\T}^\op \times \pAlg{(\T,\U)}$ and on morphisms by $M = (M_\ell,M_r) \mapsto (M_\ell^*)^\op \times M^*$.
\end{prop}

We now identify these two isomorphic categories, regarding \eqref{eq:mor_pb} and \eqref{eq:mor_pb2} as two notations for the same morphism of $\AlgO^* \tensor[_P]{\times}{_L} \BAlgO$.  Objective \ref{obj:global_func_alg_valued_hom} now amounts to defining a functor between the split op-fibred categories determined by a specific pair of $\CAT$-valued functors, and to accomplish this we require a result of Kelly 
\cite[\S 6]{Ke:ClDo}, which we now recall:

\begin{para}[\textbf{Kelly's 2-functoriality of the Grothendieck construction}]\label{thm:kelly_functor_from_lax_transf}
If $\Phi,\Psi:\C \rightarrow \K$ are functors, where $\C$ is a category and $\K$ is a 2-category, then a \bit{lax natural transformation} \cite[\S 2.2]{Ke:ClDo} $\Gamma:\Phi \wavy \Psi$ is given by a family of 1-cells $\Gamma_C:\Phi C \rightarrow \Psi C$ $(C \in \ob\C)$ in $\K$ together with an assignment to each morphism $c:C \rightarrow C'$ in $\C$ a 2-cell
$$
\xymatrix{
\Phi C \ar[r]^{\Gamma_C} \ar[d]_{\Phi c} & \Psi C \ar@{}[dl]|(0.3){}="s1"|(.7){}="t1" \ar[d]^{\Psi c}\\
\Phi C' \ar[r]_{\Gamma_{C'}} & \Psi C'
\ar@{=>}"s1";"t1"_{\Gamma_c}
}
$$
in $\K$ such that (1) $\Gamma_{1_C} = 1_{\Gamma_C}$ for each object $C$ of $\C$, and (2) for each pair of composable morphisms $a:C \rightarrow C'$ and $b:C' \rightarrow C''$ in $\C$, the 2-cell $\Gamma_{b \cdot a}$ is equal to the 2-cell obtained by pasting $\Gamma_{a}$ and $\Gamma_{b}$.

By \cite[\S 3]{Ke:ClDo}, there is a 2-category $\CAT \int \CAT$ whose objects are pairs $(\C,\Phi)$ consisting of a category $\C$ and a functor $\Phi:\C \rightarrow \CAT$, and whose 1-cells $(T,\Gamma):(\C,\Phi) \rightarrow (\D,\Theta)$ are pairs consisting of a functor $T:\C \rightarrow \D$ and a lax natural transformation $\Gamma:\Phi \wavy \Theta T$.  Kelly showed in \cite[\S 6]{Ke:ClDo} that there is a 2-functor 
$$\Opfib\;:\;\textstyle{\CAT \int \CAT} \rightarrow \CAT$$
that sends each object $(\C,\Phi)$ to the op-fibred category $\Opfib(\C,\Phi)$ over $\C$ determined by $\Phi:\C \rightarrow \CAT$.  This 2-functor sends each 1-cell $(T,\Gamma):(\C,\Phi) \rightarrow (\D,\Theta)$ to a functor
\begin{equation}\label{eq:func_from_lax_transf}\Opfib(T,\Gamma)\;:\;\Opfib(\C,\Phi) \rightarrow \Opfib(\D,\Theta)\end{equation}
that is given on objects by $(C \in \ob\C,F \in \ob \Phi C) \mapsto (TC \in \ob\D,\Gamma_C(F) \in \ob \Theta TC)$. Given a morphism $(c,f):(C,F) \rightarrow (C',F')$ in $\Opfib(\C,\Phi)$, where $c:C \to C'$ in $\C$ and $f:(\Phi c)F \to F'$ in $\Phi C'$, the functor $\Opfib(T,\Gamma)$ sends $(c,f)$ to the morphism $(TC,\Gamma_C(F)) \to (TC',\Gamma_{C'}(F'))$ in $\Opfib(\D,\Theta)$ consisting of $Tc:TC \rightarrow TC'$ and the composite
$$(\Theta Tc)(\Gamma_C(F)) \xrightarrow{(\Gamma_c)_F} \Gamma_{C'}(\Phi c)(F) \xrightarrow{\Gamma_{C'}(f)} \Gamma_{C'}(F')$$
in $\Theta TC'$.
\end{para}

The latter result of Kelly reduces Objective \ref{obj:global_func_alg_valued_hom} to the task of defining a suitable lax natural transformation:

\begin{para}[\textbf{The lax natural transformation $\mathsf{hom}$}]\label{para:lax_transf_hom}
For each pair of $\J$-theories $(\T,\U)$, let us write
\begin{equation}\label{eq:functors_hom}\mathsf{hom}_{(\T,\U)}\;:\;\pAlg{\T}^\op \times \pAlg{(\T,\U)} \longrightarrow \pAlg{\U}\end{equation}
to denote the (ordinary) bifunctor determined by the $\V$-functor $\Hom_\T$ of \ref{defn:hom-t}.  We now show that these functors $\mathsf{hom}_{(\T,\U)}$ are the components of a lax natural transformation $\mathsf{hom}$ of the following form:
\begin{equation}\label{eq:lax_transf_hom}
\xymatrix{
*[r]{(\ThJ^2)^\op \;=} & \ThJ^\op \times \ThJ^\op \ar[rr]^{\pi_2} \ar[dr]_{\hspace{-10ex}\pAlg{(-)}^\op \times \pAlg{(-,\blanktwo)}} & & \ThJ^\op \ar@{}[dll]|(.35){}="t1"|(.5){}="s1" \ar[dl]^{\pAlg{(-)}}\\
 & & \CAT & 
\ar@{~>}"s1";"t1"^{\mathsf{hom}}
}
\end{equation}
In order to define such a transformation $\mathsf{hom}$, we must also specify for each morphism $M = (M_\ell,M_r):(\T_1,\U_1) \rightarrow (\T_2,\U_2)$ in $(\ThJ^2)^\op$ a natural transformation
\begin{equation}\label{eq:hom_m}
\xymatrix{
\pAlg{\T_1}^\op \times \pAlg{(\T_1,\U_1)} \ar[rr]^(.65){\mathsf{hom}_{(\T_1,\U_1)}} \ar[d]_{(M_\ell^*)^\op \times M^*} & & \pAlg{\U_1} \ar@{}[dll]|(0.3){}="s1"|(.7){}="t1" \ar[d]^{M_r^*}\\
\pAlg{\T_2}^\op \times \pAlg{(\T_2,\U_2)} \ar[rr]_(.65){\mathsf{hom}_{(\T_2,\U_2)}} & & \pAlg{\U_2}
\ar@{=>}"s1";"t1"|{\mathsf{hom}_M}
}
\end{equation}
which we now define. We go slightly further and define a family of morphisms
\begin{equation}\label{eq:hom_m_ad}(\mathsf{hom}_M)_{AD} \;:\; M_r^*\Hom_{\T_1}(A,D) \rightarrow \Hom_{\T_2}(M_\ell^* A,M^* D)\;\;\;\;\textnormal{in $\Alg{\U_2}$}\end{equation}
that is \textit{$\V$-natural} in $A \in \Alg{\T_1}$ and $D \in \Alg{(\T_1,\U_1)}$. By \ref{para:dualn-func} and \eqref{eq:hom-nat}, \begin{equation}\label{eq:m-star-hom}M_r^*\Hom_{\T_1}(-,D) = M_r^*\llbracket-,D_{r\ell}\rrbracket = \llbracket-,D_{r\ell}M_r\rrbracket\;:\;\Alg{\T_1}^\op \to \Alg{\U_2}\;,\end{equation}
while
$$\Hom_{\T_2}(M_\ell^*-,M^* D) = \llbracket M_\ell^*-,M_\ell^* D_{r\ell}M_r\rrbracket\;:\;\Alg{\T_1}^\op \to \Alg{\U_2}$$
because $(M^*D)_{r\ell}:\U_2 \to \Alg{\T_2}$ is the composite $M_\ell^* D_{r\ell}M_r$. By \ref{exa:ind-transf-between-homs}, the $\V$-functor $M_\ell^*:\Alg{\T_1} \to \Alg{\T_2}$ induces morphisms
\begin{equation}\label{eq:defn-hom-lax-transf}(\mathsf{hom}_M)_{AD} := (M^*_\ell)_{A,D_{r\ell}M_r}:\llbracket A,D_{r\ell}M_r\rrbracket \to \llbracket M^*_\ell A,M^*_\ell D_{r\ell}M_r\rrbracket\end{equation}
in $\Alg{\U_2}$ that are $\V$-natural in $A \in \Alg{\T_1}, D \in \Alg{(\T_1,\U_1)}$ and correspond via the Yoneda principle (\ref{thm:yon-alg-val}) to the identities $1_{M^*_\ell D_{r\ell}M_r}$. Indeed, in this case the Yoneda principle provides a bijection between $\V$-natural transformations $$\llbracket -,D_{r\ell}M_r\rrbracket \Rightarrow \llbracket M^*_\ell -,M^*_\ell D_{r\ell}M_r\rrbracket$$
and $\V$-natural transformations $\sigma:M_\ell^*D_{r\ell}M_r \Rightarrow M_\ell^* D_{r\ell}M_r:\U_2 \to \Alg{\T_2}$, equivalently, morphisms $\sigma:M^*D \to M^*D$ in $\pAlg{(\T_2,\U_2)}$, under which $\sigma = 1_{M^* D}$ corresponds to $(\mathsf{hom}_M)_{-D}$. Concretely, $(\mathsf{hom}_M)_{(A,D)}$ is the $\U_2$-homomorphism whose component at $K \in \ob\U_2$ is
$$(M^*_\ell)_{A,\:D(-,M_r K)}\;:\;\Alg{\T_1}(A,D(-,M_r K)) \longrightarrow \Alg{\T_2}(AM_\ell,D(M_\ell-,M_r K)).$$

This defines the needed natural transformation $\mathsf{hom}_M$ as in \eqref{eq:hom_m}, and we can now prove the following:
\end{para}

\begin{lem}\label{thm:lax_transf_hom}
The functors $\mathsf{hom}_{(\T,\U)}$ in \eqref{eq:functors_hom} and the natural transformations $\mathsf{hom}_M$ in \eqref{eq:hom_m} together constitute a lax natural transformation of the form \eqref{eq:lax_transf_hom}.
\end{lem}
\begin{proof}
After applying the definitions, one finds that conditions (1) and (2) in \ref{thm:kelly_functor_from_lax_transf} follow from the functoriality of $\Alg{(-)}:\ThJ^\op \to \VCAT$.
\end{proof}

\begin{para}
Writing
$$\Phi \;=\; \pAlg{(-)}^\op \times \pAlg{(-,\blanktwo)}\;\;:\;\;(\ThJ^2)^\op \longrightarrow \CAT\;\;\textnormal{and}$$
$$\Psi \;=\; \pAlg{(-)}\;\;:\;\;\ThJ^\op \longrightarrow \CAT\;,$$
we have a lax natural transformation
$$
\xymatrix{
(\ThJ^2)^\op \ar[rr]^{\pi_2} \ar[dr]_{\Phi} & & \ThJ^\op \ar@{}[dll]|(.35){}="t1"|(.5){}="s1" \ar[dl]^{\Psi}\\
 & \CAT & 
\ar@{~>}"s1";"t1"^{\mathsf{hom}}
}
$$
by \ref{thm:lax_transf_hom}.  Thus we obtain a 1-cell
$$(\pi_2,\mathsf{hom})\;:\;((\ThJ^2)^\op,\Phi) \longrightarrow (\ThJ^\op,\Psi)$$
in Kelly's 2-category $\CAT \int \CAT$, so by Kelly's result in \ref{thm:kelly_functor_from_lax_transf} we obtain a functor
\begin{equation}\label{eq:Hom_as_groth_pi2_hom}\Opfib(\pi_2,\mathsf{hom})\;:\;\Opfib((\ThJ^2)^\op,\Phi) \rightarrow \Opfib(\ThJ^\op,\Psi)\;.\end{equation}
But by  \ref{thm:pb_cat_as_gr_constr} and \ref{defn:algs} we know that
$$\Opfib((\ThJ^2)^\op,\Phi) \;=\; \AlgO^* \tensor[_P]{\times}{_L} \BAlgO\;,\;\;\;\;\Opfib(\ThJ^\op,\Psi) \;=\; \AlgO\;,$$
so we arrive at the following:
\end{para}

\begin{defn}
We define $\sfHom:\AlgO^* \tensor[_P]{\times}{_L} \BAlgO \rightarrow \AlgO$ to be the functor \newline $\Opfib(\pi_2,\mathsf{hom})$ in \eqref{eq:Hom_as_groth_pi2_hom}.
\end{defn}

Using the definition of the 2-functor $\Opfib$ in \ref{thm:kelly_functor_from_lax_transf}, a straightforward unpacking of definitions yields the following detailed description of this functor $\sfHom$, and Objective \ref{obj:global_func_alg_valued_hom} is thus fulfilled:

\begin{thm}\label{thm:descn_hom}
There is a functor
$$\sfHom\;\;:\;\;\AlgO^* \tensor[_P]{\times}{_L} \BAlgO \;\;\longrightarrow\;\; \AlgO$$
that sends each object $(\T,\U,A,D)$ to $(\U,\Hom_\T(A,D))$ and sends each morphism
$$(M,f,g)\;:\;(\T_1,\U_1,A_1,D_1) \rightarrow (\T_2,\U_2,A_2,D_2)$$
(written as in \ref{eq:mor_pb2}) to the morphism
$$(M_r,\psi^{(M,f,g)})\;:\;(\U_1,\Hom_{\T_1}(A_1,D_1)) \rightarrow (\U_2,\Hom_{\T_2}(A_2,D_2))$$
in $\AlgO$, where 
\begin{equation}\label{eq:psi}\psi^{(M,f,g)}\;:\;M_r^*\Hom_{\T_1}(A_1,D_1) \longrightarrow \Hom_{\T_2}(A_2,D_2)\end{equation}
is defined as the composite
$$M_r^*\Hom_{\T_1}(A_1,D_1) \xrightarrow{(\mathsf{hom}_M)_{A_1D_1}} \Hom_{\T_2}(M_\ell^* A_1,M^*D_1) \xrightarrow{\Hom_{\T_2}(f,g)} \Hom_{\T_2}(A_2,D_2)$$
in the category of $\U_2$-algebras.
\end{thm}

\section{The biequivalence of bifold algebras and algebraic dual adjunctions}\label{sec:bieq-bif-alg-alg-dual-adjn}

We now show that the 2-category $\ADAJ$ of $\J$-algebraic dual adjunctions is biequivalent to the category of bifold algebras, $\BAlgO$, when the latter is regarded as a locally discrete 2-category.

\begin{para}\label{para:def_func_from_bifolds_to_jalgdualadjns}
We begin by defining a functor
$$\Dualn\;:\;\BAlgO \longrightarrow \ADAJ$$
that sends each bifold algebra $(\T,\U,D)$ to its dualization adjunction
$$\Dualn(\T,\U,D) \;\;=\;\; (\Alg{\T},\:\Alg{\U},\:\Hom_\U(-,D) \dashv \Hom_\T(-,D))\;,$$
discussed in \ref{thm:dualn_adjn} and \ref{exa:dualn_adjn}.  Given an (op-theoretic) morphism of bifold algebras,
$$(M,g)\;:\;(\T_1,\U_1,D_1) \rightarrow (\T_2,\U_2,D_2)\;\;\;\;\text{in $\BAlgO$,}$$
we define a 1-cell
$$\Dualn(M,g) \;:\; \Dualn(\T_1,\U_1,D_1) \rightarrow \Dualn(\T_2,\U_2,D_2)$$
in $\ADAJ$ as follows.  For brevity of notation, let us write
$$\Dualn(\T_i,\U_i,D_i) = (\A_i,\B_i,\nabla_i \dashv \Delta_i)\;\;\;\;(i = 1,2)\;.$$
First let us recall that $M = (M_\ell,M_r)$ consists of morphisms of $\J$-theories $M_\ell:\T_2 \rightarrow \T_1$ and $M_r:\U_2 \rightarrow \U_1$, and $g:M^* D_1 \to D_2$ is a morphism in $\pAlg{(\T_2,\U_2)}$. By applying the functor $\sAlgs:\ThJ^\op \rightarrow \ALGCATJ$ of \pref{eqn:func_alg_valued_in_psslice} we obtain 1-cells $\sAlgs M_\ell = (M_\ell^*,1):\A_1 \rightarrow \A_2$ and $\sAlgs M_r = (M_r^*,1):\B_1 \rightarrow \B_2$ in $\ALGCATJ$.  Using the notation of \eqref{eq:1cell_algdual_1}, we define
$$\Dualn(M,g) = (M^*_\ell,1,M_r^*,1,\phi^{(M,g)})\;:\;(\A_1,\B_1,\nabla_1 \dashv \Delta_1) \longrightarrow (\A_2,\B_2,\nabla_2 \dashv \Delta_2)$$
to be the 1-cell in $\ADAJ$ consisting of these 1-cells $\sAlgs M_\ell$ and $\sAlgs M_r$ in $\ALGCATJ$ together with a $\V$-natural transformation
\begin{equation}\label{eq:1cell_dualn}
\xymatrix{
\Alg{\T_1}^\op \ar[r]^{(M_\ell^*)^\op} \ar[d]_{\Hom_{\T_1}(-,D_1)} & \Alg{\T_2}^\op \ar[d]^{\Hom_{\T_2}(-,D_2)}\\
\Alg{\U_1} \ar@{}[ur]|(.4){}="t1"|(.6){}="s1"^{\phi^{(M,g)}} \ar[r]_{M_r^*} & \Alg{\U_2}
\ar@{<=}"s1";"t1"
}
\end{equation}
that we now define.  Given a $\T_1$-algebra $A$, we obtain an associated morphism
$$(M,\:1_{M_\ell^*A},\:g)\;:\;(\T_1,\U_1,A,D_1) \rightarrow (\T_2,\U_2,M_\ell^* A,D_2)$$
in the category $\AlgO^* \tensor[_P]{\times}{_L} \BAlgO$ that is discussed in \ref{para:pb_algstar_bifoldalg}, \ref{para:discn_pb_algstar_bifoldalg}, and \ref{thm:pb_cat_as_gr_constr}.  Hence we may apply the functor $\sfHom:\AlgO^* \tensor[_P]{\times}{_L} \BAlgO \rightarrow \AlgO$ \pref{thm:descn_hom} to obtain a morphism
$$\sfHom(M,\:1_{M_\ell^*A},\:g)\;:\;\sfHom(\T_1,\U_1,A,D_1) \rightarrow \sfHom(\T_2,\U_2,M_\ell^* A,D_2)$$
in $\AlgO$, which we may write with the notation of \ref{thm:descn_hom}  as
$$\left(M_r,\psi^{(M,1_{M_\ell^*A},g)}\right)\;:\;(\U_1,\Hom_{\T_1}(A,D_1)) \rightarrow (\U_2,\Hom_{\T_2}(M_\ell^* A,D_2))\;.$$
The second component is a morphism in $\pAlg{\U_2}$ that we write as
$$\phi^{(M,g)}_A := \psi^{(M,1_{M_\ell^* A},g)}\;:\;M_r^*\Hom_{\T_1}(A,D_1) \longrightarrow \Hom_{\T_2}(M_\ell^*A,D_2)$$
and that by \eqref{eq:psi} is the composite
$$M_r^*\Hom_{\T_1}(A,D_1) \xrightarrow{(\mathsf{hom}_M)_{A D_1}} \Hom_{\T_2}(M_\ell^* A,M^*D_1) \xrightarrow{\Hom_{\T_2}(M_\ell^* A,g)} \Hom_{\T_2}(M_\ell^* A,D_2)\;.$$
By \eqref{eq:hom_m_ad}, this family of morphisms $\phi^{(M,g)}_A$ is $\V$-natural in $A \in \Alg{\T_1}$ and so constitutes a $\V$-natural transformation $\phi^{(M,g)}$ of the needed form \eqref{eq:1cell_dualn}.
\end{para}

\begin{prop}\label{thm:dualn_func}
The assignments described in \ref{para:def_func_from_bifolds_to_jalgdualadjns} define a functor
$$\Dualn:\BAlgO \longrightarrow \ADAJ$$
that sends each bifold algebra to its dualization adjunction (\ref{defn:dualn_adjn}, \ref{exa:dualn_adjn}).
\end{prop}
\begin{proof}
The functoriality of $\Dualn$ follows from that of
$$\sfHom:\AlgO^* \tensor[_P]{\times}{_L} \BAlgO \rightarrow \AlgO\;.$$
In detail, suppose we are given bifold algebras $(\T_i,\U_i,D_i)$ $(i = 1,2,3)$, which we write simply as $D_i$, and let $(M,g):D_1 \rightarrow D_2$ and $(N,h):D_2 \rightarrow D_3$ in $\BAlgO$.  Then the composite in $\BAlgO$ is $(N,h) \cdot (M,g) = (M \cdot N,h \cdot N^*g)$ where we write $M \cdot N:(\T_3,\U_3) \rightarrow (\T_1,\U_1)$ for the composite in $\ThJ^2$.  Letting $A$ be an arbitrary $\T_1$-algebra, we may consider the associated composable pair of morphisms
$$(A,D_1) \xrightarrow{(M,1,g)} (M_\ell^*A,D_2) \xrightarrow{(N,1,h)} (N_\ell^*M_\ell^*A,D_3)$$
in $\AlgO^* \tensor[_P]{\times}{_L} \BAlgO$, whose composite is
$$(A,D_1) \xrightarrow{(M \cdot N,\;1,\;h \cdot N^*g)} ((M \cdot N)_\ell^*A,D_3)\;,$$
where $h \cdot N^* g:(M \cdot N)^* D_1 = N^*M^* D_1 \to D_3$ in $\pAlg{(\T_3,\U_3)}$. By applying $\sfHom:\AlgO^* \tensor[_P]{\times}{_L} \BAlgO \rightarrow \AlgO$, we obtain a diagram
$$
\xymatrix{
\sfHom(A,D_1) \ar[dr]_{\sfHom(M,1,g)\hspace{3ex}} \ar[rr]^{\sfHom(M \cdot N,1,h \cdot N^*g)} & & \sfHom(N_\ell^*M_\ell^*A,D_3)\\
             & \sfHom(M_\ell^*A,D_2) \ar[ur]_{\hspace{3ex}\sfHom(N,1,h)}
}
$$
in $\AlgO$, which commutes, by the functoriality of $\sfHom$.  But in view of \ref{para:def_func_from_bifolds_to_jalgdualadjns}, this diagram can be written as
$$
\xymatrix{
(\U_1,\Hom_{\T_1}(A,D_1)) \ar[dr]_{\left(M_r,\phi^{(M,g)}_A\right)\hspace{3ex}} \ar[rr]^{\left((M \cdot N)_r,\phi^{(N,h) \cdot (M,g)}_A\right)} & & (\U_3,\Hom_{\T_3}(N_\ell^*M_\ell^*A,D_3))\\
             & (\U_2,\Hom_{\T_2}(M_\ell^*A,D_2)) \ar[ur]_{\hspace{3ex}\left(N_r,\phi^{(N,h)}_{M_\ell^* A}\right)}
}
$$
and since this diagram commutes for each $\T_1$-algebra $A$ it follows that $\phi^{(N,h) \cdot (M,g)}$ is equal to the pasted 2-cell
$$
\xymatrix{
\Alg{\T_1}^\op \ar[d]_{\Hom_{\T_1}(-,D_1)} \ar[rr]^{(M_\ell^*)^\op} & & \Alg{\T_2}^\op \ar@{}[dll]|(.4){}="t1"|(.6){}="s1" \ar[d]|{\Hom_{\T_2}(-,D_2)} \ar[rr]^{(N_\ell^*)^\op} & & \Alg{\T_3}^\op \ar@{}[dll]|(.4){}="t2"|(.6){}="s2" \ar[d]^{\Hom_{\T_3}(-,D_3)}\\
\Alg{\U_1} \ar[rr]_{M_r^*} & & \Alg{\U_2} \ar[rr]_{N_r^*} & & \Alg{\U_3}.
\ar@{=>}"s1";"t1"^{\phi^{(M,g)}}
\ar@{=>}"s2";"t2"^{\phi^{(N,h)}}
}
$$
Therefore $\Dualn((N,h) \cdot (M,g)) = \left(N_\ell^*M_\ell^*,1,N_r^*M_r^*,1,\phi^{(N,h) \cdot (M,g)}\right)$ is equal to the composite $\Dualn(N,h) \cdot \Dualn(M,g)$ in $\ADAJ$.  It is straightforward to verify by a similar method that $\Dualn$ preserves identity morphisms.
\end{proof}

\begin{lem}\label{thm:lem_for_local_equiv}
Let $(\T_i,\U_i,D_i)$ $(i = 1,2)$ be bifold algebras, let $M = (M_\ell,M_r):(\T_2,\U_2) \rightarrow (\T_1,\U_1)$ be a morphism in $\ThJ^2$. Every 2-cell
\begin{equation}\label{eq:xi}
\xymatrix{
\Alg{\T_1}^\op \ar[r]^{(M_\ell^*)^\op} \ar[d]_{\Hom_{\T_1}(-,D_1)} & \Alg{\T_2}^\op \ar[d]^{\Hom_{\T_2}(-,D_2)}\\
\Alg{\U_1} \ar@{}[ur]|(.4){}="t1"|(.6){}="s1"^{\xi} \ar[r]_{M_r^*} & \Alg{\U_2}
\ar@{<=}"s1";"t1"
}
\end{equation}
in $\VCAT$ is of the form $\xi = \phi^{(M,g)}$ for a unique morphism $g:M^* D_1 \rightarrow D_2$ in $\Alg{(\T_2,\U_2)}$, with the notation of \eqref{eq:1cell_dualn}.
\end{lem}
\begin{proof}
In view of \ref{para:dualn-func} and \eqref{eq:m-star-hom}, the composite $\V$-functors in \eqref{eq:xi} can be written as
$$M_r^*\Hom_{\T_1}(-,D_1) = \llbracket-,D_{1r\ell} M_r\rrbracket,\;\;\;\;\Hom_{\T_2}(M_\ell^*-,D_2) = \llbracket M_\ell^*-,D_{2r\ell}\rrbracket,$$
where we write $D_{ir\ell} = (D_i)_{r\ell}$ $(i = 1,2)$. By \ref{exa:ind-transf-between-homs}, the Yoneda principle provides a bijection between $\V$-natural transformations $\xi:\llbracket-,D_{1r\ell} M_r\rrbracket \Rightarrow \llbracket M_\ell^*-,D_{2r\ell}\rrbracket$ as in \eqref{eq:xi} and $\V$-natural transformations $g:M_\ell^*D_{1r\ell} M_r \Rightarrow D_{2r\ell}:\U_2 \to \Alg{\T_2}$, equivalently, morphisms $g:M^*D_1 \to D_2$. Under this bijection, $\phi^{(M,g)}$ corresponds to $g$, in view of \ref{para:lax_transf_hom} and \ref{para:def_func_from_bifolds_to_jalgdualadjns}.
\end{proof}

Let us now regard the category $\BAlgO$ as a locally discrete 2-category, so that $\Dualn:\BAlgO \rightarrow \ADAJ$ may be regarded as a 2-functor.

\begin{thm}\label{thm:bifolds_biequiv_ada}
The 2-functor $\Dualn:\BAlgO \rightarrow \ADAJ$ is a biequivalence
$$\BAlgO \;\;\;\simeq\;\;\; \ADAJ$$
between the (locally discrete) 2-category of bifold algebras and the 2-category of \linebreak$\J$-algebraic dual adjunctions.
\end{thm}
\begin{proof}
By Theorem \ref{thm:alg_dual_bifold}, the 2-functor $\Dualn$ is biessentially surjective on objects, so by \ref{sec:bieq} it suffices to show that $\Dualn$ is a local equivalence.  Letting $(\T_i,\U_i,D_i)$ $(i = 1,2)$ be bifold algebras, which we write simply as $D_i$, it suffices to show that the functor
$$\Dualn_{D_1D_2}\;:\;\BAlgO(D_1,D_2) \rightarrow \ADAJ(\Dualn D_1,\Dualn D_2)$$
is an equivalence, where the set $\BAlgO(D_1,D_2)$ is regarded as a discrete category.

Let us write $\Dualn D_i = (\A_i,\B_i,\nabla_i \dashv \Delta_i)$ $(i = 1,2)$, so that $\A_i = \Alg{\T_i}$ and $\B_i = \Alg{\U_i}$.  To show that $\Dualn_{D_1D_2}$ is essentially surjective on objects, let $P = (P_\ell,\lambda^P,P_r,\rho^P,\phi^P):\Dualn D_1 \rightarrow \Dualn D_2$ be a 1-cell in $\ADAJ$, with the notation of \eqref{eq:1cell_algdual_1}.  In particular, $(P_\ell,\lambda^P):\Alg{\T_1} \rightarrow \Alg{\T_2}$ and $(P_r,\rho^P):\Alg{\U_1} \rightarrow \Alg{\U_2}$ are then 1-cells in $\ALGCATJ$, but, by \ref{thm:bieq_th_algcat}, the 2-functor $\sAlgs:\ThJ^\op \rightarrow \ALGCATJ$ is a biequivalence and hence a local equivalence, when $\ThJ^\op$ is regarded as a locally discrete 2-category, so there are morphisms $M_\ell:\T_2 \rightarrow \T_1$ and $M_r:\U_2 \rightarrow \U_1$ in $\ThJ$ and invertible 2-cells $\zeta_\ell:\sAlgs M_\ell = (M^*_\ell,1) \Rightarrow (P_\ell,\lambda^P):\Alg{\T_1} \rightarrow \Alg{\T_2}$ and $\zeta_r:(P_r,\rho^P) \Rightarrow \sAlgs M_r = (M^*_r,1):\Alg{\U_1} \rightarrow \Alg{\U_2}$ in $\ALGCATJ$.  In particular, we obtain a morphism $M = (M_\ell,M_r):(\T_1,\U_1) \rightarrow (\T_2,\U_2)$ in $(\ThJ^2)^\op$.  Let $\xi$ denote the pasted 2-cell
$$
\xymatrix@!0@C=8ex @R=4ex{
\A_1^\op \ar[dd]_{\Delta_1} \ar@/^2pc/[rrr]|{(M_\ell^*)^\op}="t2" \ar[rrr]|{P_\ell^\op}="s2"  & & \ar@{}[ddl]|(.3){}="t1"|(.7){}="s1" & \A_2^\op \ar[dd]^{\Delta_2} \\
 & & & &\\
\B_1 \ar@/_2pc/[rrr]|{M_r^*}="s4" \ar[rrr]|{P_r}="t4" & & & \B_2
\ar@{=>}"s1";"t1"^(.2){\phi^P}
\ar@{}"s2";"t2"|(.25){}="s3"|(.75){}="t3"
\ar@{}"s3";"t3"|(.1){}="s5"|(.9){}="t5"
\ar@{=>}"s5";"t5"_{\zeta_\ell^\op}
\ar@{}"s4";"t4"|(.3){}="s6"|(.7){}="t6"
\ar@{=>}"s6";"t6"_{\zeta_r^{-1}}
}
$$
in $\VCAT$.  Then by \ref{thm:lem_for_local_equiv} we deduce that there is a unique morphism $g:M^* D_1 \rightarrow D_2$ in $\pAlg{(\T_2,\U_2)}$ such that that $\xi = \phi^{(M,g)}$, and in view of the definition of $\xi$ this entails that the pasted 2-cells
$$
\xymatrix@!0@C=8ex @R=4ex{
\A_1^\op \ar[dd]_{\Delta_1} \ar@/^1pc/[rrr]|{(M_\ell^*)^\op}="t2" \ar@/_1pc/[rrr]|{P_\ell^\op}="s2"  & & & \A_2^\op \ar[dd]^{\Delta_2} \ar@{}[ddll]|(.6){}="t1"|(.8){}="s1" & & \A_1^\op \ar[dd]_{\Delta_1} \ar@/^1pc/[rrr]^{(M_\ell^*)^\op} & & \ar@{}[dl]|(.3){}="t5"|(.7){}="s5" & \A_2^\op \ar[dd]^{\Delta_2}\\
 & & & & & & & & \\
\B_1 \ar@/_1pc/[rrr]|{P_r} & & & \B_2 & & \B_1 \ar@/^1pc/[rrr]|{M_r^*}="s4" \ar@/_1pc/[rrr]|{P_r}="t4" & & & \B_2
\ar@{=>}"s1";"t1"_{\phi^P}
\ar@{}"s2";"t2"|(.25){}="s3"|(.75){}="t3"
\ar@{=>}"s3";"t3"^{\zeta_\ell^\op}
\ar@{=>}"s5";"t5"^{\phi^{(M,g)}}
\ar@{}"s4";"t4"|(.25){}="t6"|(.75){}="s6"
\ar@{=>}"s6";"t6"^{\zeta_r}
}
$$
are equal.  Hence, since $\Dualn(M,g) = (M_\ell^*,1,M_r^*,1,\phi^{(M,g)})$, this entails that the pair $\zeta = (\zeta_\ell,\zeta_r)$ is an invertible 2-cell $\zeta:P \Rightarrow \Dualn(M,g)$ in $\ADAJ$ (in view of \eqref{eq:2cell_algdual}).

This shows that $\Dualn_{D_1D_2}$ is essentially surjective on objects.  To show that $\Dualn_{D_1D_2}$ is fully faithful, let $(M,g),(N,h):D_1 \rightarrow D_2$ be morphisms in $\BAlgO$, and let $\gamma = (\gamma_\ell,\gamma_r):\Dualn(M,g) \Rightarrow \Dualn(N,h):\Dualn D_1 \rightarrow \Dualn D_2$ be a 2-cell in $\ADAJ$.  By \ref{para:def_func_from_bifolds_to_jalgdualadjns}, we may write
$$\Dualn(M,g) = (\sAlgs M_\ell,\sAlgs M_r,\phi^{(M,g)}),\;\;\;\;\Dualn(N,h) = (\sAlgs N_\ell,\sAlgs N_r,\phi^{(N,h)})$$
 where $\sAlgs:\ThJ^\op \rightarrow \ALGCATJ$ is the functor defined in \eqref{eqn:func_alg_valued_in_psslice}.  Hence, by \eqref{eq:2cell_algdual} $\gamma_\ell:\sAlgs N_\ell \Rightarrow \sAlgs M_\ell$ and $\gamma_r:\sAlgs M_r  \Rightarrow \sAlgs N_r$ are 2-cells in $\ALGCATJ$, so since $\sAlgs:\ThJ^\op \rightarrow \ALGCATJ$ is a biequivalence when $\ThJ^\op$ is considered as a locally discrete 2-category (\ref{thm:bieq_th_algcat}), we deduce that $N_\ell = M_\ell$ and $M_r = N_r$ (so that $M = N$) and that $\gamma_r$ and $\gamma_\ell$ are identity 2-cells in $\ALGCATJ$.  But since $\gamma = (\gamma_\ell,\gamma_r)$ satisfies the pasting equation required of a 2-cell in $\ADAJ$, this entails that $\phi^{(M,g)} = \phi^{(N,h)}$.  Consequently, $\Dualn(M,g) = \Dualn(N,h)$, and $\gamma$ is the identity 2-cell on $\Dualn(M,g)$.  Also, since $M = N$ and $\phi^{(M,g)} = \phi^{(N,h)}$, we deduce by \ref{thm:lem_for_local_equiv} that $g = h$, so $(M,g) = (N,h)$ and the result follows, since $\BAlgO$ is locally discrete.
\end{proof}

\section{Basic example: Dualization of modules with respect to a bimodule}\label{exa:bim}

In this section, we consider a special case of Theorem \ref{thm:bifolds_biequiv_ada} where we let $\J = \{I\}$ be the system of arities $\{I\} \hookrightarrow \V$ consisting of just the unit object $I$ of a given symmetric monoidal closed category $\V$ for which $\V_0$ has finite limits.  Recall from \ref{exa:sys-ar} that $\{I\}$-theories $\R$ are equivalently monoids $R$ in $\V$, so that we have an equivalence $\Th_{\{I\}} \simeq \Mon(\V)$, and $\R$-algebras are left $R$-modules in $\V$.  Moreover, $\pAlg{\R}$ is the category $\pMod{R}$ of left $R$-modules in $\V$, which therefore underlies a $\V$-category $\Mod{R} = \Alg{\T}$.

Given monoids $R$ and $S$ in $\V$, let us regard $R$ and $S^\op$ as $\{I\}$-theories $\R$ and $\sS^\op$, respectively.  Then an $(\R,\sS^\op)$-algebra is equivalently described as an \mbox{\bit{$R$-$S$-bimodule}} in $\V$ (i.e., an $R$-$S$-biobject in the sense of Pareigis \cite[\S 3]{Parei:I}).  In particular, when $\V = \Ab$ is the category of abelian groups, with the usual tensor product and the system of arities $\{\ZZ\}$, we recover the usual notion of $R$-$S$-bimodule for a pair of rings $R$ and $S$. We write $\pMod{\text{$R$-$S$}}$ to denote the category of $R$-$S$-bimodules in $\V$, with morphisms in $\V$ that preserve both the left $R$- and right $S$-actions, noting that $\pMod{\text{$R$-$S$}} \cong \pAlg{(\R,\sS^\op)}$.   Writing $\Mon(\V)^2= \Mon(\V) \times \Mon(\V)$, there is a functor $\Phi:(\Mon(\V)^2)^\op \rightarrow \CAT$
given on objects by $\Phi(R,S) = \pMod{\text{$R$-$S$}}$ and on morphisms by restriction of scalars.  We write $\Bimod(\V)$ to denote the op-fibred category over $(\Mon(\V)^2)^\op$ determined by $\Phi$.  The objects of $\Bimod(\V)$ may be called \bit{bimodules in $\V$}, for we may write them as triples $(R,S,D)$, where $R$ and $S$ are monoids in $\V$ and $D$ is an $R$-$S$-bimodule in $\V$.  A morphism $(f_\ell,f_r,g):(R_1,S_1,D_1) \rightarrow (R_2,S_2,D_2)$ in $\Bimod(\V)$ consists of monoid morphisms $f_\ell:R_2 \rightarrow R_1$ and $f_r:S_2 \rightarrow S_1$ together with a morphism $g:\tensor[_{R_2}]{(D_1)}{_{S_2}} \rightarrow D_2$ in \pMod{\text{$R_2$-$S_2$}}, where we write $\tensor[_{R_2}]{(D_1)}{_{S_2}}$ for the $R_2$-$S_2$-bimodule obtained from $D_1$ by restriction of scalars along $f_\ell$ and $f_r$.

The system of arities $\{I\} \hookrightarrow \V$ is eleutheric (\ref{exa:sys-ar}), and a $\V$-functor $G:\A \rightarrow \V$ is $\{I\}$-algebraic iff $\A \simeq \Mod{R}$ in $\VCAT \sslash \V$ for some monoid $R$ in $\V$, where $\A$ is equipped with $G$ and $\Mod{R}$ is equipped with its `carrier' $\V$-functor; note that it makes no difference whether one employs left or right modules in this definition. Thus we refer to $\{I\}$-algebraic $\V$-categories over $\V$ as \bit{module $\V$-categories over $\V$}.

By a \bit{dual adjunction of module $\V$-categories} we mean a triple $(\A,\B,\nabla \dashv \Delta)$ consisting of module $\V$-categories $\A$ and $\B$ over $\V$ together with a $\V$-adjunction $\nabla \dashv \Delta:\A^\op \rightarrow \B$.  As a special case of \ref{thm:dualn_adjn}, every $R$-$S$-bimodule $D$ in $\V$ determines a dual adjunction of module $\V$-categories
\begin{equation}\label{eq:dualn_bimod_cor}\RAdjn{\Mod{R}^\op}{\Hom_R(-,D)}{\Hom_{S^\op}(-,D)}{}{}{\Mod{S^\op}}\end{equation}
that we  call the \textit{dualization adjunction} for $D$. In the special case where $\V = \Ab$, Morita studied dual \textit{equivalences} between full subcategories of $\Mod{R}$ and $\Mod{S^\op}$ for rings $R$ and $S$, employing as a basic tool the observation that every such dual equivalence is given by dualizing into an $R$-$S$-bimodule $D$ \cite[Theorem 1.1]{Morita}.

Dual adjunctions of module $\V$-categories are precisely $\{I\}$-algebraic dual adjunctions for the system of arities $\{I\} \hookrightarrow \V$.  Hence, by \ref{defn:dualadjj} there is a 2-category $\ADA_{\{I\}}$ whose objects are dual adjunctions of module $\V$-categories.

\begin{cor}\label{cor:dual_adj_mod_cats}
Let $\V$ be a symmetric monoidal closed category with finite limits.  Then the 2-category $\ADA_{\{I\}}$ of dual adjunctions of module $\V$-categories is biequivalent to the locally discrete 2-category $\Bimod(\V)$ of bimodules in $\V$.  Under this biequivalence, each $R$-$S$-bimodule $D$ in $\V$ corresponds to its dualization adjunction \eqref{eq:dualn_bimod_cor}.
\end{cor}
\begin{proof}
In view of the preceding discussion, this follows from Theorem \ref{thm:bifolds_biequiv_ada}, for upon identifying monoids in $\V$ with $\{I\}$-theories we obtain an isomorphism $\Bimod(\V) \cong \BAlg_{\{I\}}$ given on objects (in both directions) by $(R,S,D) \mapsto (R,S^\op,D)$.
\end{proof}

By returning to Morita's setting \cite{Morita} of $\V = \Ab$, we obtain a corollary on dual adjunctions between categories of modules in the usual sense. By instead taking $\V = \Set$ we obtain a corollary on dual adjunctions between toposes of monoid actions.

\section{Supporting theory II: Commutants and commutation}\label{sec:background-II}

Returning to the general setting of the paper (\ref{para:setting2}), we now review some background material from \cite{Lu:Cmt,Lu:BAlgCmt}.

\begin{para}\label{para:cats-algs}
Applying the classic Grothendieck construction (\ref{para:fibr_op}), the functor $\pAlg{(-)}:\ThJ^\op \rightarrow \CAT$ determines a fibred category over $\ThJ$ that we write as
$$\AlgT = \Fib(\ThJ,\pAlg{(-)}),$$
and call the \bit{fibred category of algebras}, as its objects are algebras $(\T,A)$ in the sense of \ref{para:talgs}; these we also write simply as $A$, calling $A$ a \bit{($\T$-)algebra on $C$} if $C = |A|$. By definition, a \bit{theoretic morphism of algebras} is a morphism $(M,f):(\T,A) \to (\U,B)$ in $\AlgT$ and so consists of a morphism of $\J$-theories $M:\T \to \U$ and a $\T$-homomorphism $f:A \to M^*B$. By \ref{para:fibr_op} and \ref{para:alg-star}, $\AlgT^\op = \AlgO^*$.

We now review some related categories of algebras defined in \cite[3.1, 3.4]{Lu:BAlgCmt}. A theoretic morphism of algebras $(M,f)$ is \bit{strong} if $f$ is an isomorphism. We write $\AlgS$ for the non-full subcategory of $\AlgT$ given by the strong morphisms. Writing $\C_\si$ for the \textit{groupoid core} of each category $\C$, i.e.~the (non-full) subcategory whose morphisms are all the isomorphisms in $\C$, we obtain by \eqref{eq:func_talg} a functor $\pAlg{(-)}_\si:\ThJ^\op \rightarrow \CAT$ such that
$$\AlgS = \Fib(\ThJ,\pAlg{(-)}_\si)\;.$$

Let $C$ be an object of $\V$.  For each $\J$-theory $\T$, let $\pAlg{\T}(C)$ denote the fibre over $C$ of the functor $G^\T:\pAlg{\T} \rightarrow \V_0$, so that $\pAlg{\T}(C)$ is the (non-full) subcategory of $\pAlg{\T}$ consisting of $\T$-algebras on $C$ and $\T$-homomorphisms $f:A \rightarrow B$ with $|f| = 1_C$.  By \cite[3.4]{Lu:BAlgCmt}, $\pAlg{\T}(C)$ is both a groupoid and also a preordered class, so simply a class equipped with an equivalence relation $\cong$.  By \eqref{eq:func_talg}, we obtain an evident functor $\pAlg{(-)}(C):\ThJ^\op \rightarrow \CAT$.  Applying the Grothendieck construction to this functor, we obtain a category
$$\AlgT(C) = \Fib(\ThJ,\pAlg{(-)}(C))$$
that we call the \bit{fibred category of algebras on $C$}.  Indeed, $\AlgT(C)$ is the (non-full) subcategory of $\AlgT$ whose objects are algebras on $C$  and whose morphisms are theoretic morphisms $(M,f)$ with $|f| = 1_C$.   But since $\pAlg{\T}(C)$ is both a groupoid and a preordered class, a morphism  $M:(\T,A) \rightarrow (\U,B)$ in $\AlgT(C)$ is equivalently given by a morphism of $\J$-theories $M:\T \rightarrow \U$ with the property that $A \cong M^*(B)$ in $\pAlg{\T}(C)$.
\end{para}

\begin{para}\label{sec:c-morph-balg}
There is a category $\BAlg_\sx$ whose objects are bifold algebras $(\T,\U,D)$, in which a morphism
$$(M,N,f):(\T,\U,D) \rightarrow (\T',\U',D')$$
is called a \bit{strong cross-morphism of bifold algebras} and is given by morphisms of $\J$-theories $M:\T \rightarrow \T'$ and $N:\U' \rightarrow \U$ (in opposite directions) and an isomorphism $f:D(-,N\blanktwo) \rightarrow D'(M-,\blanktwo)$ in $\pAlg{(\T,\U')}$ \cite[7.2]{Lu:BAlgCmt}.  Indeed, as discussed in \cite[7.1, 7.2]{Lu:BAlgCmt}, $\BAlg_\sx$ may be obtained by applying the \textit{Grothendieck construction for two-sided fibrations} \cite{Str:FibrYon2Cats} to the functor $\pAlg{(-,\blanktwo)}_\si:\ThJ^\op \times \ThJ^\op \rightarrow \CAT$ that sends each pair of $\J$-theories $(\T,\U)$ to the groupoid core $\pAlg{(\T,\U)}_\si$ of $\pAlg{(\T,\U)}$.
\end{para}

\begin{para}[\textbf{Commutants and commuting algebra pairs}]\label{para:cmt}
Given an algebra $(\T,A)$, the \bit{commutant of $\T$ with respect to $A$} \cite{Lu:Cmt} is the $\J$-theory $\T^\perp_A$ with hom-objects $\T^\perp_A(J,K) = \Alg{\T}(A^J,A^K)$ for all $J,K \in \ob\T^\perp_A = \ob\J$, and with composition and identities as in $\Alg{\T}$, where $A^J$ denotes the pointwise power of $A$ by $J$ in $\Alg{\T}$.  As discussed in \cite{Lu:Cmt,Lu:BAlgCmt}, there is a $\T^\perp_A$-algebra $A^\perp:\T^\perp_A \rightarrow \V$ that is given on objects by $J \mapsto |A|^J$ and on homs by $A^\perp_{JK} = G^\T_{A^JA^K}:\Alg{\T}(A^J,A^K) \hookrightarrow \V(|A|^J,|A|^K)$ with the notation of \ref{para:talgs}.  Hence $(\T^\perp_A,A^\perp)$ is an algebra, which we call \bit{the commutant of $(\T,A)$} \cite{Lu:BAlgCmt} and denote by
$$(\T,A)^\perp = (\T^\perp_A,A^\perp)\;.$$
By our convention in \ref{para:cats-algs}, we also write the algebra $(\T,A)^\perp$ simply as $A^\perp$ and call it \textit{the commutant of $A$}. Note that $|A^\perp| = |A|$.

Given algebras $A$ and $B$ on an object $C$ of $\V$, there is at most one morphism $M:B \rightarrow A^\perp$ in $\AlgT(C)$ \cite[4.1]{Lu:BAlgCmt}, with the notation of \ref{para:cats-algs}, and if there exists such a morphism $M$ then we say that $A$ \bit{commutes with} $B$, and we denote $M$ by $[B|A]:B \rightarrow A^\perp$ \cite[4.2]{Lu:BAlgCmt}. This succinct definition deserves elaboration: First, we may write $A$ and $B$ as $(\T,A)$ and $(\U,B)$, respectively, and then $A$ commutes with $B$ iff there exists a (necessarily unique) morphism of $\J$-theories $[B|A]:\U \rightarrow \T^\perp_A$ such that $B \cong [B|A]^*(A^\perp)$ in $\pAlg{\U}(C)$ \cite[4.1]{Lu:BAlgCmt}. The elementary characterization of commutation in \cite[4.4]{Lu:BAlgCmt} helps build intuition: Recalling that $|A| = |B| = C$, $A$ commutes with $B$ iff for each $J \in \ob\J = \ob\U$ the morphism $B_{JI}:\U(J,I) \to \V(BJ,BI) \cong \V(C^J,C)$ factors through the monomorphism $G^\T_{A^JA}:\Alg{\T}(A^J,A) \hookrightarrow \C(C^J,C)$. Loosely speaking, this captures the idea that the $\U$-operations are $\T$-homomorphisms. Further discussion of commutation can be found in \cite{Lu:Cmt,Lu:CvxAffCmt}. Commutation is symmetric, i.e.~$A$ commutes with $B$ iff $B$ commutes with $A$ \cite[8.10]{Lu:BAlgCmt}. An algebra $A$ is \bit{commutative} if $A$ commutes with itself \cite[13.1]{Lu:BAlgCmt}. A $\J$-theory $\T$ is \bit{commutative} iff every $\T$-algebra in $\V$ is commutative; see \cite[13.3]{Lu:BAlgCmt} for equivalent definitions of this concept.

We say that a pair of algebras $(A,B)$ is an \bit{algebra pair} on an object $C$ of $\V$ if $|A| = |B| = C$; if moreover $A$ commutes with $B$, then we call $(A,B)$ a \bit{commuting algebra pair} \cite[11.1]{Lu:BAlgCmt}.  If $(A,B)$ is a commuting algebra pair in which $A$ is a $\T$-algebra and $B$ a $\U$-algebra, then we say that $(A,B)$ is a \bit{commuting $\T$-$\U$-algebra pair}.  Commuting $\T$-$\U$-algebra pairs $(A,B)$ constitute a full sub-$\V$-category $\CAlgPair{\T}{\U}$ of the pullback $\Alg{\T} \times_{\V} \Alg{\U}$, and by \cite[11.11]{Lu:BAlgCmt} there is an equivalence of $\V$-categories $\Alg{(\T,\U)} \simeq \CAlgPair{\T}{\U}$, which sends each $(\T,\U)$-algebra $D$ to the pair $(D_\ell,D_r)$ consisting of the left and right faces of $D$.  Moreover, there is a functor $\CAlgPair{-}{\blanktwo}:(\ThJ^2)^\op \to \VCAT$, and the equivalence $\Alg{(\T,\U)} \xrightarrow{\sim} \CAlgPair{\T}{\U}$ is (strictly) natural in $\T,\U \in \ThJ$ \cite[11.12]{Lu:BAlgCmt}.  In particular, if $(A,B)$ is any commuting $\T$-$\U$-algebra pair then, up to $\V$-natural isomorphism, there exists a unique $(\T,\U)$-algebra $D:\T \otimes \U \rightarrow \V$ such that $(D_\ell,D_r) \cong (A,B)$ in $\CAlgPair{\T}{\U}$, and as in \cite[11.14]{Lu:BAlgCmt} we denote this $(\T,\U)$-algebra $D$ by
$$\langle A,B\rangle\;.$$
Writing the algebras $(\T,A)$ and $(\U,B)$ simply as $A$ and $B$, respectively, we also write the bifold algebra $(\T,\U,\langle A,B\rangle)$ simply as $\langle A,B\rangle$.
\end{para}

\begin{para}[\textbf{Commutant bifold algebras}]\label{para:cmtnt_bif_algs}
Given an algebra pair $(A,B)$ on an object $C$ of $\V$, we say that $B$ \bit{is the commutant of} $A$ if $B \cong A^\perp$ in $\AlgT(C)$ \cite[11.22]{Lu:BAlgCmt}.  If we write $A$ and $B$ as $(\T,A)$ and $(\U,B)$, respectively, then, by \cite[11.23]{Lu:BAlgCmt}, $B$ is the commutant of $A$ iff $A$ commutes with $B$ and the morphism of $\J$-theories $[B|A]:\U \rightarrow \T^\perp_A$ is an isomorphism.  We say that an algebra pair $(A,B)$ is a \bit{right-commutant algebra pair} (resp.~\bit{left-commutant algebra pair}) if $B$ is the commutant of $A$ (resp.~if $A$ is the commutant of $B$) \cite[11.22]{Lu:BAlgCmt}.  A \bit{commutant algebra pair} is an algebra pair that is both left-commutant and right-commutant.

Given an algebra $A$, we obtain a right-commutant algebra pair $(A,A^\perp)$. We say that the algebra $A$ is \bit{saturated} if $A$ is the commutant of $A^\perp$, i.e., if $A \cong A^{\perp\perp}$ in $\AlgT(C)$ where $C = |A|$, equivalently, if $A \cong A^{\perp\perp}$ in $\AlgS$, equivalently, if $A$ is the commutant of some algebra \cite[10.2, 10.3]{Lu:BAlgCmt}.

A \bit{right-commutant bifold algebra} (resp.~\bit{left-commutant bifold algebra}, \bit{commutant bifold algebra}) is a bifold algebra $D$ such that the algebra pair $(D_\ell,D_r)$ is right-commutant (resp.~left-commutant, commutant) \cite[9.1, 10.7]{Lu:BAlgCmt}. Various examples of these concepts are discussed in Section \ref{sec:exa_ada}. Given a bifold algebra $D = (\T,\U,D:\T \otimes \U \to \V)$, we have the following results from \cite[9.2, 9.3, 9.5, 10.8, 11.20]{Lu:BAlgCmt}: (1) $D$ is right-commutant iff $D \cong \langle A,A^\perp\rangle$ in $\BAlg_\sx$ for some algebra $A$, iff the transpose $D_{r\ell}:\U \rightarrow \Alg{\T}$ is fully faithful, (2) $D$ is left-commutant iff $D \cong \langle B^\perp, B\rangle$ in $\BAlg_\sx$ for some algebra $B$, iff the transpose $D_{\ell r}:\T \rightarrow \Alg{\U}$ is fully faithful, and (3) $D$ is commutant iff $D \cong \langle A,A^\perp\rangle$ in $\BAlg_\sx$ for some saturated algebra $A$, iff $D \cong \langle B^\perp,B\rangle$ in $\BAlg_\sx$ for some saturated algebra $B$.  Consequently, the full subcategory of $\BAlg_\sx$ spanned by right-commutant (resp.~left-commutant, commutant) bifold algebras is replete (i.e.~closed under isomorphism).  By \cite[9.4]{Lu:BAlgCmt} there is an equivalence between $\AlgS$ and the full subcategory $\RComBAlg_\sx$ of $\BAlg_\sx$ spanned by right-commutant bifold algebras, under which an algebra $A$ is sent to the bifold algebra $\langle A,A^\perp\rangle$ and, conversely, a right-commutant bifold algebra $D$ is sent to its left face $D_\ell$.  By \cite[10.10]{Lu:BAlgCmt}, the latter equivalence restricts to an equivalence $\SatAlgS \simeq \ComBAlg_\sx$ between the full, replete subcategories of $\AlgS$ and $\BAlg_\sx$ spanned by saturated algebras and commutant bifold algebras, respectively; in particular, an algebra $A$ is saturated iff $\langle A,A^\perp\rangle$ is commutant.  Similarly, there is an equivalence between $(\AlgS)^\op$ and the full subcategory $\LComBAlg_\sx$ of $\BAlg_\sx$ spanned by left-commutant bifold algebras \cite[9.4]{Lu:BAlgCmt}, given on objects by $B \mapsto \langle B^\perp, B\rangle$ and $D_r \mapsfrom D$, and this equivalence restricts also to an equivalence $\SatAlgS^\op \simeq \ComBAlg_\sx$ between saturated algebras and commutant bifold algebras \cite[10.10]{Lu:BAlgCmt}.

\begin{parasub}\label{para:one-commutative-face}
The following result is established in \cite[14.2, 14.3]{Lu:BAlgCmt}: Suppose that $(\T,\U,D)$ is a right-commutant bifold algebra whose left face $(\T,D_\ell)$ is commutative. Then every $\U$-algebra $B$ is the right face of a $(\T,\U)$-algebra. In more detail, there is a unique morphism $M:(\T,D_\ell) \to (\U,D_r)$ in $\AlgT(C)$, where $C = |D|$, and the morphism of $\J$-theories $M:\T \to \U$ witnesses that every $\U$-algebra $B$ has an underlying $\T$-algebra $BM$ for which $(BM,B)$ is a commuting $\T$-$\U$-algebra pair.
\end{parasub}
\end{para}

\section{Algebraic dual adjunctions and commuting algebra pairs}\label{sec:alg-dual-adj-comm-alg-pairs}

We write $\pCAlgPair{\T}{\U}$ for the ordinary category underlying the $\V$-category of commuting $\T$-$\U$-algebra pairs (\ref{para:cmt}). By \ref{para:cmt}, there is a functor
$$\pCAlgPair{-}{?}:(\ThJ^2)^\op \longrightarrow \CAT\;,$$
and we write $\CAlgPairO$ for the op-fibred category over $(\ThJ^2)^\op$ determined by this functor, calling $\CAlgPairO$ the \bit{(op-fibred) category of commuting algebra pairs}.  By \ref{para:cmt}, there is an equivalence $\pAlg{(-,\blanktwo)} \simeq \pCAlgPair{-}{\blanktwo}$ in the functor 2-category from $(\ThJ^2)^\op$ to $\CAT$, so (by \ref{thm:kelly_functor_from_lax_transf}) we obtain an equivalence
\begin{equation}\label{eq:balgo_equiv_cpairo}\BAlgO \;\simeq\; \CAlgPairO\end{equation}
under which each bifold algebra $D$ is sent to the commuting algebra pair $(D_\ell,D_r)$ consisting of its left and right faces, and each commuting algebra pair $(A,B)$ is sent to the bifold algebra $\langle A,B \rangle$ of \ref{para:cmt}.

Combining \eqref{eq:balgo_equiv_cpairo} with Theorem \ref{thm:bifolds_biequiv_ada}, we immediately obtain the following: 

\begin{thm}
There is a biequivalence
$$\CAlgPairO \;\;\simeq\;\;\ADAJ$$
between the (locally discrete) 2-category of commuting algebra pairs and the 2-category of $\J$-algebraic dual adjunctions, under which a commuting algebra pair $(A,B)$ is sent to the dualization adjunction for the bifold algebra $\langle A,B \rangle$, with the notation of \ref{para:cmt}.
\end{thm}

\begin{para}\label{para:dualn-adj-for-calgpair}
Given a commuting $\T$-$\U$-algebra pair $(A,B)$, we abuse notation by writing the dualization adjunction for $\langle A,B \rangle$ as $\Hom_\U(-,B) \dashv \Hom_\T(-,A):\Alg{\T}^\op \to \Alg{\U}$ or even as $\Alg{\U}(-,B) \dashv \Alg{\T}(-,A):\Alg{\T}^\op \to \Alg{\U}$.
\end{para}

\section{Stable algebraic dual adjunctions and commutant bifold algebras}\label{sec:stabl-alg-dual-adj}

\begin{defn}\label{defn:stable_jald_dual_adjn}
Let $\D = (\A,\B,\nabla \dashv \Delta)$ be a $\J$-algebraic dual adjunction, and write $F_\ell \dashv G_\ell:\A \to \V$ and $F_r \dashv G_r:\B \to \V$ for the associated $\V$-adjunctions (\ref{para:jalg}).  $\D$ is \bit{left-stable} if for every object $J$ of $\J \hookrightarrow \V$, the object $F_\ell J$ of $\A$ is $\nabla\Delta$-reflexive \pref{defn:alg_dual_terms}.  Similarly, $\D$ is \bit{right-stable} if for every object $J$ of $\J$, the object $F_rJ$ of $\B$ is $\Delta\nabla$-reflexive.  $\D$ is \bit{stable} if $\D$ is both left- and right-stable.
\end{defn}

We discuss several examples of these notions in Section \ref{sec:exa_ada}.

\begin{rem}\label{rem:jgen_free}
Let us say that an object $A$ of a $\J$-algebraic $\V$-category $\A$ over $\V$ is a \bit{$\J$-generated free object} of $\A$ if $\A(A,-) \cong \V(J,G-)$ for some object $J$ of $\J$, where $G:\A \to \V$ is the associated $\V$-functor. For example, by the Yoneda lemma, the $\J$-generated free objects of $\Alg{\T}$ for a $\J$-theory $\T$ are precisely the representables, i.e.~the $\T$-algebras isomorphic to $\T(J,-):\T \to \V$ for some $J \in \ob\J = \ob\T$. In the situation of \ref{defn:stable_jald_dual_adjn}, $\D$ is left-stable (resp.~right-stable) iff every $\J$-generated free object of $\A$ is $\nabla\Delta$-reflexive (resp.~every $\J$-generated free object of $\B$ is $\Delta\nabla$-reflexive).  In particular, Definition \ref{defn:stable_jald_dual_adjn} does not depend on the chosen left adjoints to $G_\ell$ and $G_r$.
\end{rem}

\begin{rem}\label{rem:stable_dcirc}
A $\J$-algebraic dual adjunction $\D$ is left-stable if and only if $\D^\circ$ \pref{eq:rev} is right-stable.
\end{rem}

Having established in Theorem \ref{thm:bifolds_biequiv_ada} a biequivalence between algebraic dual adjunctions $\D$ and bifold algebras $D$, we now show that (right- or left-)stable algebraic dual adjunctions correspond to (right- or left-)commutant bifold algebras, using the following observations.

\begin{rem}\label{rem:cmt_dcirc}
A bifold algebra $D$ is left-commutant if and only if $D^\circ$ \eqref{eq:transpose-of-bif-alg} is right-commutant. Indeed, the left and right faces of $D^\circ$ are the right and left faces of $D$, respectively.
\end{rem}

\begin{lem}\label{thm:lemma_on_adjns}
Let $F \dashv G:\E \rightarrow \C$ be a $\V$-adjunction, with counit $\varepsilon:FG \Rightarrow 1_\E$, and let $K:\K \rightarrow \E$ be a fully faithful, strongly cogenerating $\V$-functor \cite[\S 3.6]{Ke:Ba}.  Then the composite $GK:\K \rightarrow \C$ is fully faithful if and only if $\varepsilon K:FGK \Rightarrow K$ is an isomorphism.
\end{lem}
\begin{proof}
This can be proved by a straightforward variation on the proof in \cite[\S 1.11, (1.53)]{Ke:Ba} of the well-known result that $G$ is fully faithful iff $\varepsilon$ is an isomorphism.
\end{proof}

\begin{thm}\label{thm:bif_leftcmt_iff_dadj_left_stable}
A bifold algebra $(\T,\U,D)$ is left-commutant (resp.~right-commutant, resp.~commutant) if and only if its dualization adjunction is left-stable (resp.~right-stable, resp.~stable).
\end{thm}
\begin{proof}
By \ref{rem:stable_dcirc} and \ref{rem:cmt_dcirc}, it suffices to prove that the bifold algebra $D$ is left-commutant if and only if its dualization adjunction $\D = \Dualn(D)$ is left-stable. Let us write $\D$ as $\nabla \dashv \Delta:\Alg{\T}^\op \rightarrow \Alg{\U}$, so that $\Delta = \Hom_\T(-,D)$.  By \ref{para:cmtnt_bif_algs}, $D$ is left-commutant iff its transpose $D_{\ell r}:\T \to \Alg{\U}$ is fully faithful. But by \ref{thm:bifold-alg-corresp-to-dual-adjn}, $D_{\ell r}$ is isomorphic to the composite
$$\T \xrightarrow{Y^\op} \Alg{\T}^\op \xrightarrow{\Delta} \Alg{\U}\;.$$
Hence $D$ is left-commutant iff $\Delta Y^\op$ is fully faithful. But $Y:\T^\op \rightarrow \Alg{\T}$ is fully faithful and dense (\ref{para:talgs}), so $Y^\op$ is fully faithful and codense. Since $Y^\op$ is codense, $Y^\op$ is strongly cogenerating, by \cite[\S 5.3]{Ke:Ba}. Hence, by Lemma \ref{thm:lemma_on_adjns}, $\Delta Y^\op$ is fully faithful iff $e^\D Y^\op:\nabla\Delta Y^\op \Rightarrow Y^\op$ is an isomorphism, where $e^\D:\nabla\Delta \Rightarrow 1_{\A^\op}$ is the counit of $\D$.  By \ref{defn:alg_dual_terms}, $e^\D = (d^{\D^\circ})^\op$, so $D$ is left-commutant iff $d^{\D^\circ}_{YJ}:YJ \rightarrow \nabla\Delta YJ$ is an isomorphism in $\Alg{\T}$ for each object $J$ of $\T$, iff the representable $\V$-functors from $\T$ to $\V$ are $\nabla\Delta$-reflexive objects of $\Alg{\T}$. But by \ref{rem:jgen_free}, the $\J$-generated free objects of $\Alg{\T}$ are precisely the representables.
\end{proof}

\begin{defn}\label{defn:2cat_sada}
Let us write $\LSADAJ$, $\RSADAJ$, and $\SADAJ$ for the full sub-2-categories of $\ADAJ$ spanned by the left-stable (resp.~right-stable, resp.~stable) $\J$-algebraic dual adjunctions.  Let us also write $\LComBAlgO$, $\RComBAlgO$, and $\ComBAlgO$ for the full subcategories of $\BAlgO$ spanned by the left-commutant (resp.~right-commutant, resp.~commutant) bifold algebras, and regard these categories as locally discrete 2-categories.
\end{defn}

\begin{thm}\label{thm:bieq_combalgs_st_jalg_dualadjns}
The biequivalence in Theorem \ref{thm:bifolds_biequiv_ada} restricts to a biequivalence between the (locally discrete) full sub-2-category of $\BAlgO$ spanned by left-commutant (resp.~right-commutant, resp.~commutant) bifold algebras and the full sub-2-category of $\ADAJ$ spanned by left-stable (resp.~right-stable, resp.~stable) $\J$-algebraic dual adjunctions:
$$\LComBAlgO \simeq \LSADAJ,\;\RComBAlgO \simeq \RSADAJ,\;\ComBAlgO \simeq \SADAJ\;.$$
\end{thm}
\begin{proof}
Each of the given full sub-2-categories of $\ADAJ$ is clearly closed under equivalence, so the result follows from Theorems \ref{thm:bifolds_biequiv_ada} and \ref{thm:bif_leftcmt_iff_dadj_left_stable}.
\end{proof}

\section{Stable algebraic dual adjunctions induced by algebras}\label{sec:stable-ada-alg}

By \ref{para:cmtnt_bif_algs}, every algebra $A$ determines a right-commutant bifold algebra $\langle A,A^\perp\rangle$, whose left and right faces are isomorphic to $A$ and $A^\perp$, respectively. Recall from \ref{para:cmt} that if we write $A$ as $(\T,A)$, then we may write its commutant $A^\perp$ as $(\T^\perp_A,A^\perp)$. Hence, \ref{thm:dualn_adjn},  \ref{thm:rl_dualizer}, and \ref{thm:bif_leftcmt_iff_dadj_left_stable} entail the following:

\begin{cor}\label{thm:alg_ind_rstable_ada}
Every algebra $(\T,A)$ induces a right-stable $\J$-algebraic dual adjunction, namely the dualization adjunction $\Dualn \langle A,A^\perp\rangle$ induced by the bifold algebra $\langle A,A^\perp\rangle$. With the notation of \ref{para:dualn-adj-for-calgpair}, we may write the latter adjunction as
\begin{equation}\label{eq:rcmt_ada}\RAdjn{\Alg{\T}^\op}{\Hom_\T(-,A)}{\Hom_{\T^\perp_A}(-,A^\perp)}{}{}{\Alg{\T^\perp_A}}.\end{equation}
The left and right dualizers of \eqref{eq:rcmt_ada} are isomorphic to $A$ and $A^\perp$, respectively.  Similarly every algebra $(\T,A)$ induces a left-stable $\J$-algebraic dual adjunction, namely the dualization adjunction $\Dualn \langle A^\perp,A\rangle$ induced by $\langle A^\perp,A\rangle$.
\end{cor}

\begin{defn}\label{defn:dualn-adjn-induced-from-the-left}
Given an algebra $(\T,A)$, we call the right-stable $\J$-algebraic dual adjunction $\Dualn\langle A,A^\perp\rangle$ in \eqref{eq:rcmt_ada} the \bit{dualization adjunction induced from the left} by $(\T,A)$, while we call the left-stable $\J$-algebraic dual adjunction $\Dualn\langle A^\perp,A\rangle$ in \ref{thm:alg_ind_rstable_ada} the \bit{dualization adjunction induced from the right} by $(\T,A)$.  In view of \eqref{eq:balgo_equiv_cpairo} we know that $\langle A^\perp,A\rangle \cong \langle A,A^\perp\rangle^\circ$ in $\BAlgO$, so $\Dualn\langle A^\perp,A\rangle \cong \bigl(\Dualn \langle A,A^\perp\rangle\bigr)^\circ$ by \eqref{eq:dualnop}.
\end{defn}

By, \ref{para:cmtnt_bif_algs}, an algebra $A$ is saturated iff $\langle A,A^\perp\rangle$ is a commutant bifold algebra, iff $\langle A^\perp,A\rangle$ is a commutant bifold algebra.  Hence,  Theorem \ref{thm:bif_leftcmt_iff_dadj_left_stable} and Corollary \ref{thm:alg_ind_rstable_ada} entail the following result:

\begin{cor}\label{thm:sat_alg_ind_stable_ada}
Every saturated algebra $(\T,A)$ induces a stable $\J$-algebraic dual adjunction, namely the dualization adjunction induced from the left by $(\T,A)$.  Moreover, every saturated algebra $(\T,A)$ induces two stable $\J$-algebraic dual adjunctions, namely $\Dualn\langle A,A^\perp\rangle$ and $\Dualn\langle A^\perp,A\rangle$, the second of which is isomorphic to the opposite of the first.  The dualization adjunction induced from the left by an algebra $(\T,A)$ is stable if and only if $(\T,A)$ is saturated, if and only if the dualization adjunction induced from the right by $(\T,A)$ is stable.
\end{cor}

\begin{para}
Next we show that by assigning to each algebra $(\T,A)$ the dualization adjunction induced from the left by $(\T,A)$, we obtain a biequivalence between algebras and right-stable $\J$-algebraic dual adjunctions, which then restricts to a biequivalence between saturated algebras and stable $\J$-algebraic dual adjunctions.  By \ref{para:cmtnt_bif_algs}, we have an equivalence $\AlgS \simeq \RComBAlg_\sx$ given on objects by $A \mapsto \langle A,A^\perp\rangle$, while by \ref{thm:bieq_combalgs_st_jalg_dualadjns} we have a biequivalence $\RComBAlgO \simeq \RSADAJ$ given on objects by $D \mapsto \Dualn(D)$.  Clearly we cannot simply compose these equivalences, because the morphisms in $\RComBAlg_\sx$ are strong cross-morphisms while those in $\RComBAlgO$ are op-theoretic morphisms.  For this reason, we instead consider the groupoid cores of these categories, which we show are isomorphic.
\end{para}

\begin{defn}
We write $\AlgO^\si$ to denote the groupoid core of $\AlgO$, and we write $\BAlgO^\si$ to denote the groupoid core of $\BAlgO^\si$.  We call $\AlgO^\si$ the \bit{groupoid of algebras} and $\BAlgO^\si$ the \bit{groupoid of bifold algebras}.  We write $\SatAlgO^\si$ to denote the full subcategory of $\AlgO^\si$ spanned by the saturated algebras \pref{para:cmtnt_bif_algs}.
\end{defn}

\begin{lem}\label{thm:groupoids-of-algebras}
The groupoid of algebras $\AlgO^\si$ is isomorphic to the groupoid core $\AlgS^\si$ of the category $\AlgS$ of algebras with strong theoretic morphisms. The groupoid of bifold algebras $\BAlgO^\si$ is isomorphic to the groupoid core $\BAlg_\sx^\si$ of the category of bifold algebras and strong cross-morphisms.  Indeed, there are identity-on-objects isomorphisms
$$\AlgO^\si \;\cong\; \AlgS^\si,\;\;\;\;\BAlgO^\si \;\cong\; \BAlg_\sx^\si.$$
\end{lem}
\begin{proof}
Given an isomorphism $(M,f):(\T,A) \rightarrow (\U,B)$ in $\AlgO$, we know that $M:\U \rightarrow \T$ and $f:AM \Rightarrow B$ are invertible, so $fM^{-1}:A \Rightarrow BM^{-1}$ is invertible, and we obtain an isomorphism $(M^{-1},fM^{-1}):(\T,A) \rightarrow (\U,B)$ in $\AlgS$.  Thus we obtain an isomorphism $\AlgO^\si \cong \AlgS^\si$.  Given instead an isomorphism $(M,N,f):(\T,\U,D) \rightarrow (\T',\U',D')$ in $\BAlgO$, we know that $M:\T' \rightarrow \T$, $N:\U' \rightarrow \U$, and $f:D(M-,N\blanktwo) \Rightarrow D'(-,\blanktwo)$ are invertible, so the isomorphisms $\tilde{f}_{JK} = f_{M^{-1}J,K}:D(J,NK) \rightarrow D'(M^{-1}J,K)$ in $\V$ $(J \in \T,K \in \U')$ provide an isomorphism $(M^{-1},N,\tilde{f}):(\T,\U,D) \rightarrow (\T',\U',D')$ in $\BAlg_\sx$, and the result follows.
\end{proof}

\begin{defn}
A \textit{bigroupoid} \cite{HKK:HomotopyBigroupoid} is a bicategory in which every 1-cell is an equivalence and every 2-cell is invertible. A \textit{bigroupoidal 2-category}, or (for brevity) a \mbox{\bit{2-groupoid}}, is a 2-category $\K$ that (qua bicategory) is a bigroupoid. Given an arbitrary 2-category $\K$, there is a sub-2-category
$$\K^\seq$$
of $\K$ whose objects are the same as those of $\K$, whose 1-cells are precisely those that are equivalences in $\K$, and whose 2-cells are those that are invertible in $\K$.  The resulting 2-category $\K^\seq$ is a 2-groupoid, and we call it the \bit{2-groupoid core} of $\K$.
\end{defn}

\begin{defn}
The \bit{2-groupoid of right-stable $\J$-algebraic dual adjunctions} is defined as the 2-groupoid core $\RSADAJ^\seq$ of $\RSADAJ$, with the notation of \ref{defn:2cat_sada}.  Similarly, the \bit{2-groupoid of left-stable (resp.~stable) $\J$-algebraic dual adjunctions} is defined as $\LSADAJ^\seq$ (resp.~$\SADAJ^\seq$).
\end{defn}

In the following, we implicitly regard groupoids as locally discrete 2-groupoids.

\begin{thm}\label{thm:biequiv_algs_rsada_lsada}
There is a biequivalence between the groupoid of algebras, $\AlgO^\si$, and the 2-groupoid of right-stable $\J$-algebraic dual adjunctions, $\RSADAJ^\seq$, under which an algebra $(\T,A)$ is sent to the dualization adjunction induced from the left by $(\T,A)$:
$$\AlgO^\si\;\;\simeq\;\;\RSADAJ^\seq\;.$$
Similarly, there is a biequivalence
$$\AlgO^\si\;\;\simeq\;\;\LSADAJ^\seq\;,$$
under which an algebra $(\T,A)$ is sent to the dualization adjunction induced from the right by $(\T,A)$.
\end{thm}
\begin{proof}
By \ref{thm:bieq_combalgs_st_jalg_dualadjns}, there is a biequivalence $\RComBAlgO \simeq \RSADAJ$ given on objects by $D \mapsto \Dualn D$, which therefore restricts to a biequivalence $\RComBAlgO^\si \simeq \RSADAJ^\seq$.  But by \ref{thm:groupoids-of-algebras} we also have an identity-on-objects isomorphism $\BAlg_\sx^\si \cong \BAlgO^\si$, which restricts to an isomorphism $\RComBAlg_\sx^\si \cong \RComBAlgO^\si$.  Furthermore, by \ref{para:cmtnt_bif_algs} there is an equivalence $\AlgS \simeq \RComBAlg_\sx$, which therefore restricts to an equivalence $\AlgS^\si \simeq \RComBAlg_\sx^\si$, given on objects by $A \mapsto \langle A,A^\perp\rangle$, while by \ref{thm:groupoids-of-algebras} we also have an identity-on-objects isomorphism $\AlgO^\si \cong \AlgS^\si$.  Hence
$$\AlgO^\si \;\cong\; \AlgS^\si \;\simeq\; \RComBAlg_\sx^\si \;\cong\; \RComBAlgO^\si \;\simeq\; \RSADAJ^\seq\;.$$
By \ref{para:cmtnt_bif_algs} we also have an equivalence $(\AlgS)^\op \simeq \LComBAlg_\sx$, which restricts to an equivalence $(\AlgS^\si)^\op \simeq \LComBAlg_\sx^\si$, so since $\AlgS^\si$ is a groupoid and so is self-dual, we find that
$$\AlgO^\si \;\cong\; \AlgS^\si \;\cong\; (\AlgS^\si)^\op\;\simeq\; \LComBAlg_\sx^\si \;\cong\; \LComBAlgO^\si \;\simeq\; \LSADAJ^\seq$$
by \ref{thm:groupoids-of-algebras} and \ref{thm:bieq_combalgs_st_jalg_dualadjns}.
\end{proof}

\begin{cor}\label{thm:every_rst_ada_ind_by_alg}
A $\J$-algebraic dual adjunction is right-stable (resp.~left-stable) if and only if it is equivalent, in $\ADAJ$, to the dualization adjunction induced from the left (resp.~from the right) by an algebra. In that case, the latter algebra is unique up to isomorphism in $\AlgO$.
\end{cor}

\begin{thm}
There is a biequivalence
$$\SatAlgO^\si \;\;\simeq\;\; \SADAJ^\seq$$
between the groupoid of saturated algebras and the 2-groupoid of stable $\J$-algebraic dual adjunctions.  Indeed, one such biequivalence sends each saturated algebra $(\T,A)$ to the dualization adjunction induced from the left by $(\T,A)$, while another sends each saturated algebra $(\T,A)$ to the dualization adjunction induced from the right by $(\T,A)$.
\end{thm}
\begin{proof}
By Corollary \ref{thm:sat_alg_ind_stable_ada}, the biequivalences in Theorem \ref{thm:biequiv_algs_rsada_lsada} restrict as needed.
\end{proof}

\begin{cor}
A $\J$-algebraic dual adjunction is stable if and only if it is equivalent, in $\ADAJ$, to the dualization adjunction induced from the left by some saturated algebra, if and only if it is equivalent to the dualization adjunction induced from the right by some saturated algebra.  
\end{cor}

\section{Examples of algebraic dual adjunctions}\label{sec:exa_ada}

\subsection{Returning to a basic example: Bimodules}\label{exa:dln_mod}

As discussed in Section \ref{exa:bim}, if $R$ and $S$ are monoids in $\V$ and $D$ is an $R$-$S$-bimodule in $\V$, then we obtain an $\{I\}$-algebraic dual adjunction
\begin{equation}\label{eq:dualn_adj_bimod}\RAdjn{\Mod{R}^\op}{\Hom_R(-,D)}{\Hom_{S^\op}(-,D)}{}{}{\Mod{S^\op}}\end{equation}
for the system of arities $\{I\} \hookrightarrow \V$. Let us consider the special case of $\V = \Ab$ and the system of arities $\{\ZZ\} \hookrightarrow \Ab$. As discussed in \ref{cor:dual_adj_mod_cats}, a ring $R$ is equivalently a  $\{\ZZ\}$-theory, and there is an isomorphism $\Bimod(\Ab) \cong \BAlg_{\{\ZZ\}}$ under which a bimodule $(R,S,D)$ (consisting of rings $R$, $S$ and an $R$-$S$-bimodule $D$) corresponds to a bifold algebra $(R,S^\op,D)$ in which we regard the rings $R$ and $S^\op$ as $\{\ZZ\}$-theories. By Theorem \ref{thm:bif_leftcmt_iff_dadj_left_stable}, the $\{\ZZ\}$-algebraic dual adjunction \eqref{eq:dualn_adj_bimod} is left-stable (resp.~right-stable) iff the bifold algebra $(R,S^\op,D)$ is left-commutant (resp.~right-commutant), which as noted in \cite[16.1]{Lu:BAlgCmt} requires precisely that the canonical ring homomorphism $R \to \End(D_S)$ (resp.~$S^\op \to \End(\tensor[_R]{D}{})$) be an isomorphism, where we write $D_S$ (resp.~$\tensor[_R]{D}{}$) for the right $S$-module (resp.~left $R$-module) underlying $D$, and we write $\End(D_S)$ (resp.~$\End(\tensor[_R]{D}{})$) for its endomorphism ring.

In particular, we may regard every ring $R$ as an $R$-$R$-bimodule $\tensor[_R]{R}{_R}$, and it is noted in \cite[16.1]{Lu:BAlgCmt} that $(R,R^\op,\tensor[_R]{R}{_R})$ is a commutant bifold algebra; hence, by Theorem \ref{thm:bif_leftcmt_iff_dadj_left_stable}, its dualization adjunction $\Hom_{R^\op}(-,R) \dashv \Hom_R(-,R):\Mod{R}^\op \rightarrow \Mod{R^\op}$ is a stable $\{\ZZ\}$-algebraic dual adjunction.

Given a ring $R$ and a left $R$-module $A$, we obtain a right-commutant bifold algebra $\langle A,A^\perp \rangle$ by \ref{para:cmtnt_bif_algs}, which by \cite[16.1]{Lu:BAlgCmt} is precisely the triple $(R,S^\op,\tensor[_R]{A}{_S})$ where $S^\op = \End(A)$ is the endomorphism ring of the left $R$-module $A$. Hence, by Theorem \ref{thm:bif_leftcmt_iff_dadj_left_stable}, we obtain a right-stable $\{\ZZ\}$-algebraic dual adjunction
$$\RAdjn{\Mod{R}^\op}{\Hom_R(-,A)}{\Hom_{\End(A)}(-,A)}{}{}{\Mod{\End(A)}},$$
namely the dualization adjunction induced from the left by $A$ (\ref{defn:dualn-adjn-induced-from-the-left}). By the above, this right-commutant bifold algebra $(R,S^\op,\tensor[_R]{A}{_S})$ (with $S^\op = \End(A)$) is a commutant bifold algebra iff it is left-commutant, iff the canonical ring homomorphism $R \to \End(A_S)$ is an isomorphism, noting that the right $S$-module $A_S$ is equivalently the left $\End(A)$-module $\tensor[_{\End(A)}]{A}{}$, so that $\End(A_S)$ is the ring of left $\End(A)$-linear endomorphisms of $A$.

\subsection{Strongly finitary algebraic dual adjunctions}\label{sec:str-fin-alg-dual-adjn}

In this subsection, we let $\V$ be a complete and cocomplete cartesian closed category. Recall from Example \ref{exa:sys-ar}(6) that there is a subcategory of arities $\SF(\V) \hookrightarrow \V$ spanned by the finite copowers $n \cdot 1$ $(n \in \NN)$ of the terminal object $1$, for which $\SF(\V)$-ary $\V$-monads are precisely \textit{strongly finitary $\V$-monads} on $\V$ in the sense of \cite{KeLa:SF}, and for which $\SF(\V)$-theories are the enriched algebraic theories of Borceux and Day \cite{BoDay}. Following \cite{Lu:EnrAlgTh,Lu:FDistn,Lu:Cmt,Lu:BAlgCmt}, we may take the objects of $\SF(\V)$ to be the natural numbers $n \in \NN$, so that the fully faithful $\V$-functor $\SF(\V) \hookrightarrow \V$ is given on objects by $n \mapsto n \cdot 1$, and an $\SF(\V)$-theory is given by a $\V$-category $\T$ with $\ob\T = \NN$ in which each object $n \in \NN$ carries the structure of a conical $n$-th power of the object $1 \in \NN$ \cite[\S 3.2]{Lu:FDistn}.

By definition, a \bit{strongly finitary algebraic dual $\V$-adjunction} is an $\SF(\V)$-algebraic dual adjunction. In this subsection, we consider examples of strongly finitary algebraic dual $\V$-adjunctions whose corresponding bifold algebras were discussed \cite{Lu:FDistn}, \cite{Lu:CvxAffCmt}, \cite[16.2]{Lu:BAlgCmt}, so we follow these papers quite closely in discussing the aspects of these examples concerning the bifold algebras (while the associated dual adjunctions were \textit{not} treated in these papers).

\begin{parasub}\label{para:cmonv-abv}
As in \cite{Lu:FDistn,Lu:BAlgCmt}, we write $\CMon(\V)$ and $\Ab(\V)$ for the symmetric monoidal closed categories of (internal) commutative monoids in $\V$ and abelian groups in $\V$, respectively \cite[6.3.2--6.3.7]{Lu:FDistn}; for example, $\CMon(\Set) = \CMon$ is the category of commutative monoids, written additively, with the tensor product of commutative monoids. These monoidal categories $\CMon(\V)$ and $\Ab(\V)$ underlie symmetric monoidal closed $\V$-categories, $\sCMon(\V)$ and $\sAb(\V)$, which are $\SF(\V)$-algebraic over $\V$ \cite[6.3.2--6.3.7]{Lu:FDistn}. An \bit{(internal) rig} in $\V$ is by definition an internal unital semiring $R = (R,+,\cdot,0,1)$ in $\V$; see \cite[2.7]{Lu:FDistn} for an explicit definition. In particular, a rig in $\Set$ (or simply a \textit{rig}) is a unital semiring in the usual sense and is equivalently given by a monoid in the monoidal category $\CMon = \CMon(\Set)$. More generally, a rig in $\V$ is equivalently given by a monoid in the monoidal category $\CMon(\V)$ \cite[6.4.1]{Lu:FDistn}. A left (resp.~right) $R$-module for an internal rig $R$ is given by a left (resp.~right) $R$-module for the monoid $R$ in $\CMon(\V)$ \cite[6.4.1]{Lu:FDistn}. An \textit{(internal) ring} in $\V$ is equivalently a monoid in $\Ab(\V)$, equivalently, a rig in $\V$ whose underlying monoid is an internal group; see \cite[2.7, 6.4.1]{Lu:FDistn}.
\end{parasub}

\begin{exasub}[\textbf{Dualization of modules for internal rings and rigs in $\V$}]\label{exa:dualn_mod}
Let $R$ be a rig in $\V$. Since $\CMon(\V)$ is a complete symmetric monoidal closed category and $R$ is a monoid in $\CMon(\V)$, we obtain a $\CMon(\V)$-category $\Mod{R}$ by applying \ref{exa:sys-ar}(4) for the base of enrichment $\CMon(\V)$. As discussed in \cite[16.2.1]{Lu:BAlgCmt}, $\Mod{R}$ has an underlying $\V$-category that we denote also by $\Mod{R}$ (or $\Mod{R}(\V)$) and that is obtained by change of base along the forgetful functor $\CMon(\V) \to \V$, which underlies a symmetric monoidal functor. By \cite[6.4.5]{Lu:FDistn}, there is an $\SF(\V)$-theory with $\Mat(R)(m,n) = R^{n \times m}$ in $\V$ for all $m,n \in \NN$, and with composition given by matrix multiplication, and $\Mod{R}$ is isomorphic  to the $\V$-category of normal $\Mat(R)$-algebras, i.e. $\Mod{R} \cong \Alg{\Mat(R)}^{\,!}$ in $\VCAT_0 \slash \V$, so $\Mod{R}$ is strictly $\SF(\V)$-algebraic over $\V$.

Let $D$ be an \textit{$R$-$S$-bimodule} for rigs $R$ and $S$ in $\V$, equivalently, an $R$-$S$-bimodule for monoids $R$ and $S$ in $\CMon(\V)$. As a special case of \ref{exa:dln_mod} we obtain a $\CMon(\V)$-adjunction $\Hom_{S^\op}(-,D) \dashv \Hom_R(-,D):\Mod{R}^\op \rightarrow \Mod{S^\op}$, so by change of base this adjunction has an underlying $\V$-adjunction, for which we use the same notation; but the $\V$-categories $\Mod{R}$ and $\Mod{S^\op}$ are $\SF(\V)$-algebraic over $\V$, so this $\V$-adjunction is an $\SF(\V)$-algebraic dual adjunction.

We now consider the special case where $R = S$ and we take $D = R$, regarded as an $R$-$R$-bimodule. The left (resp.~right) $R$-module underlying $R$ can be regarded equivalently as a normal $\Mat(R)$-algebra (resp.~$\Mat(R^\op)$-algebra) that we write also as $R$.  As discussed in \cite[16.2]{Lu:BAlgCmt}, the resulting algebras $(\Mat(R),R)$ and $(\Mat(R^\op),R)$ are commutants of one another and so are the left and right faces of a commutant bifold algebra, which we write as $(\Mat(R),\Mat(R^\op),R)$.  Therefore, we obtain a stable $\SF(\V)$-algebraic dual adjunction
\begin{equation}\label{eq:dualn_adj_for_r}\RAdjn{\Mod{R}^\op}{\Hom_R(-,R)}{\Hom_{R^\op}(-,R)}{}{}{\Mod{R^\op}}.\end{equation}
Of course, if $R$ is commutative, then $R^\op = R$. Even in this simple special case, the notion of dualization that arises from the dual adjunction $\Hom_R(-,R) \dashv \Hom_R(-,R):\Mod{R}^\op \rightarrow \Mod{R}$ for an internal commutative ring $R$ in $\V$ is of interest in analysis for certain choices of $\V$, as we now illustrate via two examples.

Firstly, take $\V = \Conv$ (the category of convergence spaces, \ref{para:specific-cccs}) and $R = \RR$, the commutative ring of reals with its usual topology, regarded as a topological ring and so as a ring in $\Conv$. In this case, $\Mod{\RR} = \Mod{\RR}(\Conv)$ is the category of \bit{convergence vector spaces} \cite{BeBu}, into which the category of topological $\RR$-vector spaces embeds as a full subcategory. We have a stable $\SF(\Conv)$-algebraic dual adjunction
\begin{equation}\label{eq:dualn-convergence-vect-sp}\Hom_\RR(-,\RR) \dashv \Hom_\RR(-,\RR):\Mod{\RR}(\Conv)^\op \rightarrow \Mod{\RR}(\Conv)\;.\end{equation}
Moreover, there are broad classes of examples of convergence vector spaces $E$ that are reflexive with respect to this dual adjunction, meaning that the canonical morphism $d_E:E \to \Hom_\RR(\Hom_\RR(E,\RR),\RR)$ is an isomorphism in $\Mod{\RR}(\Conv)$. For example, Butzmann \cite{Bu} proved that if $E$ is a Hausdorff locally convex topological vector space, then $d_E$ presents $\Hom_\RR(\Hom_\RR(E,\RR),\RR)$ as a completion of $E$, so that in this case $E$ is reflexive with respect to \eqref{eq:dualn-convergence-vect-sp} iff $E$ is complete; see \cite[4.3.20, 4.3.21]{BeBu}. As another example, recall from \ref{para:specific-cccs} that if $X$ is a given convergence space then the internal hom $C_c(X,\RR)$ in $\Conv$ is the set  of all continuous functions from $X$ to $\RR$, under continuous convergence. This convergence space $C_c(X,\RR)$ underlies a convergence vector space, and it is a theorem of Butzmann \cite{Bu} that $C_c(X,\RR)$ is reflexive, with respect to the dual adjunction \eqref{eq:dualn-convergence-vect-sp}; see \cite[4.2.4]{BeBu}. Its dual $\Hom_\RR(C_c(X,\RR),\RR)$ (with respect to \ref{eq:dualn-convergence-vect-sp}) is of interest in its own right: For example, in the special case where $X$ is a locally compact Hausdorff topological space, the elements of $\Hom_\RR(C_c(X,\RR),\RR)$ are the compactly supported $\RR$-valued Radon measures on $X$; see \cite[5.1]{Lu:FDistn}.

Secondly, we mention in passing that in the case where $\V = \Cah$ is the Cahiers topos (\ref{para:specific-cccs}) and $R$ is the line object in $\V$, if $X$ is a smooth manifold, regarded as an object of $\V$, and we write $R^X$ as the $\V$-enriched power of $R$ by $X$ in $\Mod{R}(\V)$, then the global elements of the dual $\Hom_R(R^X,R)$ are equivalently given by compactly supported $\RR$-valued Schwartz distributions on $X$ \cite[5.3]{Lu:FDistn}. A similar result in the case where $\V$ is instead the category of all $\Set$-valued functors on the category of finitely generated $C^\infty$-rings had been given in \cite[II.3.6]{MoeRey}.
\end{exasub}

\begin{exasub}[\textbf{Pontryagin duality and Binz-Butzmann duality}]\label{exa:binz-butmzmann-duality}
In this example, we take $\V$ to be the cartesian closed category of convergence spaces, $\Conv$ (\ref{para:specific-cccs}).  A \bit{convergence abelian group} is by definition an abelian group in $\Conv$.  For example, since the inclusion $\Top \hookrightarrow \Conv$ preserves limits (\ref{para:specific-cccs}), every topological abelian group is an example of a convergence abelian group. Using the term \bit{locally compact group} as a synonym for the notion of locally compact Hausdorff topological abelian group, let us recall that the \textit{Pontryagin duality} is an equivalence between the category of locally compact groups and its opposite, under which each locally compact group $A$ corresponds to its \bit{Pontryagin dual}, i.e.~the group of all continuous group homomorphisms from $A$ to the circle group $T$, regarded as a topological group under the compact-open topology; see \cite{Mor,Roe}. The Pontryagin duality was extended by Binz and Butzmann to an adjunction between the category of convergence abelian groups and its opposite \cite{Binz,Bu:DualConvGrp}.  By \ref{para:cmonv-abv}, the category of convergence abelian groups, $\Ab(\Conv)$, underlies a symmetric monoidal closed $\Conv$-category, $\sAb(\Conv)$.  Hence, by regarding $T$ as an object of $\Ab(\Conv)$, we obtain a $\Conv$-adjunction
\begin{equation}\label{eq:bi_bu_duality}\RAdjn{\sAb(\Conv)^\op}{\Hom(-,T)}{\Hom(-,T)}{}{}{\sAb(\Conv)},\end{equation}
where for each convergence abelian group $A$ we write $\Hom(A,T)$ for the internal hom from $A$ to $T$ in $\Ab(\Conv)$, which is the convergence abelian group of continuous group homomorphisms from $A$ to $T$, with the convergence structure given by continuous convergence (\ref{para:specific-cccs}).  By \ref{para:cmonv-abv}, $\sAb(\Conv)$ is $\SF(\Conv)$-algebraic over $\Conv$, so this adjunction \eqref{eq:bi_bu_duality} is an $\SF(\Conv)$-algebraic dual adjunction, which we now show is stable.  

If $A$ is any locally compact group (i.e.~locally compact Hausdorff topological abelian group), then $\Hom(A,T)$ carries (the convergence structure determined by) the compact-open topology, by \cite{Binz}, so that $\Hom(A,T)$ is the Pontryagin dual of $A$.  Hence, by Pontryagin duality, the canonical homomorphism $A \rightarrow \Hom(\Hom(A,T),T)$ is an isomorphism of topological groups and is therefore an isomorphism in $\Ab(\Conv)$.  Therefore, every locally compact group is reflexive with respect to the $\SF(\Conv)$-algebraic dual adjunction \eqref{eq:bi_bu_duality}, which therefore restricts to a dual equivalence between the full subcategories of locally compact groups, namely the Pontryagin duality.  Writing $\ZZ$ for the discrete topological group of integers, this object $\ZZ$ of $\Ab(\Conv)$ is free on the terminal object $1$ of $\Conv$.  Hence, since finite products in $\Ab(\Conv)$ are biproducts (as in any additive category), we deduce that the $\SF(\Conv)$-generated free objects of $\sAb(\Conv)$ (in the sense of \ref{rem:jgen_free}) are those isomorphic to the discrete groups $\ZZ^n$, each of which is a locally compact group and so is reflexive with respect to the $\SF(\Conv)$-algebraic dual adjunction \eqref{eq:bi_bu_duality}, which is therefore stable. The commutant bifold algebra to which this adjunction corresponds via Theorem \ref{thm:bieq_combalgs_st_jalg_dualadjns} is discussed in \cite[16.2.6]{Lu:BAlgCmt}.
\end{exasub}

In the following two examples, we return to the setting of an arbitrary complete and cocomplete cartesian closed category $\V$.

\begin{exasub}[\textbf{Internal left $R$-affine spaces and pointed right $R$-modules}]\label{exa:r_aff}
Given an internal rig $R$ in $\V$, we now recall an example of a bifold algebra from \cite[\S 9]{Lu:FDistn} and \cite[16.2.2]{Lu:BAlgCmt}, and we discuss the $\SF(\V)$-algebraic dual adjunction that it induces. Recall that the $\SF(\V)$-theory $\Mat(R)$ of \ref{exa:dualn_mod} has hom-objects $\Mat(R)(m,n) = R^{n \times m}$ $(m,n \in \NN)$. As discussed in \cite[8.4]{Lu:FDistn}, there is a subtheory $\Mat^\aff(R) \hookrightarrow \Mat(R)$, called the \textit{affine core} of $\Mat(R)$, in which $\Mat^\aff(R)(m,n) \hookrightarrow R^{n \times m}$ is the subobject carved out by the equations that assert that each row has sum $1$. A \bit{(left) $R$-affine space in $\V$} is, by definition, a normal $\Mat^\aff(R)$-algebra in $\V$; see \cite[8.5]{Lu:FDistn} for an explicit description of the structure carried by an $R$-affine space in $\V$. We write $\sAff{R} = \Alg{\Mat^\aff(R)}^{\,!}$ for the $\V$-category of $R$-affine spaces. The subtheory embedding $\Mat^\aff(R) \hookrightarrow \Mat_R$ induces a $\V$-functor $\Mod{R} \to \sAff{R}$ witnessing that every left $R$-module has an underlying $R$-affine space.

On the other hand, a \bit{pointed left $R$-module} is by definition a left $R$-module $B$ equipped with a a designated morphism $*:1 \to B$ in $\V$. As discussed in \cite[9.2]{Lu:FDistn} and \cite[16.2.2]{Lu:BAlgCmt}, pointed left $R$-modules are the objects of a $\V$-category that is written as $\Mod{R}_*$ or $\Mod{R}_*(\V)$ and that is isomorphic to the $\V$-category $\Alg{\Mat^*(R)}$ of normal $\Mat^*(R)$-algebras for an $\SF(\V)$-theory $\Mat^*(R)$. A \bit{pointed right $R$-module} is by definition a pointed left $R^\op$-module.

We may regard $R$ as a left $R$-module and so as an $R$-affine space, and we may also regard $R$ as a pointed right $R$-module by equipping $R$ with the multiplicative identity $1 \to R$. Thus we obtain a pair of algebras that we may write as $(\Mat^\aff(R),R)$ and $(\Mat^*(R^\op),R)$, and it is shown in \cite[9.3]{Lu:FDistn} that the former is necessarily the commutant of the latter (in the sense discussed in \ref{para:cmtnt_bif_algs}). Hence, as discussed in \cite[16.2.2]{Lu:BAlgCmt}, the object $R$ of $\V$ underlies a left-commutant bifold algebra $(\Mat^\aff(R),\Mat^*(R^\op),R)$. Therefore, by Theorem \ref{thm:bif_leftcmt_iff_dadj_left_stable} we obtain a left-stable $\SF(\V)$-algebraic dual adjunction
\begin{equation}\label{eq:dualn_adj_aff_rmodstar}\RAdjn{\sAff{R}^\op}{\sAff{R}(-,R)}{\Mod{R^\op}_*(-,R)}{}{}{\Mod{R^\op}_*}\;.\end{equation}
If $R$ is an \textit{ring} in $\V$, then $(\Mat^\aff(R),\Mat^*(R^\op),R)$ is a commutant bifold algebra \cite[16.2.2]{Lu:BAlgCmt}, so \eqref{eq:dualn_adj_aff_rmodstar} is a stable $\SF(\V)$-algebraic dual adjunction, by Theorem \ref{thm:bif_leftcmt_iff_dadj_left_stable}. If $R$ is, moreover, a commutative ring in $\V$, then $R = R^\op$, and the $\SF(\V)$-theory $\Mat^\aff(R)$ is commutative \cite[8.7]{Lu:FDistn}, so by \ref{para:one-commutative-face} every pointed $R$-module $B$ underlies a bifold algebra $(\Mat^\aff(R),\Mat^*(R),B)$ (as noted in \cite[16.2.2]{Lu:BAlgCmt}), which therefore induces an $\SF(\V)$-algebraic dual adjunction 
$$\RAdjn{\sAff{R}^\op}{\sAff{R}(-,B)}{\Mod{R}_*(-,B)}{}{}{\Mod{R}_*}\;.$$

Among internal rigs $R$ that are not rings, there are some for which the bifold algebra $(\Mat^\aff(R),\Mat^*(R^\op),R)$ is \textit{not} right-commutant \cite[16.2.2, 16.2.4]{Lu:BAlgCmt}, as we discuss in Example \ref{exa:dualn_semilattices}, while there some for which it is, as we discuss next.
\end{exasub}

\begin{exasub}[\textbf{$R$-convex spaces and pointed $R_+$-modules}]\label{exa:cvx}
We now discuss a class of examples of commutant bifold algebras from \cite[\S 10]{Lu:FDistn} and \cite[16.2.3]{Lu:BAlgCmt}, and we consider the associated $\SF(\V)$-algebraic dual adjunctions. By definition, a \bit{preordered ring} in $\V$ is an internal ring $R$ in $\V$ equipped with a subrig $R_+ \hookrightarrow R$, called the \textit{positive part} of $R$. (We mention in passing that the pullback of $R_+ \hookrightarrow R$ along the composite $R \times R \xrightarrow{s} R \times R \xrightarrow{-} R$ is then an internal preorder in $\V$, where $s$ is the symmetry and $-$ is the subtraction morphism.) A \bit{(left) $R$-convex space} (in $\V$) is by definition a (left) $R_+$-affine space. We write $\Cvx{R}$ or $\Cvx{R}(\V)$ for the $\V$-category $\sAff{R_+}$ of $R$-convex spaces, which is a strictly $\SF(\V)$-algebraic $\V$-category over $\V$, described by the $\SF(\V)$-theory $\Mat^\aff(R_+)$, by \ref{exa:r_aff}. In view of \ref{exa:r_aff}, $R_+$ carries the structure of a left-commutant bifold algebra $(\Mat^\aff(R_+),\Mat^*(R_+^\op),R_+)$, and we obtain a left-stable $\SF(\V)$-algebraic dual adjunction
\begin{equation}\label{eq:cvx_adj}\RAdjn{\Cvx{R}^\op}{\Cvx{R}(-,R_+)}{\Mod{R_+^\op}_*(-,R_+)}{}{}{\Mod{R_+^\op}_*}.\end{equation}
Conditions that entail that $(\Mat^\aff(R_+),\Mat^*(R_+^\op),R_+)$ is a commutant bifold algebra are given in \cite[10.6]{Lu:FDistn} and discussed in \cite[16.2.3]{Lu:BAlgCmt}, building on \cite[10.20]{Lu:CvxAffCmt}. Here we simply recall from these papers three specific examples satisfying these conditions, and in each of these examples the preordered ring $R$ is also commutative, so that its subrig $R_+$ is also commutative.

Firstly, the familiar ordered ring of reals $\RR$ in $\V = \Set$ has the property that $(\Mat^\aff(\RR_+),\Mat^*(\RR_+),\RR_+)$ is a commutant bifold algebra, by \cite[10.20, 10.21]{Lu:CvxAffCmt}, as discussed in \cite[16.2.3]{Lu:BAlgCmt}, so that \eqref{eq:cvx_adj} is a stable finitary algebraic dual adjunction (i.e.~a stable $\SF(\Set)$-algebraic dual adjunction) in this case, by Theorem \ref{thm:bif_leftcmt_iff_dadj_left_stable}.

Secondly, if we take $\V = \Conv$ to be the category of convergence spaces (\ref{para:specific-cccs}) and we take $R = \RR$ to be the topological ring of reals, regarded as a preordered ring in $\Conv$, then $(\Mat^\aff(\RR_+),\Mat^*(\RR_+),\RR_+)$ is a commutant bifold algebra (for the system of arities $\SF(\V)$), as discussed in \cite[10.9]{Lu:FDistn} and \cite[16.2.3]{Lu:BAlgCmt}, so that \eqref{eq:cvx_adj} is a stable $\SF(\V)$-algebraic dual adjunction in this case, by Theorem \ref{thm:bif_leftcmt_iff_dadj_left_stable}. We refer to $\RR$-convex spaces in $\V = \Conv$ as \bit{convergence convex spaces} \cite[8.10]{Lu:FDistn}. The notions of dualization that arise from the dual adjunction \eqref{eq:cvx_adj} in this case are of interest from the perspective of analysis: For example, if $X$ is a locally compact Hausdorff topological space, regarded as an object of $\Conv$, then the internal hom $C_c(X,\RR_+)$ in $\Conv$ underlies a pointed $\RR_+$-module in $\Conv$, whose designated point is the constant function $1$, and whose dual is the convergence convex space $\Mod{\RR_+}_*(\Conv)(C_c(X,\RR_+),\RR_+)$ of all continuous $\RR_+$-linear maps $\mu:C_c(X,\RR_+) \to \RR_+$ that preserve $1$, which are in bijective correspondence with compactly supported Radon probability measures on $X$ \cite[5.5]{Lu:FDistn}.

Thirdly, if we take $\V = \Cah$ to be the Cahiers topos (\ref{para:specific-cccs}) and we take $R$ to be the line object in $\V$, then $R$ carries the structure of a commutative preordered ring in $\V$ such that $(\Mat^\aff(R_+),\Mat^*(R_+),R_+)$ is a commutant bifold algebra, as discussed in \cite[10.11]{Lu:FDistn} and \cite[16.2.3]{Lu:BAlgCmt}, so that \eqref{eq:cvx_adj} is a stable $\SF(\V)$-algebraic dual adjunction in this case.

In these three cases, since the $\SF(\V)$-theory $\Mat^\aff(R_+)$ is commutative (as discussed in \ref{exa:r_aff}), we deduce by \ref{para:one-commutative-face} that \textit{every} pointed $R_+$-module $B$ is the right face of a bifold algebra $(\Mat^\aff(R_+),\Mat^*(R_+),B)$ and so induces an $\SF(\V)$-algebraic dual adjunction 
$$\RAdjn{\Cvx{R}^\op}{\Cvx{R}(-,B)}{\Mod{R_+}_*(-,B)}{}{}{\Mod{R_+}_*}.$$
\end{exasub}

\begin{exasub}[\textbf{Dualization of several kinds of semilattices}]\label{exa:dualn_semilattices}
Taking $\V = \Set$ and $\J = \FinCard \hookrightarrow \Set$, we now discuss some examples of bifold algebras that involve semilattices and are discussed in \cite[\S 8]{Lu:CvxAffCmt} and \cite[16.2.4]{Lu:BAlgCmt}, and we consider the associated finitary algebraic dual adjunctions (i.e. $\FinCard$-algebraic dual adjunctions). Following \cite{Joh:StSp}, we use the term \textit{meet semilattice} to mean a partially ordered set with finite meets (including a top element $\top$ as the meet of the empty set), while we use the term \textit{binary-meet semilattice} to refer to posets in which every pair of elements has a meet. We write $\SLat_{\wedge\top}$ (resp $\SLat_{\wedge}$) for the category of meet semilattices (resp.~binary-meet semilattices) and maps preserving finite meets (resp.~binary meets).

The set $2 = \{0,1\}$ is a distributive lattice and so carries two different commutative rig structures, depending on which of the lattice operations $\vee$ or $\wedge$ we take as addition.  Taking $\wedge$ and addition and $\vee$ as multiplication, the category $\pMod{2}$ of $2$-modules is isomorphic to $\SLat_{\wedge\top}$, while the category $\Aff{2}$ of $2$-affine spaces is isomorphic to $\SLat_{\wedge}$ \cite[5.7]{Lu:FDistn}.  We now discuss three examples involving these categories, two of which are the special cases of \ref{exa:dualn_mod} and \ref{exa:r_aff} obtained by taking $R = 2$.

Firstly, by \ref{exa:dualn_mod} we obtain a stable finitary algebraic dual adjunction of the form $\Hom_{\wedge\top}(-,2) \dashv \Hom_{\wedge\top}(-,2):\SLat_{\wedge\top}^\op \rightarrow \SLat_{\wedge\top}$ under which each meet semilattice $A$ is sent to the meet semilattice $\Hom_{\wedge\top}(A,2)$ of all meet semilattice homomorphisms from $A$ to $2$.  But $\Hom_{\wedge\top}(A,2)$ is isomorphic to the meet semilattice of \textit{filters} in $A$, i.e., up-closed subsets $F \subseteq A$ that are closed under finite meets (and, in particular, contain $\top$).

Secondly, by \ref{exa:r_aff} we obtain a left-stable finitary algebraic dual adjunction of the form $(\SLat_\wedge,\SLat_{\wedge\top *}, \nabla \dashv \Delta)$ where  $\SLat_{\wedge\top *}$ is the category of \textit{pointed meet semilattices}, i.e., semilattices $B$ equipped with a chosen element $* \in B$, with semilattice homomorphisms preserving the chosen element.  However, this finitary algebraic dual adjunction is \textit{not} right-stable, because it is induced by the bifold algebra $(\Mat^\aff(2),\Mat^*(2),2)$, which is not right-commutant, as discussed in \cite[\S 8]{Lu:CvxAffCmt} and \cite[16.2.4]{Lu:BAlgCmt}.

Nevertheless, $(\Mat^\aff(2),2)$ is a saturated algebra, by \ref{para:cmtnt_bif_algs}, since it is the commutant of $(\Mat^*(2),2)$, by \ref{exa:r_aff}. Therefore, $(\Mat^\aff(2),2)$ induces a stable finitary algebraic dual adjunction that we now discuss, and which is therefore \textit{not} the preceding adjunction involving pointed semilattices.  Let $\SLat_{\wedge\top\bot}$ denote the category whose objects are meet semilattices with a bottom element (in addition to the top element that we require in any meet semilattice), with maps that preserve finite meets (including the top element) and also preserve the bottom element.  This category $\SLat_{\wedge\top\bot}$ may be identified with the category of normal $\U$-algebras for a Lawvere theory $\U$ given in \cite[8.1]{Lu:CvxAffCmt}.  As discussed in \cite[\S 8]{Lu:CvxAffCmt} and \cite[16.2.4]{Lu:BAlgCmt}, $2$ carries the structure of a commutant bifold algebra $(\Mat^\aff(2),\U,2)$, so by Theorem \ref{thm:bif_leftcmt_iff_dadj_left_stable} we obtain a stable finitary algebraic dual adjunction
\begin{equation}\label{eq:dualn_mslat_w_bot}\RAdjn{\SLat_{\wedge}^\op}{\Hom_{\wedge}(-,2)}{\Hom_{\wedge\top\bot}(-,2)}{}{}{\SLat_{\wedge\top\bot}}\;.\end{equation}
Given a meet semilattice $B$ with a bottom element, the dual $\Hom_{\wedge\top\bot}(B,2)$ of $B$ is the binary-meet semilattice whose elements are maps $B \rightarrow 2$ that preserve $\wedge$, $\top$, and $\bot$, and these may be described equivalently as \textit{proper filters} in $B$, i.e., filters in $B$ that do not contain $\bot$.  The Lawvere theory $\Mat^\aff(2)$ is commutative, by \ref{exa:r_aff}, since $2$ is a commutative rig, so by \ref{para:one-commutative-face}, if $B$ is \textit{any} meet semilattice with a bottom element, then $B$ underlies a bifold algebra $(\Mat^\aff(2),\U,B)$ and so induces a finitary algebraic dual adjunction
$$\RAdjn{\SLat_{\wedge}^\op}{\Hom_{\wedge}(-,B)}{\Hom_{\wedge\top\bot}(-,B)}{}{}{\SLat_{\wedge\top\bot}}\;.$$
\end{exasub}

\subsection{Some examples with infinite arities, over $\Set$}

In view of \ref{thm:eleutheric-implies-amenable} and \ref{exa:sys-ar}(2), a functor $U:\A \to \Set$ is $\J$-algebraic for the system of arities $\J = \Set$ consisting of all (small) sets iff $U$ is monadic. We now consider some examples of $\J$-algebraic dual adjunctions for $\J = \Set$, i.e.~dual adjunctions between pairs of categories monadic over $\Set$. The first two of these, Examples \ref{exa:sup-lat} and \ref{exa:cpct-disc}, are dual equivalences and so are stable, while the third, Example \ref{exa:los}, is not an equivalence, and it is stable if and only if there are no measurable cardinals in $\Set$. The bifold algebras that induce these three algebraic dual adjunctions were discussed in \cite[16.3]{Lu:BAlgCmt}.

\begin{exasub}[\textbf{The self-duality of complete sup-lattices}]\label{exa:sup-lat}
Let $\textnormal{Sup}$ denote the category whose objects are complete lattices and whose morphisms are maps that preserve suprema (i.e., joins of arbitrary subsets).  By \cite[I.1.1]{JoyalTierney} there is an adjoint equivalence
\begin{equation}\label{eq:self_duality_sup}\EquivAlt{\textnormal{Sup}^\op}{(-)^\bullet}{(-)^\bullet}{}{}{\textnormal{Sup}}\end{equation}
under which each complete lattice $L$ corresponds to its opposite $L^\bullet := L^\op$ and each morphism $f:L \to M$ in $\textnormal{Sup}$ corresponds to the map $f^\bullet = (f_*)^\op:M^\op \to L^\op$ where $f_*:M \to L$ is the right adjoint to $f$ (qua functor). $\textnormal{Sup}$ is symmetric monoidal closed with the internal hom obtained by equipping the hom-set $\textnormal{Sup}(L,M)$ $(L,M \in \ob\textnormal{Sup})$ with the pointwise order \cite[I.5]{JoyalTierney}. There is an isomorphism $L^\bullet \cong \textnormal{Sup}(L,2)$ in $\textnormal{Sup}$, natural in $L \in \textnormal{Sup}$ \cite[I.3]{JoyalTierney}.  It is well known that $\textnormal{Sup}$ is monadic over $\Set$ and, in particular, is isomorphic to the category of $\PP$-algebras for the powerset monad $\PP$ (see, e.g., \cite[3.2.2]{SealTholen:MonTopII-Monoidal-Structures}).  Hence $\textnormal{Sup}$ is $\J$-algebraic over $\Set$ for the system of arities $\J = \Set$ consisting of all (small) sets, and the adjoint equivalence \eqref{eq:self_duality_sup} is a stable $\J$-algebraic dual adjunction, since every object of $\textnormal{Sup}$ is reflexive.
\end{exasub}

\begin{exasub}[\textbf{The duality of compact and discrete topological groups}]\label{exa:cpct-disc}
For each locally compact Hausdorff topological abelian group $G$, let us write $\widehat{G}$ for the Pontryagin dual of $G$ (\ref{exa:binz-butmzmann-duality}).  When $G$ is compact (resp.~discrete), $\widehat{G}$ is discrete (resp.~compact), and the Pontryagin duality restricts to an adjoint equivalence
\begin{equation}\label{eq:duality_cpct_disc_ab_grp}\EquivAlt{\textnormal{CAb}^\op}{\widehat{(-)}}{\widehat{(-)}}{}{}{\Ab}\end{equation}
where $\textnormal{CAb}$ is the category of compact Hausdorff topological abelian groups and $\Ab$ is the category of (discrete) abelian groups; see \cite[Thm.~12]{Mor}.  By \cite[\S 5]{Lin:Eq}, $\textnormal{CAb}$ is $\J$-algebraic over $\Set$ for the system of arities $\J = \Set$, as is $\textnormal{Ab}$, so the adjoint equivalence \eqref{eq:duality_cpct_disc_ab_grp} is a stable $\J$-algebraic dual adjunction, since all the objects of $\textnormal{CAb}$ and of $\Ab$ are reflexive with respect to this dual adjunction.
\end{exasub}

\begin{exasub}[\textbf{Dualization of abelian groups and measurable cardinals}]\label{exa:los}
Consider the dual adjunction
\begin{equation}\label{eq:dualn_ab}\RAdjn{\Ab^\op}{\Hom_\ZZ(-,\ZZ)}{\Hom_\ZZ(-,\ZZ)}{}{}{\Ab}\end{equation}
determined by the object $\ZZ$ of the category of abelian groups, $\Ab$. Since $\Ab$ is monadic over $\Set$, this is an example of a $\J$-algebraic dual adjunction for the system of arities $\J = \Set$ consisting of all small sets, and we now address the question of whether it is stable.

Given any $J \in \ob\Set$, the free abelian group on $J$ is the direct sum $\oplus_{j \in J} \ZZ$. By a theorem of Ehrenfeucht and \L o\'s \cite{EhrLos}, the abelian group $\oplus_{j \in J} \ZZ$ is reflexive (with respect to the dual adjunction \ref{eq:dualn_ab}) if and only if the cardinality of $J$ is less than all measurable cardinals (i.e.~either there are no measurable cardinals or the cardinality of $J$ is less than the first measurable cardinal); also see \cite{Eda,Shelah}. Hence \eqref{eq:dualn_ab} is a stable $\J$-algebraic dual adjunction for $\J = \Set$ if and only if there are no measurable cardinals in $\Set$. The corresponding observation that the bifold algebra inducing \eqref{eq:dualn_ab} is a commutant bifold algebra (for $\J = \Set$) iff there are no measurable cardinals in $\Set$ was treated in \cite[16.3.2]{Lu:BAlgCmt}.
\end{exasub}

\bibliographystyle{amsplain}
\bibliography{bib}

\end{document}